\documentclass[12pt]{article}

\usepackage{amssymb}
\usepackage{amsmath}
\usepackage{amsthm}
\usepackage{graphicx}
\usepackage{amsfonts}
\usepackage{xspace}
\usepackage{bm} 
\usepackage{subcaption}
\usepackage[ruled]{algorithm2e}
\usepackage{hyperref}
\hypersetup{
  colorlinks=true,
  linkcolor=red,
  citecolor=blue,
  urlcolor=blue
}

\newcommand{\mat}[1]{\ensuremath{\mathsf{#1}}}
\newcommand{\fnc}[1]{\ensuremath{\mathcal{#1}}}

\renewcommand{\d}[0]{\mathsf{d}}
\newcommand{\cell}[0]{\mathsf{c}}
\newcommand{\face}[0]{\mathsf{f}}
\newcommand{\sdim}[0]{\mathsf{d}}

\DeclareMathOperator{\diag}{\mathsf{diag}}
\DeclareMathOperator{\erfi}{\mathsf{erfi}}

\newcommand{\etal}[0]{{\em et~al.\@}\xspace}
\newcommand{\eg}[0]{{e.g.\@}\xspace}
\newcommand{\ie}[0]{{i.e.\@}\xspace}

\newtheorem{definition}{Definition}

\newtheorem{remark}{Remark}
\newtheorem{theorem}{Theorem}

\begin{document}

\title{Constructing stable, high-order finite-difference operators on point clouds over complex geometries}

\author{Jason Hicken \and Ge Yan \and Sharanjeet Kaur}

\date{\today}

\maketitle

\begin{abstract}
High-order difference operators with the summation-by-parts (SBP) property can be used to build stable discretizations of hyperbolic conservation laws; however, most high-order SBP operators require a conforming, high-order mesh for the domain of interest.  To circumvent this requirement, we present an algorithm for building high-order, diagonal-norm, first-derivative SBP operators on point clouds over level-set geometries.  The algorithm is \emph{not} mesh-free, since it uses a Cartesian cut-cell mesh to define the sparsity pattern of the operators and to provide intermediate quadrature rules; however, the mesh is generated automatically and can be discarded once the SBP operators have been constructed.  Using this temporary mesh, we construct local, cell-based SBP difference operators that are assembled into global SBP operators.  We identify conditions for the existence of a positive-definite diagonal mass matrix, and we compute the diagonal norm by solving a sparse system of linear inequalities using an interior-point algorithm.  We also describe an artificial dissipation operator that complements the first-derivative operators when solving hyperbolic problems, although the dissipation is not required for stability.  The numerical results confirm the conditions under which a diagonal norm exists and study the distribution of the norm's entries.  In addition, the results verify the accuracy and stability of the point-cloud SBP operators using the linear advection equation.
\end{abstract}

\section{Introduction}\label{sec:intro}

Scientists and engineers frequently rely on the simulation of fluid flows around or inside complicated geometries; for example, simulations are used to predict the drag on aircraft, the energy output of fusion reactors, and the blood flow within arteries.  The impact that a flow simulation can have in a particular application is limited by the time it takes to set up and run the simulation.  This motivates the development of discretization methods that can rapidly pre-process complex geometries while also being efficient during run time.

One way to address run-time efficiency is to use a high-order discretization.   Indeed, the efficiency of high-order discretizations for solving hyperbolic conservation laws has been known for over half a century~\cite{Kreiss1972comparison,Swartz1974relative}.  However, the widespread adoption of high-order methods has been delayed, in part, by lack of robustness.  To address this, entropy-stable discretizations~\cite{Tadmor2003entropy,Ismail2009affordable, Fisher2013high, Carpenter2014entropy, Crean2018entropy,Chan2018discretely} based on summation-by-parts (SBP) operators~\cite{kreiss:1974, Gassner2013skew, Fernandez2014generalized, Hicken2016multi,Glaubitz2023fsbp} have emerged as a promising framework for guaranteeing stable and, therefore, robust high-order methods.

High-order SBP discretizations have the potential to accelerate the run time of flow simulations, but they do not necessarily reduce the significant pre-processing time that can limit the impact of simulations.  Like most high-order methods, SBP discretizations typically require a high-order, conforming grid, and generating such a mesh for a complicated geometry can require significant effort.  The challenges with high-order meshing have inspired the development of cut-cell SBP methods~\cite{Taylor2024energy}.

The present contribution aims to accelerate both the run and pre-processing time of high-order SBP discretizations.  Specifically, we present an algorithm to construct high-order, sparse SBP operators whose nodal degrees of freedom are defined on point clouds.  The combination of these two elements --- high-order SBP and point-cloud nodes --- yields an accurate and stable discretization with considerable flexibility.

\subsection{Related work}

Our work is closely related to the conservative mesh-free scheme of Chui \etal~\cite{Chiu2012conservative}, and the discontinuous Galerkin difference (DGD)~\cite{Hagstrom2019discontinuous} extensions presented in \cite{Kaur2023high} and \cite{Yan2022thesis}.  This section briefly reviews these works to help relate and distinguish them from the present contribution.

In \cite{Chiu2012conservative}, the authors present an algorithm to construct a high-order finite-difference method on point clouds.  While they do not use the term summation-by-parts, their sparse-matrix operators are closely related to multidimensional SBP operators~\cite{Hicken2016multi}.  The authors present two algorithms to solve for their sparse operators: a segregated method that solves for the diagonal mass matrix first and then solves for the coefficients of the stiffness matrices; and a coupled method that solves for the mass matrix and stiffness matrices simultaneously.  The coupled method has recently been extended to general function spaces by Glaubitz \etal~\cite{Glaubitz2024optimization}.  In addition, Kwan and Chan~\cite{Kwan2024robust} have proposed a simpler construction for a first-order variant of \cite{Chiu2012conservative}.

Unlike \cite{Chiu2012conservative}, our construction algorithm is not mesh free; however, because we assemble the global stiffness matrices from cell-based operators, our algorithm is likely much faster.  Like the segregated method in \cite{Chiu2012conservative}, we need to solve a global problem for the diagonal mass matrix.  This involves solving a sparse linear inequality, for which we can leverage mature and efficient algorithms for linear optimization.

The results presented here include up to fifth-order design accuracy, which also distinguishes our work from that in \cite{Chiu2012conservative}.  While the theory in \cite{Chiu2012conservative} considered operators of arbitrary polynomial degree, the numerical results were limited to second-order discretizations.

Finally, the present work also builds upon Yan's generalized DGD framework~\cite{Yan2022thesis} and Kaur and Hicken's cut-cell DGD method~\cite{Kaur2023high}.  Like the DGD methods, we use cells to define stencils that in turn determine the accuracy and sparsity of the global operators.  However, the DGD methods produce a non-diagonal (but sparse) mass matrix, which complicates the use of explicit time-marching schemes and the implementation of entropy stability~\cite{Yan2023entropy}.  By contrast, the algorithm presented here yields a diagonal mass matrix, so existing entropy-stable SBP theory can be applied directly.

\subsection{Paper outline}

The bulk of the paper is devoted to describing our construction algorithm.  Section~\ref{sec:background} covers preliminary material and notation that we rely on in subsequent sections.  In particular, Section~\ref{sec:background} introduces the concept of a degenerate SBP pair, in which the diagonal mass matrix is not necessarily positive definite.  The construction of degenerate SBP pairs is the subject of Section~\ref{sec:construct}, which begins with the construction of cell-based operators and then covers the assembly process for global operators.  Section~\ref{sec:optimize} explains how we convert a degenerate SBP pair into a diagonal-norm SBP operator by solving a linear inequality.  The operators produced by the construction algorithm are studied and verified in Section~\ref{sec:results}.  Section~\ref{sec:conclude} concludes the paper with a summary and a discussion of potential future work.

\section{Preliminaries}\label{sec:background}

\subsection{Domain, background mesh, and nodes}

We are interested in discretizing initial boundary-value problems that are posed on some bounded domain $\Omega \subset \mathbb{R}^{\sdim}$, where $\sdim$ is the spatial dimension.  The boundary of $\Omega$ is denoted $\Gamma$, which may be partitioned into non-overlapping subsets, $\Gamma_i \subset \Gamma$, where distinct boundary conditions can be applied.  An example domain --- a unit box with a circle of radius $\frac{1}{4}$ excised from its center --- is illustrated in Figure~\ref{fig:ex_domain}.

We use level-set functions to define some boundaries. For example, boundary $\Gamma_5$ in Figure~\ref{fig:ex_domain} can be defined as $\{ \bm{x} \in \Omega \;|\; \phi(\bm{x}) = 0 \}$, where the level-set function in this case is $\phi(\bm{x}) =  (x - \frac{1}{2})^2 + (y - \frac{1}{2})^2 - (\frac{1}{4})^2$.

Our construction algorithm relies on a background mesh that consists of $N_{C}$ non-overlapping cells, $\{ \Omega^{\cell} \}_{\cell=1}^{N_{C}}$, where $\Omega^{\cell} \subset \Omega$ denotes the domain of cell $\cell \in \{1,2,\ldots,N_{C}\}$.  Figure~\ref{fig:ex_mesh} illustrates a possible background mesh for the domain in Figure~\ref{fig:ex_domain}.

We also need notation for the faces that make up the cell boundaries.  Let $\{ \Gamma^{\face} \}_{\face=1}^{N_{F}}$ denote the $N_{F}$ unique faces in the mesh.  The subset of faces that coincide with the boundary of cell $\cell$ will be denoted $F^{\cell} \subset \{1,2,\ldots,N_{F}\}$.  In addition, we will use $F_I \subset \{1,2,\ldots,N_{F} \}$ for the subset of all interfaces between cells, and $F_B \subset \{1,2,\ldots,N_{F}\}$ to denote the subset of faces that coincide with (or approximate) the boundary $\Gamma$. In addition to Figure~\ref{fig:ex_mesh}, the various face types are illustrated in Figure~\ref{fig:stencil}.

Nodal SBP operators act on discrete solutions that are defined on a set of points or nodes.  The SBP operators that we construct for the global domain $\Omega$ use $N$ nodes, and these nodes are collected in the set $X \equiv \{ \bm{x}_i \}_{i=1}^{N}$.  Figure~\ref{fig:ex_points} depicts a set of nodes that could be used for an SBP operator on the domain in Figure~\ref{fig:ex_domain}.

Our algorithm also requires SBP operators constructed on the cells.  These cell-based operators will use a subset of the global nodes, which we will refer to as the cell's stencil.  The stencil index set for cell $\cell$ will be represented as
\begin{equation*}
  \mathsf{nodes}(\cell) \equiv \{ \nu_1, \nu_2, \ldots, \nu_{N^{\cell}} \} \subset \{1,2,\ldots, N\},
\end{equation*}
where $N^{\cell}$ is the number of nodes in the stencil.  The node coordinates for the stencil of cell $\cell$ will be denoted $X^{\cell} \equiv \{ \bm{x}_i \in X \;|\; i \in \mathsf{nodes}(\cell) \}$.

\begin{figure}[tbp]
  \begin{subfigure}[t]{0.32\textwidth}
    \centering
    \includegraphics[width=\textwidth]{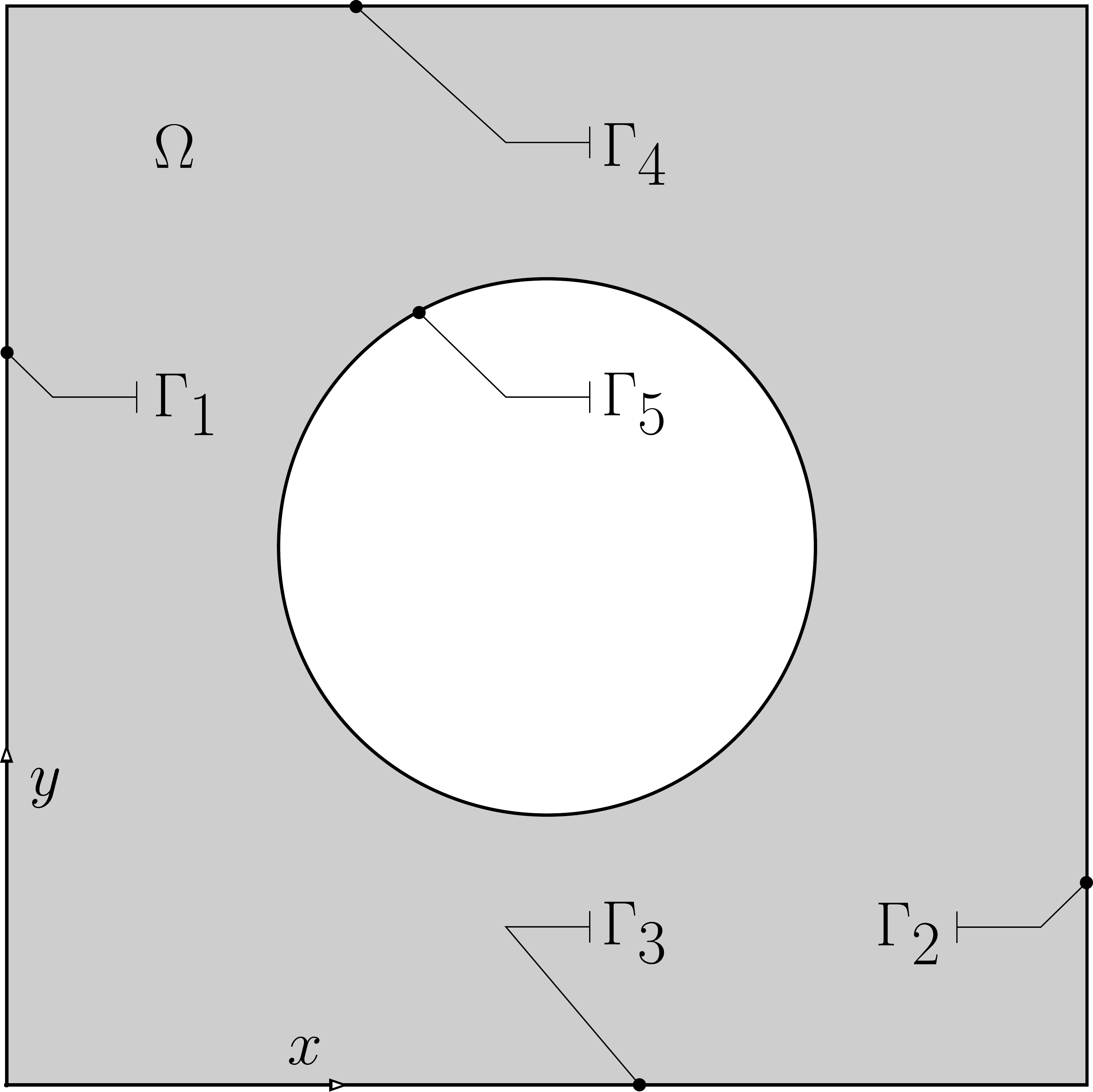}
    \caption{\small Example domain.} \label{fig:ex_domain}
  \end{subfigure}%
  \hfill%
  \begin{subfigure}[t]{0.32\textwidth}
    \centering
    \includegraphics[width=\textwidth]{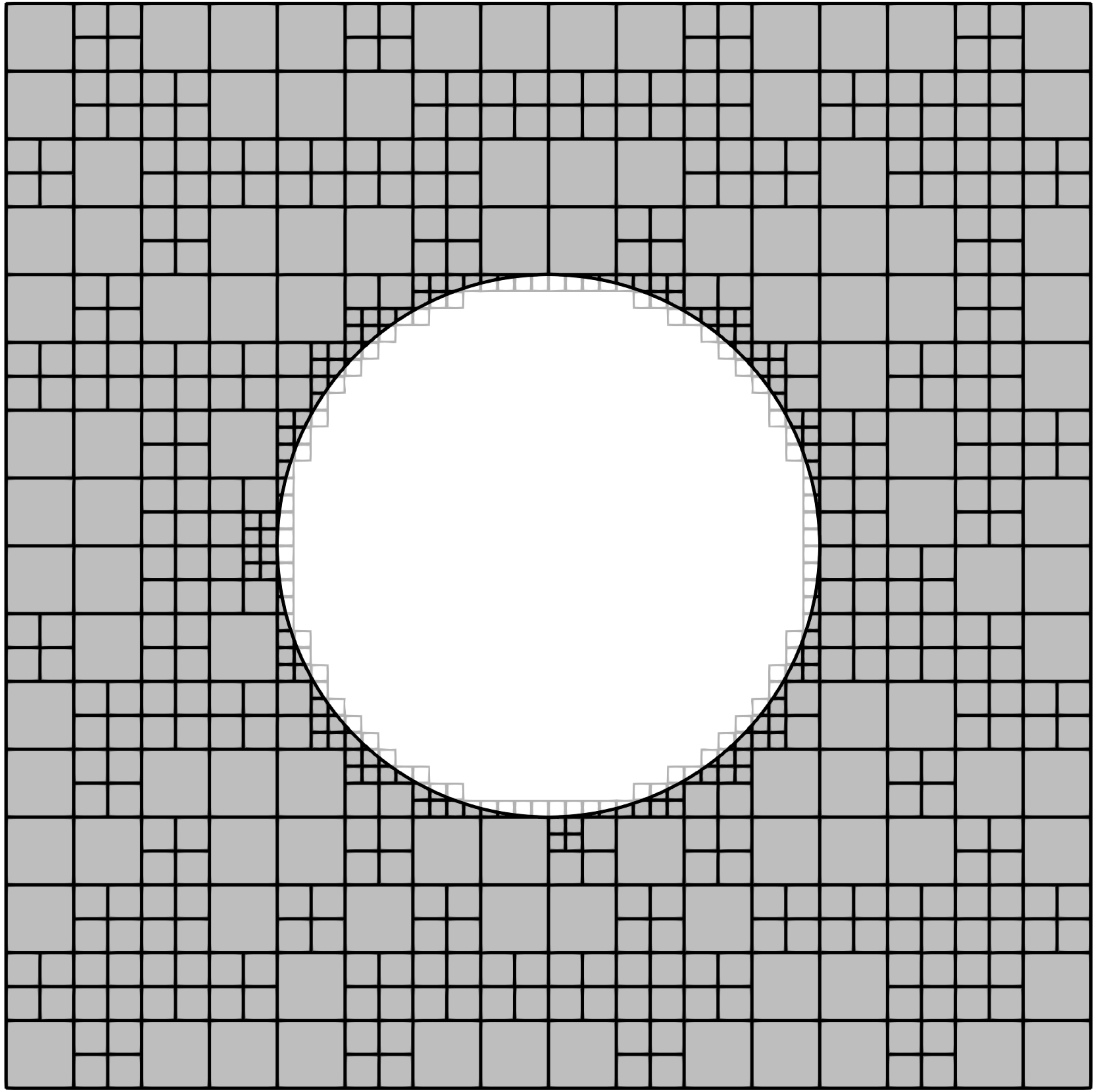}
    \caption{\small Example mesh.} \label{fig:ex_mesh}
  \end{subfigure}%
  \hfill%
  \begin{subfigure}[t]{0.32\textwidth}
    \centering
    \includegraphics[width=\textwidth]{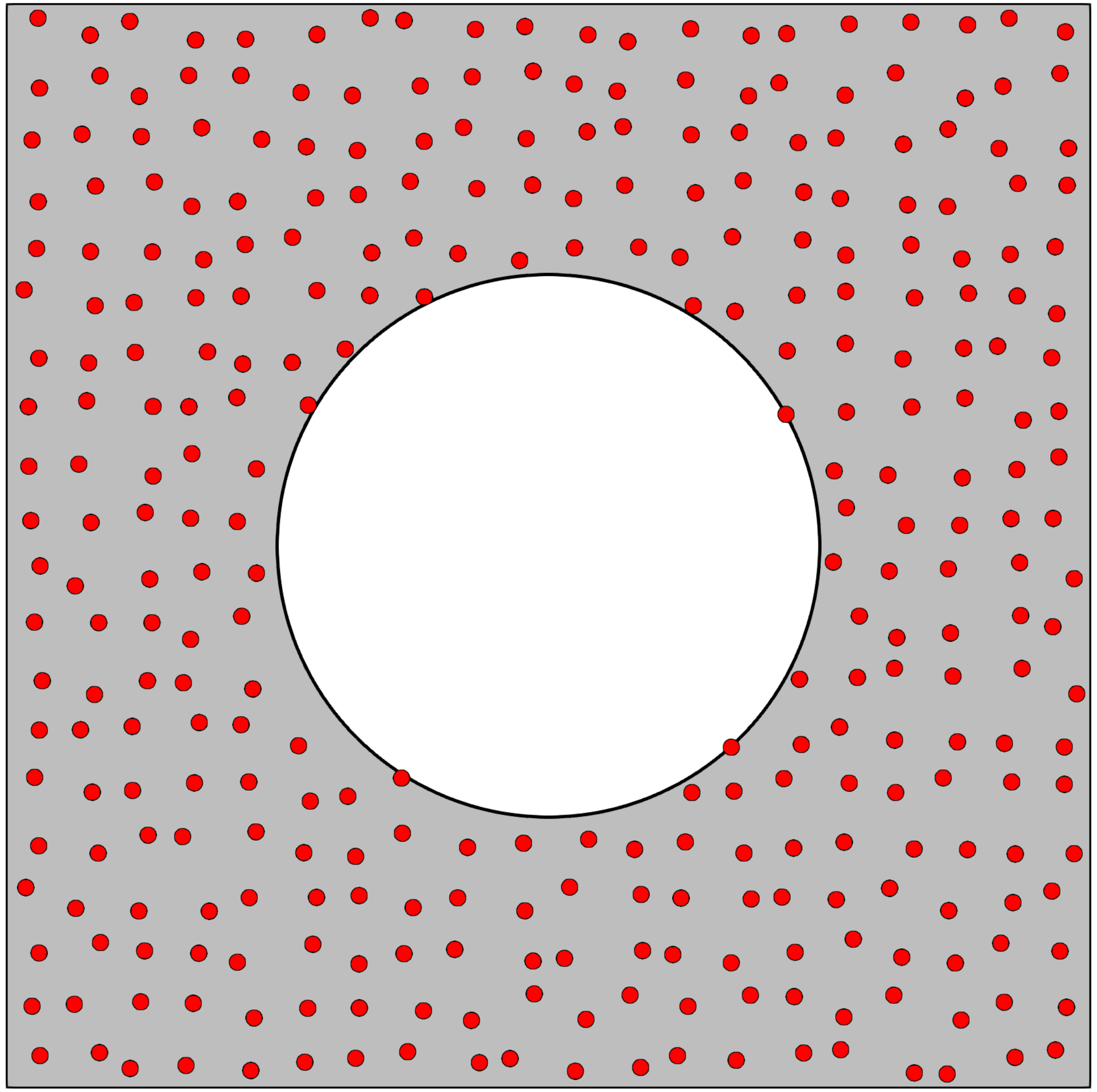}
    \caption{\small Example nodes.} \label{fig:ex_points}
  \end{subfigure}%
  \caption{\small Example of a domain (Fig.~\subref{fig:ex_domain}) and corresponding background mesh (Fig.~\subref{fig:ex_mesh}) and node set (Fig.~\subref{fig:ex_points}).} \label{fig:domain_mesh_nodes}
\end{figure}

\subsection{Polynomial spaces and accuracy}

The accuracy of various operators will be defined with respect to total degree polynomials. While the accuracy of SBP operators can be defined using other sets of functions~\cite{Glaubitz2023fsbp, Yan2022thesis}, we leave such an extension to future work.  We will use $\mathbb{P}^{p}(\Omega)$ to indicate the space of total degree $p$ polynomials over an arbitrary domain $\Omega$.

We need a basis for $\mathbb{P}^{p}(\Omega^{\cell})$ in order to define cell-based SBP operators.  We will use the notation $\{ \fnc{V}_{j} \}_{j=1}^{N_{p,d}}$ when it is necessary to refer to the basis explicitly, where $N_{p,\sdim} = {\scriptstyle \binom{p + \sdim}{\sdim}}$ is the dimension of the total degree $p$ basis in $\sdim$ dimensions.  In practice, we adopt the Proriol orthogonal basis~\cite{Proriol1957famille} for two-dimensional problems and their generalization for three-dimensional problems --- see, \eg, \cite[Chap.~10]{Hesthaven2008nodal} --- although other choices are certainly possible.

It is often necessary to evaluate the polynomial basis at some set of nodes.  The Vandermonde matrix is the matrix $\mat{V} \in \mathbb{R}^{N\times N_{p,\sdim}}$ whose rows correspond to the nodes in $X$ and whose columns correspond to the basis vectors:
\begin{equation*}
  \big[ \mat{V} \big]_{ij} = \fnc{V}_{j}(\bm{x}_{i}), \qquad \forall\, i \in \{1,2,\ldots,N\}, \quad \forall\, j \in \{1,2,\ldots, N_{p,\sdim}\}.
\end{equation*}
Similarly, $\mat{V}^{\cell}$ denotes the Vandermonde matrix of polynomials evaluated at the nodes $X^{\cell}$ of cell $\cell$.  In practice, we evaluate the Vandermonde matrix by first translating and scaling the nodes so that they lie within the reference domain where the basis is orthogonal, \ie, the reference triangle $\{ x,y > -1, x+y < 0\}$, in two dimensions, and the reference tetrahedron $\{ x,y,z>-1, x+y+z < -1\}$, in three dimensions.  Transforming the nodes to lie within the reference domain tends to improve the conditioning of the Vandermonde matrix.

We use $\mat{V}_x \in \mathbb{R}^{N \times N_{p,d}}$ to denote the matrix holding the basis' derivatives with respect to $x$, whose entries are given by 
\begin{equation*}
  \big[ \mat{V}_x \big]_{ij} = \frac{\partial \fnc{V}_{j}}{\partial x}(\bm{x}_i), \qquad \forall\, i \in \{1,2,\ldots,N\}, \quad \forall\, j \in \{1,2,\ldots, N_{p,d}\}.
\end{equation*}
Again, $\mat{V}_x^{\cell}$ is the analogous matrix for the cell's nodes $X^{\cell}$.

\subsection{First-derivative summation-by-parts operators}
 
While our goal is to construct diagonal-norm, first-derivative SBP operators, we need a broader definition of these operators than is typically used in the literature.  Specifically, we introduce the concept of a degenerate SBP pair, which relaxes the requirement for a positive-definite mass, or norm, matrix.

\begin{definition}[Degenerate SBP pair]\label{def:degen_sbp}
  The matrices $\mat{M} \in \mathbb{R}^{N\times N}$ and $\mat{Q}_x \in \mathbb{R}^{N\times N}$ form a degree $p$, degenerate summation-by-parts pair for the derivative $\partial/\partial x$ at the nodes $X = \{ \bm{x}_i \}_{i=1}^{N}$, if they satisfy the following conditions.
  \begin{enumerate}
    \item For degree $p$ polynomials, the operators satisfy the accuracy constraints
    \begin{equation}\label{eq:sbp_accuracy}
      \mat{Q}_x \mat{V} = \mat{M} \mat{V}_x.
    \end{equation}
    \item The matrix $\mat{M}$ is diagonal and $\mat{Q}_x$ has the factorization\footnote{All square matrices have such a factorization, but it is necessary to identify $\mat{E}_x$ for the third condition.}
    \begin{equation}\label{eq:sbp_factor}
      \mat{Q}_x = \mat{S}_x + {\textstyle \frac{1}{2}} \mat{E}_x,
    \end{equation}
    where $\mat{S}_x$ is skew symmetric, and $\mat{E}_x$ is symmetric.
    \item The symmetric matrix $\mat{E}_x$ satisfies 
    \begin{equation}\label{eq:sbp_boundary}
      \big[ \mat{V}^T \mat{E}_x \mat{V} \big]_{ij} = \int_{\Gamma} \fnc{V}_i \fnc{V}_j \, n_x\, \d \Gamma,
    \end{equation}
    where $n_x$ is the $x$ component of the outward unit normal vector on $\Gamma$.
  \end{enumerate}
\end{definition}

The above definition reframes first-derivative SBP operators in terms of $\mat{M}$ and $\mat{Q}_x$ rather than the difference operator $\mat{D}_x$; this allows us to consider operators with indefinite, or even singular, mass matrices $\mat{M}$. Otherwise, the degenerate SBP definition is essentially the same as the (diagonal-norm) multidimensional SBP definition introduced in~\cite{Hicken2016multi}.  Indeed, we can immediately relate diagonal-norm SBP operators to degenerate SBP pairs.

\begin{definition}[Diagonal-Norm SBP Operator]
  The matrix $\mat{D}_x \equiv \mat{M}^{-1} \mat{Q}_x \in \mathbb{R}^{N\times N}$ is a degree $p$, diagonal-norm summation-by-parts operator if $\mat{M}$ and $\mat{Q}_x$ form a degree $p$, degenerate SBP pair and $\mat{M}$ is positive definite.
\end{definition}

The similarity between degenerate and diagonal-norm SBP operators allows us to adapt several theoretical results from the latter to the former.

\begin{theorem}[SBP quadrature] Let $\mat{M}$ and $\mat{Q}_x$ form a degree $p$, degenerate SBP pair on the nodes $X$.  Then the diagonal entries in $\mat{M}$ are quadrature weights for a degree $2p-1$ exact quadrature rule at the nodes $X$.\label{thm:quad}
\end{theorem}

\begin{theorem}[SBP existence] Consider a domain $\Omega$ and a node set $X = \{\bm{x}_{i} \}_{i=1}^{N}$.  Let $\mat{V} \in \mathbb{R}^{N \times N_{p,\sdim}}$ denote the Vandermonde matrix for a total degree $p$ polynomial basis evaluated at the nodes $X$ and assume that $N \geq N_{p,d}$.  If the Vandermonde matrix is full column rank, then there exists a degenerate SBP pair if and only if there is a quadrature rule for numerical integration over $\Omega$ that is defined on the nodes $X$ and is degree $2p-1$ exact.\label{thm:exist}
\end{theorem}

The proofs of these results are omitted, because they are identical to the respective proofs of Theorems~3.2 and 3.3 from~\cite{Hicken2016multi}, except that diagonal-norm operators are replaced with degenerate SBP pairs and the requirement for positive quadrature weights is dropped.

\section{Construction algorithm}\label{sec:construct}

This section describes our algorithm to construct degenerate SBP pairs over a given domain and node set.  The high-level steps are summarized below.  The steps are similar to the algorithm presented by Yan~\cite{Yan2022thesis}, but here the resulting operators have diagonal $\mat{M}$ whereas the discontinuous Galerkin difference mass matrices in \cite{Yan2022thesis} are not diagonal.
\begin{enumerate}
\item Provide the node distribution, $X$, for the given domain $\Omega$ (\S~\ref{sec:nodes}).
\item Construct the mesh and define a stencil for each cell (\S~\ref{sec:quad_mesh}).
\item Build degenerate SBP pairs for each cell based on the stencils (\S~\ref{sec:cell_sbp}).
\item Assemble the cell-based matrices into global SBP operators (\S~\ref{sec:assemble}).
\end{enumerate}

\begin{remark} 
Steps 1 and 2 are closely related.  For example, the DGD discretizations in \cite{Yan2023entropy} and \cite{Kaur2023high} define the node locations to be the centroids of the cells, so the mesh must be constructed first.
\end{remark}

\subsection{Node distribution}\label{sec:nodes}

For the purposes of this work, we assume that the node set $X$ is provided by the user.  This does not mean the node distribution is inconsequential.  On the contrary, the node coordinates have a substantial impact on the SBP operator's accuracy and maximum time-step restriction, much like a mesh influences these factors in a conventional finite-element discretization.  The node coordinates also influence whether the diagonal mass matrix is positive definite, at least on coarse grids (see Section~\ref{sec:optimize}).  Thus, the node distribution is an important enough topic that it warrants its own investigation beyond this paper.  Until such an investigation is conducted, we use the ad hoc methods described in Section~\ref{sec:node_distribution} to generate node distributions.

\subsection{Background mesh}\label{sec:quad_mesh}

The construction algorithm uses a background mesh in a divide-and-conquer strategy: we build degenerate SBP operators on each cell of the mesh and then assemble these cell-based operators into a global SBP operator.  Once the global SBP operators are constructed, the mesh becomes extraneous and can be deleted to save memory.

Much like the node distribution, there are myriad ways to construct the background mesh.  We adopt an unstructured Cartesian cut-cell mesh with isotropic refinement, because it helps simplify and automate our algorithm on complex domains.  We acknowledge that isotropic adaptation is not well suited to anisotropic features~\cite{Aftosmis1995adaptation}, so this issue will need to be addressed in future work.

Our mesh implementation relies on the \texttt{RegionTrees.jl} Julia package~\cite{RegionTrees}.  To generate the Cartesian mesh, we start with a single quadrilateral in two dimensions (hexahedron in three dimensions).  The domain of the initial cell is given by 
\begin{equation*}
  \Omega_{\mathsf{init}} \equiv
\{ \bm{x} \in \mathbb{R}^{\sdim} \;|\; x_i \in [l_i, u_i],\; i=1,2,\ldots,\sdim \},
\end{equation*}
where the interval $[l_i,u_i]$ defines the extent of the domain in the coordinate direction $i$.  The mesh is then recursively refined by splitting cells into four (eight) equal-sized children until each child cell contains at most one node from the set $X$.  Subsequently, we recursively refine elements that are intersected by a level-set boundary --- if such a boundary is present --- until the cut cells are below some prescribed minimum cell size, $\Delta x_{i,\min}, i=1,2,\ldots,\sdim$.

\begin{remark}
  We limit each cell to one node because this condition is easy to enforce, but it is not an inherent requirement of the method and one can have multiple nodes per cell.
\end{remark}

\subsubsection{Quadrature rules}\label{sec:quad}

We assume that each cell of the mesh is equipped with a quadrature rule that is exact for polynomials in $\mathbb{P}^{2p-1}(\Omega^{\cell})$, where $p \geq 1$ is the polynomial exactness of the SBP operator.  Furthermore, we need positive quadrature rules for each face $\Gamma^{\face}$, $\face \in \{1,2,\ldots,N_{F}\}$, in order to construct the symmetric boundary operators, \eg, $\mat{E}_x^{\cell}$.  For a degree $p$ SBP operator, these surface quadrature rules must exactly integrate polynomials from $\mathbb{P}^{2p}(\Gamma^{\face})$, which is one degree higher than the volume rules. 

We use tensor-product Gauss-Legendre rules for numerical integration over line segments, quadrilaterals, and hexahedra; these geometries represent the vast majority of cells and faces in the mesh.  We use Saye's algorithm~\cite{Saye2022high} for the cells and faces that intersect the level-set $\phi(\bm{x}) = 0$.  The algorithm in \cite{Saye2022high} produces high-order, positive quadrature rules for hypercubes that are intersected by polynomial level-set functions.  When $\phi(\bm{x})$ is not a polynomial, the algorithm projects the level-set onto a high-order Bernstein polynomial approximation.

\subsubsection{Stencil construction}\label{sec:stencil}

The stencil for cell $\cell$ is determined by finding the $N^{\cell}$ nodes in $X$ that are ``closest'' to the centroid of the cell; the centroid of the \emph{uncut} quadrilateral or hexahedron is used for cells that intersect the level-set $\phi(\bm{x}) = 0$.  Refer to Figure~\ref{fig:stencil} for an example.  The notion of ``closeness'' depends on the metric used to measure distance.  In this work we use the Euclidean metric to measure distance. If the nodes are distributed anisotropically, \eg, to resolve boundary layers, a non-Euclidean metric may be necessary to ensure the stencil is not overly biased in a certain direction.

\begin{figure}[tbp]
  \begin{center}
    \includegraphics[width=0.65\textwidth]{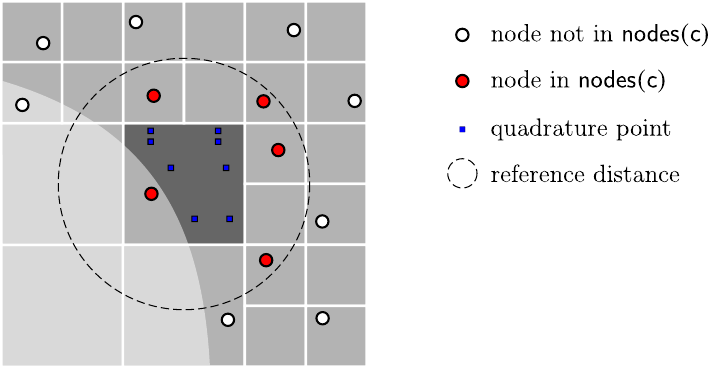}
    \caption[]{\small Illustration of a cell (dark gray) with its stencil (red circles).  The light gray squares and partial squares in the lower left indicate the immersed region where $\phi(\bm{x}) < 0$.  The dashed circle provides a reference length to verify that only the closest $N^{\cell} = 5$ nodes are included in the stencil.  The blue squares represent quadrature nodes for integration over the (cut) cell.}\label{fig:stencil}
  \end{center}
  \end{figure}

The number of nodes in the stencil strongly influences the conditioning of the Vandermonde matrix, $\mat{V}^{\cell}$, which then influences the properties of the degenerate SBP pair on cell $\cell$.  Therefore, similar to Kaur~\etal~\cite{Kaur2023high}, we add nodes to the stencil until the condition number of the degree $2p-1$ Vandermonde matrix --- denoted $\mathsf{cond}(\mat{V}_{2p-1}^{\cell})$ --- is below some threshold $\mathsf{tol}$ or a maximum number of iterations $n_{\max}$ is reached; see Algorithm~\ref{alg:stencil}.  To generate the results for Section~\ref{sec:results}, we set the maximum number of inner iterations to $n_{\max} = 4p-1$, and the condition-number threshold to $\mathsf{tol} = 5\times 10^{2p-1}$.  In our numerical experiments, the innermost loop of Algorithm~\ref{alg:stencil} always terminated before reaching $n_{\max}$; nevertheless, adding nodes is not guaranteed to reduce the condition number of the Vandermonde matrix, and some node sets and geometries may require larger $n_{\max}$ or even node repositioning.

A couple of points about Algorithm~\ref{alg:stencil} are worth highlighting. First, the polynomial degree used in the algorithm is $2p-1$ rather than $p$, because the stencil must be sufficiently large to construct a $2p-1$ exact quadrature rule; this will be discussed more in Section~\ref{sec:norm}.  Second, since the nested loop in Algorithm~\ref{alg:stencil} starts at $n=1$, the default stencil size is $N^{\cell} = N_{2p-1,\sdim} + 1$, which is one larger than the dimension of the total degree $2p-1$ polynomial basis.  This default stencil size is motivated by our approach to ensure a positive norm.  In short, the larger-than-necessary stencil introduces additional degrees of freedom in the cell-based quadrature that can be used to enforce positivity; see Section~\ref{sec:optimize}.

\begin{algorithm}[tbp]\DontPrintSemicolon
  \SetKwInOut{Input}{Input}\SetKwInOut{Output}{Output}
  \Input{nodes $X$, mesh cells $\{\Omega^{\cell}\}_{\cell=1}^{N_{C}}$, poly. degree $p$, $n_{\max}$, $\mathsf{tol}$.}
  \Output{Stencils, $\mathsf{nodes}(\cell)$, for all $\cell = 1,2,\ldots,N_{C}$}
  \For{each cell $\cell$ in the mesh}
  {
    \For{$n = 1,2,\ldots,n_{\max}$} 
    {
      set stencil size $N^{\cell} \gets N_{2p-1,\sdim} + n$\;
      $\mathsf{nodes}(\cell) \gets $ indices of the $N^{\cell}$ closest nodes to centroid of cell $\cell$\;
      \If{$\mathsf{cond}(\mat{V}_{2p-1}^{\cell}) < \mathsf{tol}$}
      {
        break 
      }
    }
  }
  \caption{{\small Cell stencil generation algorithm.}}\label{alg:stencil}
\end{algorithm}

\subsection{Cell-based SBP operators}\label{sec:cell_sbp}

Once the background mesh and stencils are available, we construct degenerate SBP pairs for each cell.  We follow the same sequence often used to construct multidimensional diagonal-norm SBP operators~\cite{Hicken2016multi,Fernandez2017simultaneous}:
\begin{enumerate}
\item Generate a quadrature rule at the nodes $X^{\cell}$ that exactly integrates degree $2p-1$ polynomials over $\Omega^{\cell}$.
\item Construct the symmetric boundary operators, \eg, $\mat{E}_x^{\cell}$, using quadrature rules over the faces of the cell, $\face \in F^{\cell}$.
\item Construct the skew-symmetric operators, \eg, $\mat{S}_x^{\cell}$.
\end{enumerate}
Sections~\ref{sec:norm}, \ref{sec:sym_part}, and \ref{sec:skew_part}, respectively, explain these three steps for the pair $(\mat{M}^{\cell}, \mat{Q}_x^{\cell})$; the steps are analogous for the other spatial directions.

\subsubsection{Cell-based norm/quadrature}\label{sec:norm}

The non-trivial entries in the diagonal mass matrix $\mat{M}^{\cell}$ are the weights of a $2p-1$ exact quadrature rule over $\Omega^{\cell}$ based on the cell's nodes $X^{\cell}$; see Theorems~\ref{thm:quad} and \ref{thm:exist}.  Thus, the diagonal entries in the mass matrix satisfy the following system of $N_{2p-1,\sdim}$ equations in $N^{\cell}$ unknowns:
\begin{equation}\label{eq:norm_conditions_scalar}
  \sum_{i = 1}^{N^{\cell}} \fnc{V}_{j}(\bm{x}_{i}) \big[ \mat{M}^{\cell} \big]_{ii} = \int_{\Omega^{\cell}} \fnc{V}_j(\bm{x}) \, d\Omega,\qquad \forall\, \fnc{V}_{j} \in \mathbb{P}^{2p-1}(\Omega^{\cell}),
\end{equation}
where $\bm{x}_i \in X^{\cell}$.  The integrals on the right-hand side are evaluated using the quadrature rules described in Section~\ref{sec:quad}.

The linear system \eqref{eq:norm_conditions_scalar} can be expressed in matrix form as
\begin{equation}\label{eq:norm_conditions}
  [\mat{V}_{2p-1}^{\cell}]^T \bm{m}^{\cell} = \bm{b}^{\cell},
\end{equation}
where the unknowns $\bm{m}^{\cell} \in \mathbb{R}^{N^{\cell}}$ are the diagonal entries in $\mat{M}^{\cell}$.  The subscript $2p-1$ on the Vandermonde matrix indicates that the basis for the quadrature conditions uses degree $2p-1$ polynomials and not merely degree $p$ polynomials.  The entries in $\bm{b}^{\cell}$ are the integrals on the right-hand side of \eqref{eq:norm_conditions_scalar}, ordered consistently with the columns in the Vandermonde matrix.

The system \eqref{eq:norm_conditions} has fewer equations than unknowns, because the 
stencil size satisfies $N^{\cell} > N_{2p-1,\sdim}$ (see Section~\ref{sec:stencil}).  Furthermore, the stencil was generated to ensure the Vandermonde matrix $\mat{V}_{2p-1}^{\cell}$ has a bounded condition number.  Thus, the general solution to \eqref{eq:norm_conditions} can be expressed as 
\begin{equation}\label{eq:cell_norm_sol}
  \bm{m}^{\cell} = \bm{m}_{\min}^{\cell} + \mat{Z}^{\cell} \bm{y}^{\cell},
\end{equation}
where $\bm{m}_{\min}^{\cell}$ is the minimum-norm solution, the columns of $\mat{Z}^{\cell}$ are a basis for the null-space of the transposed Vandermonde matrix, \ie, $[\mat{V}_{2p-1}^{\cell}]^T \mat{Z}^{\cell} = \mat{0}$, and $\bm{y}^{\cell} \in \mathbb{R}^{N^{\cell} - N_{2p-1,\sdim}}$ is an arbitrary vector.

For the remainder of Section~\ref{sec:construct}, we will adopt the minimum-norm solution for the cell's quadrature weights, $\bm{m}^{\cell} = \bm{m}_{\min}^{\cell}$.  However, we will revisit this choice in Section~\ref{sec:optimize}, where the additional degrees of freedom provided by the null-space will be critical to ensuring a positive-definite mass matrix.

\subsubsection{Cell-based symmetric part of $\mat{Q}_x^{\cell}$}\label{sec:sym_part}

The symmetric part of $\mat{Q}_x^{\cell}$ is constructed as described in \cite{Fernandez2017simultaneous}: 
\begin{equation}\label{eq:sym_part}
  \mat{E}_x^{\cell} = \sum_{\face \in F^{\cell}} \big[\mat{R}^{\face,\cell} \big]^T \; \mat{B}^{\face} \; \mat{N}_{x}^{\face,\cell} \; \mat{R}^{\face,\cell},
\end{equation}
where $\mat{R}^{\face,\cell}$ is an interpolation operator (defined below), the diagonal matrix $\mat{B}^{\face}$ holds the quadrature weights for face $\face$ along its diagonal, and the diagonal matrix $\mat{N}_{x}^{\face,\cell}$ holds the $x$ components of the outward-facing unit normal, $n_x$, at the quadrature nodes of face $\face$.

The matrix $\mat{R}^{\face,\cell}$ is constructed such that it can exactly interpolate degree $p$ polynomials from the cell's nodes $X^{\cell}$ to the quadrature points of face $\face \in F^{\cell}$.  Specifically, we define this interpolation operator to be the following minimum-norm solution to the interpolation conditions~\cite{Fernandez2017simultaneous}:
\begin{equation*}
  \mat{R}^{\face,\cell} = \mat{V}^{\face}\big[ \mat{V}^{\cell} \big]^{\dagger},
\end{equation*}
where $\mat{V}^{\face}$ is the matrix of polynomials evaluated at the quadrature nodes of face $\face$, and $\big[ \mat{V}^{\cell} \big]^{\dagger}$ denotes the Moore-Penrose pseudo inverse of the cell Vandermonde matrix.

Based on the definitions of $\mat{R}^{\face,\cell}$, $\mat{B}^{\face}$ and $\mat{N}_{x}^{\face,\cell}$, it follows from \cite[Theorem~1]{Fernandez2017simultaneous} that $\mat{E}_x^{\cell}$ satisfies Equation~\eqref{eq:sbp_boundary} in Definition~\ref{def:degen_sbp}.

\subsubsection{Cell-based skew-symmetric part of $\mat{Q}_x^{\cell}$}\label{sec:skew_part}

At this point, the cell mass matrices $\mat{M}^{\cell}$ and the boundary matrices $\mat{E}_x^{\cell}$ have been constructed to satisfy conditions two and three in Definition~\ref{def:degen_sbp}.  This leaves the skew-symmetric matrices $\mat{S}_x^{\cell}$ to satisfy the first condition, \ie, the accuracy constraints.  Substituting the factorization $\mat{Q}_x^{\cell} = \mat{S}_x^{\cell} + \frac{1}{2} \mat{E}_x^{\cell}$ into Equation~\eqref{eq:sbp_accuracy} and rearranging, we find that the accuracy constraints are linear in $\mat{S}_x^{\cell}$:
\begin{equation}\label{eq:skew_system}
  \mat{S}_x^{\cell} \mat{V}^{\cell} = \mat{M}^{\cell} \mat{V}_x^{\cell} - {\textstyle \frac{1}{2}} \mat{E}_x^{\cell} \mat{V}^{\cell}.
\end{equation}

Solving the linear system \eqref{eq:skew_system} directly to find $\mat{S}_x^{\cell}$ can be costly, especially in three dimensions.  Table~\ref{tab:prob_size} lists the size of the linear system for polynomial degrees one through four as a function of problem dimension $\sdim$.  The system size is for a cell with the smallest possible stencil based on Algorithm~\ref{alg:stencil}.  The number of entries in the system matrix (last column) grows rapidly with $p$; indeed, one can show it grows as $p^{4\sdim}$.  Granted, sparsity in the linear system can be exploited, but avoiding the construction and solution of large sparse systems \emph{for each cell} would be preferable.

\begin{remark}
In~\cite{Hicken2016multi} and \cite{Fernandez2017simultaneous}, Equation~\eqref{eq:skew_system} was formed explicitly and solved in a minimum-norm sense.  The cost was justified in \cite{Hicken2016multi} and \cite{Fernandez2017simultaneous} because the authors constructed SBP operators on reference elements, which could be precomputed and stored.  Here we need to solve~\eqref{eq:skew_system} on each cell.
\end{remark}

\begin{remark}
Examining Table~\ref{tab:prob_size} suggests that some of the linear systems are over determined; however, the so-called compatibility conditions --- see, \eg, \cite[Lemma 3.1]{Hicken2016multi} --- are embedded in these accuracy conditions.  For example, for $p=1$, there are 6 redundant equations in two dimensions, and 10 redundant equations in three dimensions.
\end{remark}

\begin{table}
\begin{center}
  \caption{\small Number of unknowns and equations in the linear system \eqref{eq:skew_system} as a function of spatial dimension $\sdim$ and operator degree $p$.  The column titled \textbf{unknowns} refers to the number of unique entries in $\mat{S}_x^{\cell}$, the column \textbf{equations} is the number of equations, and \textbf{matrix size} indicates the number of entries in system matrix.}\label{tab:prob_size}
  \begin{tabular}{rrrrrr}
    & & & \multicolumn{3}{c}{\textbf{Dimensions of Eq.~\eqref{eq:skew_system}}} \\\cline{4-6}
    \rule{0ex}{3ex}$\quad \sdim$ & $\quad p$ & $\quad N^{\cell}$ & \textbf{unknowns} & \textbf{equations} & \textbf{matrix size} \\\hline 
\rule{0ex}{3ex}2 & 1 & 4 & 6 & 12 & 72 \\
                 & 2 & 11 & 55 & 66 & 3630 \\
                 & 3 & 22 & 231 & 220 & 50820 \\
                 & 4 & 37 & 666 & 555 & 369630 \\\hline
\rule{0ex}{3ex}3 & 1 & 5 & 10 & 20 & 200 \\
                 & 2 & 21 & 210 & 210 & 44100 \\
                 & 3 & 57 & 1596 & 1140 & 1819440 \\
                 & 4 & 121 & 7260 & 4235 & 30746100 \\\hline
  \end{tabular}
\end{center}
\end{table}

The theorem below shows that we can find an $\mat{S}_x^{\cell}$ that satisfies Equation~\eqref{eq:skew_system} by computing a $QR$ factorization of the Vandermonde matrix.  The Vandermonde matrix has dimensions $N^{\cell} \times N_{p,\sdim}$, which grows at a more reasonable rate\footnote{The number of entries in the Vandermonde matrix is equal to the number of equations in \eqref{eq:skew_system}, which is the second-last column in Table~\ref{tab:prob_size}.} proportional to $p^{2\sdim}$.  For generality, the theorem is stated in terms of a degenerate SBP pair defined on a generic domain $\Omega$ and node set $X$; however, in practice, we apply this theorem to the cell's domain $\Omega^{\cell}$ and node set $X^{\cell}$.  Note that, to avoid confusion with previously introduced matrices, the theorem uses $\mat{U} \in \mathbb{R}^{N^{\cell} \times N_{p,\sdim}}$ and $\mat{L}^{T} \in \mathbb{R}^{N_{p,\sdim} \times N_{p,\sdim}}$ to denote the orthogonal and upper triangular matrices, respectively, in the $QR$ factorization.

\begin{theorem}\label{thm:skew_QR}
  Let $\mat{M} \in \mathbb{R}^{N\times N}$ be a diagonal matrix whose diagonal entries correspond to a $2p-1$ exact quadrature rule over the nodes $X = \{ \bm{x}_{i} \}_{i=1}^{N}$ for the domain $\Omega$.  Furthermore, let $\mat{E}_x \in \mathbb{R}^{N\times N}$ be a symmetric matrix that satisfies Equation~\eqref{eq:sbp_boundary}.  Define 
  \begin{equation*}
    \mat{G}_x \equiv \mat{M} \mat{V}_x - {\textstyle \frac{1}{2}} \mat{E}_x \mat{V},
  \end{equation*}
  where $\mat{V} \in \mathbb{R}^{N \times N_{p,\sdim}}$ and $\mat{V}_x \in \mathbb{R}^{N \times N_{p,\sdim}}$ denote the total degree $p$ basis and its derivatives evaluated at the nodes $X$, respectively.  If the Vandermonde matrix has full column rank and has the $QR$ decomposition $\mat{V} = \mat{U} \mat{L}^T$, then
  \begin{equation}\label{eq:skew_part}
    \mat{S}_x \equiv \mat{G}_x \mat{L}^{-T} \mat{U}^T - \mat{U} \mat{L}^{-1} \mat{G}_x^{T} + \mat{U} \mat{L}^{-1} \mat{G}_x^T \mat{U} \mat{U}^T,
  \end{equation}
  is a skew-symmetric matrix that satisfies Equation~\eqref{eq:sbp_accuracy}, \ie $\big( \mat{S}_x + \frac{1}{2} \mat{E}_x \big) \mat{V} = \mat{M} \mat{V}_x$.
\end{theorem}

\begin{proof}
  Using the $QR$ factorization $\mat{V} = \mat{U} \mat{L}^T$ and \eqref{eq:skew_part}, we find 
  \begin{align*}
    \big( \mat{S}_x + {\textstyle \frac{1}{2}} \mat{E}_x \big) \mat{V}
    &= \mat{G}_x \mat{L}^{-T} \mat{U}^T \mat{U} \mat{L}^T - \mat{U} \mat{L}^{-1} \mat{G}_x^{T} \mat{U} \mat{L}^T + \mat{U} \mat{L}^{-1} \mat{G}_x^T \mat{U} \mat{U}^T \mat{U} \mat{L}^T  + {\textstyle \frac{1}{2}} \mat{E}_x \mat{V} \\
    &= \mat{G}_x - \mat{U} \mat{L}^{-1} \mat{G}_x^{T} \mat{U} \mat{L}^T 
    + \mat{U} \mat{L}^{-1} \mat{G}_x^T \mat{U} \mat{L}^T
    + {\textstyle \frac{1}{2}} \mat{E}_x \mat{V} \\
    &= \mat{G}_x + {\textstyle \frac{1}{2}} \mat{E}_x \mat{V} = \mat{M} \mat{V}_x,
  \end{align*}
  where we used the definition of $\mat{G}_x$ on the last line.  Thus, $\mat{S}_x$ defined by Equation~\eqref{eq:skew_part} satisfies the degenerate SBP accuracy constraints.

  To show that $\mat{S}_x$ is skew symmetric, we need to show that the last term in Equation~\eqref{eq:skew_part} is skew symmetric, since the first two terms are jointly skew symmetric.  To this end, we rearrange the last term using the definition of $\mat{G}_x$, the symmetry of $\mat{E}_x$, and the $QR$ factorization $\mat{V} = \mat{U} \mat{L}^{T}$:
  \begin{align*}
    \mat{U} \mat{L}^{-1} \mat{G}_x^T \mat{U} \mat{U}^T
    &= \mat{U} \mat{L}^{-1} \big( \mat{G}_x^T \mat{U} \mat{L}^T \big) \mat{L}^{-T} \mat{U}^T \\
    &= \mat{U} \mat{L}^{-1} \big( \mat{V}_x^T \mat{M} \mat{V} - \textstyle{\frac{1}{2}} \mat{V}^T \mat{E}_x \mat{V} \big) \mat{L}^{-T} \mat{U}^T.
  \end{align*}
  If the term in parentheses is skew symmetric, then the entire right-hand side above will be skew symmetric and, thus, $\mat{S}_x$ will be skew symmetric.
  
  Based on the accuracy of the quadrature rule in $\mat{M}$, and the fact that $\mat{E}_x$ satisfies Equation~\eqref{eq:sbp_boundary}, the $ij$th entry in $\mat{V}_x^T \mat{M} \mat{V} - \frac{1}{2}\mat{V}^T \mat{E}_x \mat{V}$ is
  \begin{align*}
    \big[ \mat{V}_x \mat{M} \mat{V} - \textstyle{\frac{1}{2}} \mat{V}^T \mat{E}_x \mat{V} \big]_{ij} &= \int_{\Omega} \frac{\partial \fnc{V}_i}{\partial x} \fnc{V}_j \, d\Omega - \frac{1}{2} \int_{\Gamma} \fnc{V}_i \fnc{V}_j \, n_x \, d\Gamma \\
    &= \frac{1}{2}  \int_{\Omega} \frac{\partial \fnc{V}_i}{\partial x} \fnc{V}_j \, d\Omega - \frac{1}{2} \int_{\Omega} \frac{\partial \fnc{V}_j}{\partial x} \fnc{V}_i \, d\Omega,
  \end{align*}
  where we used integration by parts to obtain the final expression, which is clearly skew symmetric. Thus, we conclude that $\mat{S}_x$ is a skew-symmetric matrix.
\end{proof}

\subsection{Assembly of global SBP operators}\label{sec:assemble}

Once the cell-based (degenerate) SBP pairs are constructed, they are assembled into global operators using an approach similar to \cite[Theorem 4.1]{Hicken2016multi}.  However, unlike the symmetric operators in~\cite{Hicken2016multi}, the $\mat{E}_x^{\cell}$ in the present work are dense, so the assembly process must be modified to avoid symmetric contributions to $\mat{Q}_x$ from interior cells.

Before proceeding, we introduce some additional notation to simplify the presentation of the assembled SBP operators.  First, let $\mat{P}^{\cell} \in \mathbb{R}^{N^{\cell} \times N}$ be the binary matrix that selects the global degrees of freedom that are in the stencil of cell $\cell$.  The entries in $\mat{P}^{\cell}$ are defined by
\begin{equation*}
  \big[ \mat{P}^{\cell} \big]_{ij} =
  \begin{cases} 1, & \text{if}\; j = \nu_i \in \mathsf{nodes}(\cell),\\
    0, & \text{otherwise}.
  \end{cases}
\end{equation*}
Second, since there is only one cell adjacent to each boundary face $\face \in F_B$, we introduce $\mat{R}^{\face} = \mat{R}^{\face,\cell} \mat{P}^{\cell}$ to interpolate directly from the global nodes to the face nodes, and we will drop the superscript $\cell$ from the matrix $\mat{N}_{x}^{\face,\cell}$ when summing over boundary faces.  We also drop the superscript $\cell$ from $\mat{N}_x^{\face,\cell}$ when summing over interior interfaces, $\face \in F_I$, and adopt the convention that the diagonal entries in $\mat{N}_x^{\face}$ are all positive; recall that each interface in our cut-cell Cartesian grid is perpendicular to one of the coordinate axes.  Finally, we use $\mat{R}^{\face,+} = \mat{R}^{\face,\cell} \mat{P}^{\cell}$ to interpolate directly from the global nodes to the quadrature nodes of face $\face$ via the cell on the ``positive'' side of face $\face$; see Figure~\ref{fig:face_convention}. Similarly, we use $\mat{R}^{\face,-}$ to interpolate to quadrature nodes on the ``negative'' side of face $\face$.

\begin{figure}[tbp]
  \begin{center}
    \includegraphics[width=0.65\textwidth]{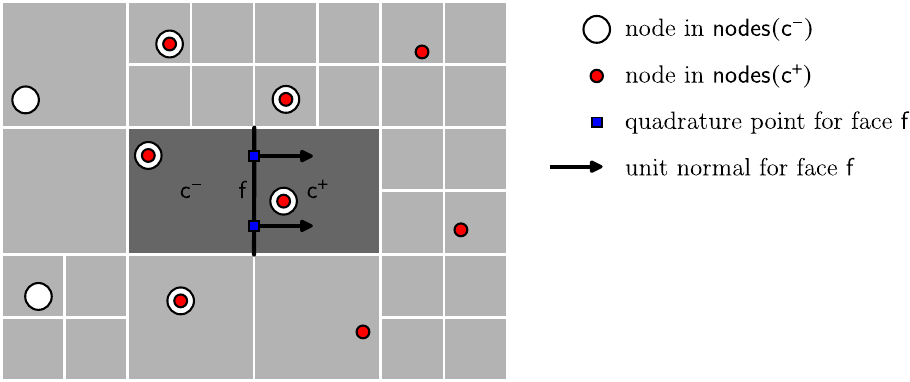}
    \caption[]{\small Illustration of the nodes and quadrature points involved in the interpolation operators $\mat{R}^{\face,-}$ and $\mat{R}^{\face,+}$ for a face $\face$ between two adjacent cells $\cell^{-}$ and $\cell^{+}$ (dark gray squares).  The stencil for the left cell is depicted with large white circles while the stencil for the right cell uses red circles; note that five nodes are shared by the stencils.}\label{fig:face_convention}
  \end{center}
  \end{figure}

\begin{theorem}[Assembly of degenerate SBP pairs]\label{thm:skew_QR}
  Let the domain $\Omega$ be partitioned into non-overlapping cells, $\{ \Omega^{\cell} \}_{\cell=1}^{N_{C}}$, and let $\mat{M}^{\cell}, \mat{Q}_x^{\cell} \in \mathbb{R}^{N^{\cell} \times N^{\cell}}$ denote a degree $p$, degenerate SBP pair on an arbitrary cell $\cell$. Define the matrices $\mat{M}$, $\mat{E}_x$, and $\mat{S}_x$ as follows:
  \begin{align}
    \mat{M} &\equiv \sum_{\cell=1}^{N_{C}} (\mat{P}^{\cell})^T \; \mat{M}^{\cell} \; \mat{P}^{\cell},  \label{eq:global_M} \\
    \mat{E}_x &\equiv \sum_{\face \in F_B} \big( \mat{R}^{\face} \big)^T \; \mat{B}^{\face} \; \mat{N}_{x}^{\face} \; \mat{R}^{\face}, \label{eq:global_E} \\
    \mat{S}_x &\equiv \sum_{\cell=1}^{N_{C}} \big( \mat{P}^{\cell} \big)^T \; \mat{S}_x^{\cell} \; \mat{P}^{\cell} 
    + \frac{1}{2} \sum_{\face \in F_I} \big( \mat{R}^{\face,-} \big)^{T} \; \mat{B}^{\face} \; \mat{N}_{x}^{\face} \; \mat{R}^{\face,+} 
    - \frac{1}{2} \sum_{\face \in F_I} \big( \mat{R}^{\face,+} \big)^{T} \; \mat{B}^{\face} \; \mat{N}_{x}^{\face} \; \mat{R}^{\face,-} \label{eq:global_S}
  \end{align}
  Then $\mat{M}$ and $\mat{Q}_x \equiv \mat{S}_x + {\textstyle \frac{1}{2}}\mat{E}_x$ form a degree $p$, degenerate SBP pair.
\end{theorem}

\begin{proof}
  By inspection, the matrix $\mat{M}$ is diagonal, the matrix $\mat{E}_x$ is symmetric, and the matrix $\mat{S}_x$ is skew symmetric.

  Next, consider the degenerate SBP accuracy constraint, $\mat{Q}_x \mat{V} = \mat{M} \mat{V}_x$.  Using the factorization $\mat{Q}_x = \mat{S}_x + \frac{1}{2}\mat{E}_x$ at both the global and cell levels, and substituting definitions~\eqref{eq:global_E} and \eqref{eq:global_S}, the left-hand side of the accuracy condition \eqref{eq:sbp_accuracy} can be written as five distinct sums (see~\ref{app:detail} for more details):
  \begin{multline}\label{eq:five_sums}
  \mat{Q}_x \mat{V} 
    = \underbrace{\sum_{\cell=1}^{N_{C}} \big( \mat{P}^{\cell} \big)^T \; \mat{Q}_x^{\cell} \; \mat{V}^{\cell}}_{\displaystyle \text{Sum 1}} 
    \quad\underbrace{- \frac{1}{2}\sum_{\cell=1}^{N_{C}} \big( \mat{P}^{\cell} \big)^T \; \mat{E}_x^{\cell} \; \mat{V}^{\cell}}_{\displaystyle \text{Sum 2}}
    \quad\underbrace{+ \frac{1}{2} \sum_{\face \in F_I} \big( \mat{R}^{\face,-} \big)^{T} \; \mat{B}^{\face} \; \mat{N}_{x}^{\face} \; \mat{V}^{\face}}_{\displaystyle \text{Sum 3}} \\[2ex]
    \underbrace{- \frac{1}{2} \sum_{\face \in F_I} \big( \mat{R}^{\face,+} \big)^{T} \; \mat{B}^{\face} \; \mat{N}_{x}^{\face} \; \mat{V}^{\face}}_{\displaystyle \text{Sum 4}}
    \quad\underbrace{+ \frac{1}{2} \sum_{\face \in F_B} \big( \mat{R}^{\face} \big)^T \; \mat{B}^{\face} \; \mat{N}_{x}^{\face} \; \mat{V}^{\face},}_{\displaystyle \text{Sum 5}}
  \end{multline}
  where we have introduced the following interpolated Vandermonde matrices:
  \begin{alignat*}{2}
  \mat{V}^{\cell} &\equiv \mat{P}^{\cell} \mat{V}, & \forall\, \cell&\in \{ 1,2,\ldots,N_{C} \}, \\
  \mat{V}^{\face} &\equiv \mat{R}^{\face} \mat{V}, & \qquad \forall\, \face &\in F_B,\\
  \text{and}\qquad \mat{V}^{\face} &\equiv \mat{R}^{\face,+} \mat{V} = \mat{R}^{\face,-} \mat{V},& \qquad \forall\, \face &\in F_I.
  \end{alignat*}
  The accuracy of the cell-based, degenerate SBP pairs can be used to simplify Sum 1 as follows:
  \begin{equation*}
    \sum_{\cell=1}^{N_{C}} \big( \mat{P}^{\cell} \big)^T \; \underbrace{\mat{Q}_x^{\cell} \; \mat{V}^{\cell}}_{\displaystyle \mat{M}^{\cell}\; \mat{V}_x^{\cell}}
    = \sum_{\cell=1}^{N_{C}} \big( \mat{P}^{\cell} \big)^T \; \mat{M}^{\cell} \; \mat{V}_x^{\cell}
    = \sum_{\cell=1}^{N_{C}} \big( \mat{P}^{\cell} \big)^T \; \mat{M}^{\cell} \; \mat{P}^{\cell} \; \mat{V}_x
    = \mat{M} \mat{V}_x.
  \end{equation*}
  Substituting this result into Equation \eqref{eq:five_sums}, we find 
  \begin{equation*}
    \mat{Q}_x \mat{V} = \mat{M} \mat{V}_x + \text{Sum 2} + \text{Sum 3} + \text{Sum 4} + \text{Sum 5}.
  \end{equation*}
  Consequently, if we can show that Sums 2, 3, 4, and 5 cancel with one another, the accuracy condition will be established.

  Sum 2 can be rearranged into a sum over the interfaces $F_I$ and boundary faces $F_B$.  Every interface is shared by two adjacent cells, so the sum over the interfaces includes two terms while the sum over boundary faces includes only one. 
  \begin{align}
    - \frac{1}{2}\sum_{\cell=1}^{N_{C}} \big( \mat{P}^{\cell} \big)^T \; \mat{E}_x^{\cell} \; \mat{V}^{\cell}
    &= -\frac{1}{2} \sum_{\cell=1}^{N_{C}} \big( \mat{P}^{\cell} \big)^T 
      \sum_{\face \in F^{\cell}} \big( \mat{R}^{\face,\cell} \big)^T \; \mat{B}^{\face} \; \mat{N}_{x}^{\face,\cell} \; \mat{R}^{\face,\cell} \; \mat{V}^{\cell} \notag \\
    &= -\frac{1}{2} \sum_{\face \in F_I} \big( \mat{R}^{\face,-} - \mat{R}^{\face,+} \big)^T \; \mat{B}^{\face} \; \mat{N}_{x}^{\face} \; \mat{V}^{\face} 
    - \frac{1}{2} \sum_{\face \in F_B} \big( \mat{R}^{\face} \big)^T \; \mat{B}^{\face} \; \mat{N}_{x}^{\face} \; \mat{V}^{\face}. \label{eq:sum2}
  \end{align}
  Comparing Sum 2 --- specifically the right-hand side of \eqref{eq:sum2} --- with Sums 3, 4, and 5, we conclude that these four sums cancel, as desired. 

  Finally, $\mat{E}_x$ satisfies Condition~\eqref{eq:sbp_boundary} in Definition~\ref{def:degen_sbp}, because the quadrature rule on each face is $2p$ exact:
  \begin{equation*}
    \big[ \mat{V}^T \mat{E}_x \mat{V} \big]_{ij}
    = \sum_{\face \in F_B} \left[ \big( \mat{V}^{\face} \big)^T \; \mat{B}^{\face} \; \mat{N}_{x}^{\face} \; \mat{V}^{\face} \right]_{ij} 
    = \sum_{\face \in F_B} \int_{\Gamma_{\face}} \fnc{V}_i \fnc{V}_j \, n_x \, d\Gamma = \int_{\Gamma} \fnc{V}_i \fnc{V}_j \, n_x \, d\Gamma.
  \end{equation*}
  Therefore, we have shown that $\mat{M}$ and $\mat{Q}_x$ form a degenerate SBP pair.
\end{proof}

\subsection{Algorithm for constructing degenerate SBP pairs}

Algorithm~\ref{alg:degenerate} summarizes the construction of degenerate SBP pairs $\mat{M}$ and $\mat{Q}_x$ for a given node distribution, domain bounds, level-set function (if necessary), minimum cell edge lengths, and target polynomial degree.  For further details, the relevant section of the paper is listed to the right of each major step.

Some steps in Algorithm~\ref{alg:degenerate} can be implemented in a mathematically equivalent form that is more efficient.  The list below highlights the alternative forms that we have adopted in our implementation.
\begin{itemize}
\item The algorithm only shows the construction of $\mat{Q}_x$, but the operators for all $\sdim$ coordinate directions are constructed at the same time to reduce cost by amortizing certain computations.  For instance, the skew symmetric operators defined by \eqref{eq:skew_part} can reuse the $QR$ factorization $\mat{V}^{\cell} = \mat{U} \mat{L}^{T}$.
\item The algorithm indicates that the quadrature rules are generated before looping over the cells and faces; instead, we generate the tensor-product and cut-cell quadrature rules ``on-the-fly'' during the loops to reduce memory.
\item The algorithm's outputs are shown as $\mat{M}$ and $\mat{Q}_x$; in practice, we return $\mat{S}_x$ and the boundary-face operators $\mat{R}^{\face}$, $\mat{B}^{\face}$, $\mat{N}_x^{\face}$, $\forall\; \face \in F_B$, because the face operators are needed to impose boundary conditions weakly.
\item The diagonal mass matrix $\mat{M}$ is stored as a vector, and we store only the upper triangular part of $\mat{S}_x$ in a sparse-matrix format.
\end{itemize}

\begin{algorithm}[tbp]\DontPrintSemicolon
  \SetKwInOut{Input}{Input}\SetKwInOut{Output}{Output}
  \SetKwComment{Comment}{\# }{}
  \Input{nodes $X$, domain bounds $\{[l_i,u_i]\}_{i=1}^{\sdim}$, level-set $\phi$, minimum cell edge lengths, $\{ \Delta x_{i,\min} \}_{i=1}^{\sdim}$, target degree $p$}
  \Output{Degenerate SBP pair $\mat{M}$ and $\mat{Q}_x$}
  \;
  Generate background mesh  \Comment*[r]{Section \ref{sec:quad_mesh}} 
  Define quadrature rules for cells and faces \Comment*[r]{Section \ref{sec:quad}}
  Create stencil for each cell using Algorithm~\ref{alg:stencil} \Comment*[r]{Section \ref{sec:stencil}}
  \;
  initialize global matrices: $\mat{M} \gets \mat{0}$, $\mat{E}_x \gets \mat{0}$, and $\mat{S}_x \gets \mat{0}$\;
  \For{each cell $\cell$ in the mesh}
  {
    $\mat{M}^{\cell} \gets$ solution of Equation~\eqref{eq:norm_conditions} \Comment*[r]{Section \ref{sec:norm}}
    $\mat{E}^{\cell} \gets$ right-hand side of Equation~\eqref{eq:sym_part} \Comment*[r]{Section \ref{sec:sym_part}}
    $\mat{S}^{\cell} \gets$ right-hand side of Equation~\eqref{eq:skew_part} \Comment*[r]{Section \ref{sec:skew_part}}

    $\mat{M} \gets \mat{M} + \big(\mat{P}^{\cell}\big)^T \; \mat{M}^{\cell} \; \mat{P}^{\cell}$ \Comment*[r]{Section \ref{sec:assemble}}
    $\mat{S}_x \gets \mat{S}_x + \big(\mat{P}^{\cell}\big)^T \;\mat{S}_x^{\cell} \; \mat{P}^{\cell}$ \Comment*[r]{Section \ref{sec:assemble}}
  }
  \;
  \For{each interface $\face \in F_I$}
  {
    $\mat{S}_x \gets \mat{S}_x + \frac{1}{2} \big( \mat{R}^{\face,-} \big)^T \; \mat{B}^{\face} \; \mat{N}_x^{\face} \; \mat{R}^{\face,+}$ \Comment*[r]{Section \ref{sec:assemble}}
    $\mat{S}_x \gets \mat{S}_x - \frac{1}{2} \big( \mat{R}^{\face,+} \big)^T \; \mat{B}^{\face} \; \mat{N}_x^{\face} \; \mat{R}^{\face,-}$ \Comment*[r]{Section \ref{sec:assemble}}
  }
  \;
  \For{each boundary face $\face \in F_B$}
  {
    $\mat{E}_x \gets \mat{E}_x + \big( \mat{R}^{\face})^T \; \mat{B}^{\face} \; \mat{N}_x^{\face} \; \mat{R}^{\face}$ \Comment*[r]{Section \ref{sec:assemble}}
  }
  \;
  $\mat{Q}_x \gets \mat{S}_x + \frac{1}{2} \mat{E}_x$

  \caption{\small Algorithm to construct degenerate SBP pairs.}\label{alg:degenerate}
\end{algorithm}

\subsection{An operator for artificial dissipation}\label{sec:dissipation}

If a partial-differential equation is energy stable, it is usually possible to discretize the equation using SBP operators and maintain (discrete) energy stability without introducing numerical dissipation.  Nevertheless, an SBP discretization with dissipation is often more accurate than one without, because the dissipation helps damp spurious high-frequency modes.

We can build an effective dissipation operator for point-cloud SBP discretizations by taking advantage of the solution jumps between cells that have distinct stencils.  To be precise, the weak-form of the dissipation operator is defined by
\begin{equation}\label{eq:dissipation}
  \mat{A} = \epsilon_{\mathsf{diss}} \sum_{\face \in F_I} (\mat{R}_{\mathsf{diss}}^{\face,+} - \mat{R}_{\mathsf{diss}}^{\face,-})^T \mat{B}_{\mathsf{diss}}^{\face} (\mat{R}_{\mathsf{diss}}^{\face,+} - \mat{R}_{\mathsf{diss}}^{\face,-}),
\end{equation}
where $\epsilon_{\mathsf{diss}} \geq 0$ is a dissipation coefficient.  We set $\epsilon_{\mathsf{diss}} = 1/4$ for the results reported in Section~\ref{sec:results}.

The matrices $\mat{R}_{\mathsf{diss}}^{\face,+}$, $\mat{R}_{\mathsf{diss}}^{\face,-}$, and $\mat{B}_{\mathsf{diss}}^{\face}$ are similar to the operators used in the construction of $\mat{S}_x$.  The only difference is that the dissipation operators interpolate to only one point on the face --- the geometric center --- rather than all the face quadrature nodes.  Consequently, the matrix $\mat{B}_{\mathsf{diss}}^{\face}$ is actually a scalar that holds the area of face $\face$.  Limiting the interpolation to one point on the face reduces the computational cost of applying the dissipation operator, which would otherwise grow in proportion to the number of face quadrature points, \ie $(p+1)^{\sdim}$ for the Gauss-Legendre rules used here.

The dissipation operator $\mat{A}$ does not impact the polynomial exactness of the discretization.  To see this, we let $\mat{A}$ act on the degree $p$ Vandermonde matrix and find that the product vanishes:
\begin{equation*}
  \mat{A} \mat{V} = \epsilon_{\mathsf{diss}} \sum_{\face \in F_I} (\mat{R}_{\mathsf{diss}}^{\face,+} - \mat{R}_{\mathsf{diss}}^{\face,-})^T \mat{B}_{\mathsf{diss}}^{\face} (\mat{V}^{\face}_c - \mat{V}^{\face}_c) = \mat{0},
\end{equation*}
where $\mat{V}^{\face}_{c} \in \mathbb{R}^{1 \times N_{p,\sdim}}$ is the row vector of polynomials evaluated at the center of face $\face$.  This also shows that the operator is globally conservative, since the constant vector is in the column space of $\mat{V}$.

Furthermore, it is straightforward to show the dissipation operator is positive semi-definite.  If we contract $\mat{A}$ with an arbitrary vector $\bm{u} \in \mathbb{R}^N$, we find 
\begin{align*}
  \bm{u}^T \mat{A} \bm{u} &= \epsilon_{\mathsf{diss}} \sum_{\face \in F_I} \bm{u}^T (\mat{R}_{\mathsf{diss}}^{\face,+} - \mat{R}_{\mathsf{diss}}^{\face,-})^T \mat{B}_{\mathsf{diss}}^{\face} (\mat{R}_{\mathsf{diss}}^{\face,+} - \mat{R}_{\mathsf{diss}}^{\face,-}) \bm{u} \\
  &= \epsilon_{\mathsf{diss}} \sum_{\face \in F_I} \mat{B}_{\mathsf{diss}}^{\face} 
  \big( \Delta u_c^{\face} \big)^2 \geq 0,
\end{align*}
where $\Delta u_c^{\face} \equiv (\mat{R}_{\mathsf{diss}}^{\face,+} - \mat{R}_{\mathsf{diss}}^{\face,-}) \bm{u}$, and we note that the face area $\mat{B}_{\mathsf{diss}}^{\face}$ is positive, by definition.

\section{Enforcing a positive-definite diagonal norm}\label{sec:optimize}

Degenerate SBP operators with indefinite $\mat{M}$ can be used to solve steady-state problems, because the weak formulation of the discretization only requires the diagonal norm to be a sufficiently accurate quadrature.  On the other hand, an indefinite $\mat{M}$ is not suitable for unsteady problems since the discretization will be unstable, in general.  Unfortunately, for a given node distribution $X$, the mass matrix produced by Algorithm~\ref{alg:degenerate} is often indefinite for $p=1$ and almost always indefinite for $p\geq 2$.

In order to apply our methodology to initial boundary value problems, we require some way of ensuring that the diagonal mass matrix $\mat{M}$ is positive definite.  To this end, we take advantage of the additional degrees of freedom in each cell's stencil to adjust the values in $\mat{M}^{\cell}$. 

To the best of our knowledge, it is not possible to adjust the cell-based mass matrices independently in order to achieve a global, positive-definite mass matrix.  This is because the diagonal entries of $\mat{M}^{\cell}$ are coupled to one another --- see Equation~\eqref{eq:cell_norm_sol} --- such that increasing the value of one entry may decrease another entry.  Consequently, when attempting to make one node's (global) quadrature weight positive we may inadvertently make another weight negative.

To address the coupled nature of the entries in the diagonal mass matrix, we formulate a large, sparse, linear inequality, which we can solve using efficient methods from linear programming.  The goal of this section is to describe this linear inequality and our approach to solving it.

\begin{remark}
In an early version of this work\footnote{The early work was presented at the North American High Order Methods Conference (NAHOMCon) in June 2024; interested readers can contact the corresponding author for a copy of the slides.}, we optimized the node locations to ensure a positive-definite mass matrix.  While this approach worked well, the unconstrained optimization problem was highly nonlinear and ill-conditioned, and it required a sophisticated algorithm to ensure efficient convergence.  By contrast, the approach described here results in a linear program for which many efficient algorithms are readily available. 
\end{remark}

\subsection{Linear inequality for positivity of $\mat{M}$}

Recall that the general solution to the cell-based quadrature is given by Equation~\eqref{eq:cell_norm_sol}.  Weights for a global quadrature can be constructed by assembling these cell-based weights using the formula below:
\begin{equation*}
  \bm{m} \equiv \sum_{\cell=1}^{N_{C}} \big(\mat{P}^{\cell} \big)^T \big( \bm{m}_{\min}^{\cell} + \mat{Z}^{\cell} \bm{y}^{\cell} \big),
\end{equation*}
which can be written succinctly as
\begin{equation}\label{eq:norm_general}
  \bm{m} = \bm{m}_{\min} + \mat{Z} \bm{y},
\end{equation}
where
\begin{equation}\label{eq:m_min_and_Z_def}
\bm{m}_{\min} = \sum_{\cell=1}^{N_{C}} \big(\mat{P}^{\cell} \big)^T \bm{m}_{\min}^{\cell}, \qquad \mat{Z} = \sum_{\cell=1}^{N_{C}} \big(\mat{P}^{\cell} \big)^T \mat{Z}^{\cell} \mat{P}^{\cell}_y,
\end{equation}
and the vector $\bm{y} \in \mathbb{R}^{N_y}$ is the concatenation of the cell-based vectors $\bm{y}^{\cell}$, with $N_{y} = \sum_{\cell=1}^{N_{C}} (N^{\cell} - N_{2p-1,\sdim})$.  We have also introduced the binary prolongation matrix $\mat{P}_{y}^{\cell}$, which maps elements in the matrix $\bm{y}$ to the cell-based unknowns: $\mat{P}_y^{\cell} \bm{y} = \bm{y}^{\cell}$.  Although the construction of $\mat{Z}$ in \eqref{eq:m_min_and_Z_def} appears similar to the global SBP matrices, we emphasize that $\mat{Z}$ is an $N \times N_y$ rectangular matrix.

We want each weight in $\bm{m} \in \mathbb{R}^{N}$ to be positive, but we may want the flexibility to place distinct lower bounds on each node independently.  This is because the nodes will be distributed non-uniformly, in general, and so we expect the weights to be larger where the nodes are more sparse and smaller where they are clustered.  To accommodate this flexibility, let $\tau_{i} > 0$ denote the positive tolerance for the weight at node $i$, \ie $m_i \geq \tau_i$, and let $\bm{\tau}$ denote the vector holding all the tolerances in an order consistent with $\bm{m}$.  Then, it follows from \eqref{eq:norm_general} that the linear inequality that we want satisfied by $\bm{y}$ is 
\begin{equation}\label{eq:norm_inequality}
  \bm{m}_{\min} + \mat{Z} \bm{y} \geq \bm{\tau},
\end{equation}
where the inequality is to be interpreted element-wise.  We will refer to \eqref{eq:norm_inequality} as the norm-inequality problem or, simply, the norm inequality.

\subsection{When is the norm-inequality problem feasible?}

It is not immediately obvious that there is a solution to \eqref{eq:norm_inequality}; indeed, as the results in Section~\ref{sec:results} demonstrate, sometimes the norm-inequality problem does not have a solution.  Fortunately, the theorem below suggests that the situations where a solution fails to exist can be anticipated and easily addressed by adding nodes, and this is supported by the results.

We establish conditions for the feasibility of the norm-inequality problem by leveraging the theory of positive least-squares cubature formulas from~\cite{Glaubitz2022construction}.  One of the hypotheses required by the theory in \cite{Glaubitz2022construction} is that the nodes $X$ are part of an equidistributed sequence\footnote{For our purposes, a sequence $X_N$ is equidistributed on $\Omega$ if \cite{Kuipers2012uniform}
\begin{equation*}
\lim_{N\rightarrow \infty} \frac{|\Omega|}{N} \sum_{i=1}^{N} \mathcal{F}(\bm{x}_i) = \int_{\Omega} \mathcal{F}(\bm{x}) \,d\Omega,\qquad \forall\, \mathcal{F} \in L^{2}(\Omega).
\end{equation*}}
parameterized by $N$, \eg Halton points~\cite{Halton1960efficiency}.  Thus, in the statement of the theorem we use the notation $X_N$ to indicate that the node set includes all the points in the equidistributed sequence up to the $N$th.

In addition, the following theorem only considers the case where the tolerances are zero in \eqref{eq:norm_inequality}, since this is the most important case: if we cannot find a solution with $\bm{\tau} = \bm{0}$ then there is no hope of finding a solution with more aggressive bounds.

\begin{theorem}[Norm-inequality feasbility]\label{thm:feasibility}
Let $X_N \subset \Omega$ be a node set consisting of the first $N \in \mathbb{N}$ points from an equidistributed sequence, and let $\Omega$ be partitioned into a non-overlapping mesh of cells.  Assume the vector $\bm{m}_{\min}$ and matrix $\mat{Z}$ are constructed according to \eqref{eq:m_min_and_Z_def} using $X_N$, the mesh, and a given polynomial degree $p \geq 1$.  If the null-space of $\mat{Z}^T$ is equal to the column space of the Vandermonde matrix $\mat{V}_{2p-1}$, 
then there is an $N^* \in \mathbb{N}$ such that inequality $\bm{m}_{\min} + \mat{Z}\bm{y} \geq \bm{0}$ has a solution for all $N \geq N^*$.
\end{theorem}

\begin{proof}
According to Farkas Lemma, either $\bm{m}_{\min} + \mat{Z}\bm{y} \geq \bm{0}$ has a solution, or there exists a vector $\bm{u} \in \mathbb{R}^{N}$ satisfying $\bm{u}^T \mat{Z} = \bm{0}^T$, $\bm{u} \geq \bm{0}$, and $\bm{m}_{\min}^T \bm{u} < 0$. We will show that, for sufficiently large $N$, no such $\bm{u}$ exists for equidistributed $X_N$.

For the proof, we adopt a polynomial basis $\{\fnc{V}_j\}_{j=1}^{N_{2p-1,\sdim}}$ that is orthogonal with respect to the integral $L^2$ norm over $\Omega$.  Furthermore, we assume that the first basis function is a constant that satisfies $\int_{\Omega} \fnc{V}_{1} \, d\Omega = 1$.  We note that changing the basis does not alter $\bm{m}_{\min}$, nor does it change $\mat{Z}$.

By assumption, the null-space of $\mat{Z}^T$ is the column space of $\mat{V}_{2p-1}$.  It follows that, in order to satisfy $\bm{u}^T \mat{Z} = \bm{0}^T$, the vector $\bm{u}$ must be a linear combination of the columns of $\mat{V}_{2p-1}$, that is
\begin{equation*}
  \bm{u} = \mat{V}_{2p-1} \bm{\alpha},
\end{equation*}
where $\bm{\alpha} \in \mathbb{R}^{N_{2p-1,\sdim}}$ are coefficients to be determined.

Since $\bm{m}_{\min}$ are quadrature weights that are exact for degree $2p-1$ polynomials, and $\bm{u}$ is a linear combination of such polynomials, we can numerically integrate $\bm{u}$ to find 
\begin{equation}\label{eq:Farkas_cond}
  \bm{m}_{\min}^T \bm{u} = \bm{m}_{\min}^T \mat{V}_{2p-1} \bm{\alpha} = \alpha_1,
\end{equation}
where we used $\bm{m}_{\min}^T \mat{V}_{2p-1} = [1,0,\ldots,0]$, which follows from the orthogonality of the basis and the normalization of $\fnc{V}_1$.  If we select $\alpha_1 < 0$ then \eqref{eq:Farkas_cond} implies that $\bm{m}_{\min}^T \bm{u} < 0$.

To conclude the proof using contradiction, suppose there are coefficients $\bm{\alpha}$ such that $\bm{u} = \mat{V}_{2p-1} \bm{\alpha} \geq \bm{0}$.  Using the hypothesis that the nodes $X_N$ are equidistributed, we can apply \cite[Corollary 3.6]{Glaubitz2022construction} to infer that there exists an $N^*$ such that, for all $N \geq N^*$, there is a least-squares cubature formula over $X_N$ with positive weights $\bm{w} \in \mathbb{R}^{N}$ that integrates polynomials in $\mathbb{P}^{2p-1}(\Omega)$ exactly.  Using these weights with the assumption that $\bm{u} \geq \bm{0}$, we find 
\begin{equation*}
 \bm{w}^T \bm{u} \geq 0 \qquad \Rightarrow\qquad 
\bm{w}^T \mat{V}_{2p-1} \bm{\alpha} = \alpha_1 \geq 0,
\end{equation*}
which contradicts $\alpha_1 < 0$.  Thus, for $N \geq N^*$, we have shown that we cannot have $\bm{u}^T \mat{Z} = \bm{0}^T$, $\bm{u} \geq 0$, and $\bm{m}_{\min}^T \bm{u} < 0$ simultaneously.  Thus, Farkas Lemma implies that $\bm{m}_{\min} + \mat{Z}\bm{y} \geq \bm{0}$ has a solution for $N \geq N^*$.
\end{proof}

\begin{remark}
In the last stage of the proof, we cannot use $\bm{m}_{\min}$ to integrate the vector of inequalities $\bm{u} \geq \bm{0}$, because some entries in $\bm{m}_{\min}$ may be negative and would flip their corresponding inequality.  Thus, we must resort to the hypothesis that the nodes $X_N$ are equidistributed and use the conditional existence of a least-squares cubature formula with positive weights.
\end{remark}

The assumptions in the statement of Theorem~\ref{thm:feasibility} raise a few practical concerns:
\begin{itemize}
\item How important is it to use an equidistributed sequence for the nodes?
\item What if $N^*$ is impractically large?
\item Under what conditions is the null-space of $\mat{Z}^T$ equal to the column space of $\mat{V}_{2p-1}$?
\end{itemize}
To address the first concern, we note that the result would still follow as long as a \emph{subset} of the nodes are equidistributed.  In the adapted proof, we would numerically integrate $\bm{u}$ by applying the least-squares weights $\bm{w}$ to the equidistributed subset of nodes and applying weights of zero to the remaining nodes.

We appeal to the numerical results in Section~\ref{sec:success_rate} for the concern regarding the size of $N^*$.  When using quasi-uniform nodes, we find that the norm-inequality only fails to have a solution when $N \approx N_{2p-1,\sdim}$, \ie, node distributions so sparse that numerical errors would be unacceptably large anyway.  When the nodes are distributed non-uniformly, \eg, to resolve a boundary layer, then $N^*$ will probably need to be larger to ensure sufficient coverage of nodes over the entire domain.

The last assumption that warrants discussion is the null-space of $\mat{Z}^T$.  First, it is straightforward to show that the null-space \emph{includes} the column space of $\mat{V}_{2p-1}$.  Recall from Section \ref{sec:norm} that the columns of $\mat{Z}^{\cell}$ are in the null-space of the transposed, cell-based Vandermonde matrix, \ie $[\mat{V}_{2p-1}^{\cell}]^T \mat{Z}^{\cell} = \mat{0}$.  Thus, the assembled matrix $\mat{Z}$ is in the null-space of the transposed (global) Vandermonde matrix $\mat{V}_{2p-1}^T$:
\begin{equation*}
\mat{V}_{2p-1}^T \mat{Z} = \mat{V}_{2p-1}^T \left( \sum_{\cell=1}^{N_{C}} \big(\mat{P}^{\cell} \big)^T \mat{Z}^{\cell} \mat{P}_y^{\cell} \right) = \sum_{\cell=1}^{N_{C}} \underbrace{\big(\mat{V}_{2p-1}^{\cell}\big)^T \mat{Z}^{\cell}}_{\displaystyle = \mat{0}} \mat{P}_y^{\cell} = \mat{0}.
\end{equation*}
The challenge is showing that the null-space \emph{is equal to} the column space of $\mat{V}_{2p-1}$.  Proving this is beyond the scope of this work, but we offer some intuition regarding when it should hold.  First, the row rank of $\mat{Z}$ must be $N - N_{2p-1,\sdim}$, otherwise there would another (nontrival) vector $\bm{u} \in \mathbb{R}^{N}$ that is not in the column space of $\mat{V}_{2p-1}$ but nevertheless satisfies $\bm{u}^T \mat{Z} = \bm{0}^T$.  To ensure that $\mat{Z}$ has this maximal rank, its rows must be sufficiently ``independent.''  For example, if each node $i$ is in the stencil of a unique set of cells --- \ie no other node is associated with the same set of cells --- then the sparsity pattern of each row of $\mat{Z}$ will be distinct.  Unfortunately, while this might provide a path to proving the result we want, we hypothesize that it is a sufficient but not necessary condition.

\subsection{Solving the norm-inequality problem}\label{sec:norm_solve}

Linear inequalities, such as \eqref{eq:norm_inequality}, can be formulated as linear optimization problems with a trivial objective.  Specifically, we consider the linear program 
\begin{equation}\label{eq:linear_program}
\begin{aligned}
\min_{\bm{y}} &\quad 0, \\
\text{s.t.} &\quad \bm{m}_{\min} + \mat{Z} \bm{y} \geq \bm{\tau}.
\end{aligned}
\end{equation}
We solve this linear optimization problem using the interior-point algorithm implemented in the Julia package \texttt{Tulip.jl}~\cite{Tanneau2021design}.  Tulip uses the homogeneous self-dual form of \eqref{eq:linear_program}, which allows it to identify optimal, infeasible, and unbounded solutions.  Moreover, Tulip is able to take advantage of the sparsity of $\mat{Z}$ and, in our experience, its solution time is a fraction ($<$ 1\%) of the total time required to assemble and solve a discretization.

The solution to \eqref{eq:linear_program}, if it exists, is not unique.  Tulip returns a particular solution that has been adequate for our purposes.  Future work should investigate other possible solutions, such as the minimum-norm solution obtained using the objective $\bm{y}^T \bm{y}$ in \eqref{eq:linear_program}.  This would make the optimization problem a convex-quadratic program with a unique solution.

\section{Numerical verifications and demonstrations}\label{sec:results}

This section presents numerical studies that verify the properties of the proposed point-cloud SBP operators.  In particular, we dedicate several studies to investigating the diagonal mass matrix, or norm, since the practical value of the proposed operators hinges on the existence and characteristics of the mass matrix.  All results were obtained using the implementation in the \texttt{CloudSBP.jl} Julia package~\cite{CloudSBP2024}.

\subsection{Geometries, node distributions, and discretizations}\label{sec:results_setup}

Most of the studies in Sections~\ref{sec:norm_studies}--\ref{sec:runtimes}  share similar geometries, node distributions, and discretizations.  These common elements are described in this section, to avoid repetition later.

\subsubsection{Geometries: box-circle, annulus, and airfoil}\label{sec:geometries}

Three geometries are adopted for the majority of the numerical studies.  The first geometry is a unit square with a circle of radius $r=1/4$ excised from its center.  This geometry was introduced in Figure~\ref{fig:ex_domain} and is defined precisely below:
\begin{equation*}
  \Omega_{\text{box}} \equiv \big\{ \bm{x} \in [0,1]^2 \;|\; \| \bm{x} - \bm{x}_c\| \geq {\textstyle \frac{1}{4}}  \big\}, 
\end{equation*}
where $\bm{x}_c = [1/2, 1/2]^T$ is the domain's geometric center.

The second geometry is an annulus, centered at the origin, whose outer radius is one and whose inner radius is 1/2:
\begin{equation*}
  \Omega_{\text{ann}} \equiv \big\{ \bm{x} \in \mathbb{R}^{2} \;|\; {\textstyle \frac{1}{2}} \leq \| \bm{x} \| \leq 1 \big\}.
\end{equation*}
The annulus geometry can be seen in Figure~\ref{fig:uniform_norm_dist}.

The third geometry is an airfoil shape that we include because it has a non-smooth feature. The airfoil is defined as 
\begin{equation*}
  \Omega_{\text{foil}} \equiv \big\{ \bm{x} \in [0,1]\times[\textstyle{-\frac{1}{10}}, \textstyle{\frac{1}{10}}] \;|\; x(x - 1)^2 - 16y^2 \geq 0 \big\}.
\end{equation*}
The rectangular domain $[0,1]\times[-1/10,1/10]$ in the definition of $\Omega_{\text{foil}}$ is chosen to produce a sharp trailing edge at $\bm{x} = [1,0]^T$ rather than a pair of intersecting curves.  Figure~\ref{fig:airfoil} shows the airfoil geometry.

\begin{figure}[tbp]
  \begin{center}
    \includegraphics[width=\textwidth]{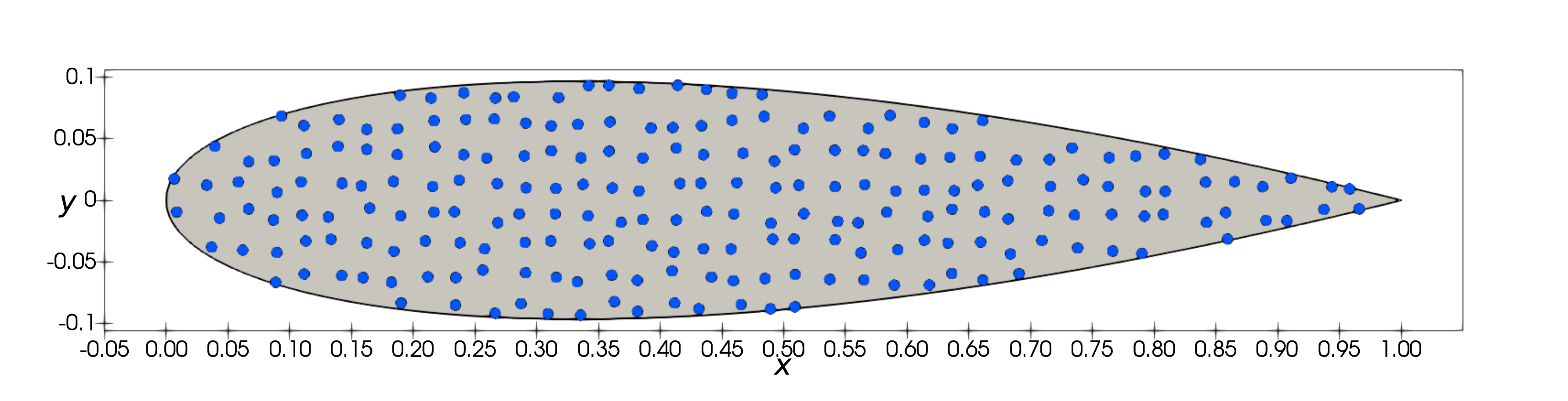}
    \caption[]{\small Airfoil geometry $\Omega_{\text{foil}}$ and an example node distribution used for the accuracy studies.}\label{fig:airfoil}
  \end{center}
\end{figure}

\subsubsection{Node distributions}\label{sec:node_distribution}

The node distributions that we use depend on the geometry.  For the box-circle geometry, we generate a uniform grid with $n_x = n_y$ nodes in each coordinate direction, and then randomly perturb the coordinates.  We then remove all nodes that are inside the circle of radius 1/4.  More precisely, consider the set of nodes $\hat{X}_{\text{box}} = \{ \bm{x}_{i} \}_{i=1}^{n_x^2}$ where the coordinates of the $i$th node are given by 
\begin{alignat*}{2}
  x_i &= \big( j - {\textstyle \frac{1}{2}} \big) \Delta x + \xi_i,  &\qquad &\forall j \in \{1,2,\ldots, n_x \}, \\
  y_i &= \big( k - {\textstyle \frac{1}{2}} \big) \Delta y + \eta_i, 
  &\qquad &\forall k \in \{1,2,\ldots, n_x \},
\end{alignat*}
with $\Delta x = \Delta y = 1/n_x$ and $i = k n_x + j$.  The realizations $\xi_i$ and $\eta_i$ are drawn from the uniform distribution $U[-\Delta x/4, \Delta x/4]$.  The final node set for a particular realization is given by 
\begin{equation*}
  X_{\text{box}} = \big\{ \bm{x} \in \hat{X}_{\text{box}} \; | \; \| \bm{x} - \bm{x}_c \| \geq {\textstyle \frac{1}{4}} \big\}.
\end{equation*}
Using this procedure the number of nodes $N = |X_{\text{box}}|$ depends on the realization.  See Figure~\ref{fig:ex_points} for an example of the resulting node distribution.

For the annulus domain, we consider both quasi-uniform and nonuniform (graded) node distributions.  We begin with a rectilinear grid in polar coordinates, which is then perturbed randomly before being mapped to Cartesian coordinates.  Let $n_r$ and $n_{\theta}$ denote the number of nodes in the radial and angular directions, respectively.  Then the polar coordinates of the nodes are $\hat{X}_{\text{ann}} = \{ [r_i, \theta_i]^T  \}_{i=1}^{n_r n_{\theta}}$ where $r_i = \rho_i + (1-\rho_i)/2$ and 
\begin{alignat*}{2}
  \rho_i &= g\left( \big( j - {\textstyle \frac{1}{2}} \big) \Delta \rho + \xi_i \right), & \qquad &\forall j \in \{1,2,\ldots, n_r \}, \\
  \theta_i &= \big( k - {\textstyle \frac{1}{2}} \big) \Delta \theta + \eta_i,
  &\qquad &\forall k \in \{1,2,\ldots, n_{\theta} \},
\end{alignat*}
with $\Delta \rho = 1/n_r$, $\Delta \theta = 2 \pi / n_{\theta}$, and $i = k n_r + j$.  Here the realization $\xi_i \sim U[-\Delta \rho/4,\Delta \rho /4]$ and the realization $\eta_i \sim U[-\Delta \theta/4, \Delta \theta/4 ]$.  The function $g : \mathbb{R} \rightarrow \mathbb{R}$ is used to control the node distribution in the radial direction.  For the following studies we use 
\begin{equation*}
  g(z) = (e^{\beta z} -1)/(e^{\beta} -1),
\end{equation*}
where $\beta \in (0,\infty)$ is the stretching parameter.  If $\beta \approx 0$ we have $g(z) \approx z$ and the distribution becomes quasi uniform (in polar coordinates).  As $\beta$ increases, the nodes cluster toward the inner radius of $\Omega_{\text{ann}}$.   The final node set for the annulus domain is given by 
\begin{equation*}
  X_{\text{ann}} = \{ [r \cos(\theta), r \sin(\theta)]^T \; | \; [r,\theta]^T \in \hat{X}_{\text{ann}} \}.
\end{equation*}

The nodes for the airfoil geometry are generated like the box-circle nodes, except there are five times more nodes in the $x$ direction than the $y$ direction.  We begin by generating the set $\hat{X}_{\text{foil}} = \{ \bm{x}_{i} \}_{i=1}^{5n_y^2}$ where the coordinates of $\bm{x}_i$ are given by 
\begin{alignat*}{2}
  x_i &= \big( j - {\textstyle \frac{1}{2}} \big) \Delta x + \xi_i,  &\qquad &\forall j \in \{1,2,\ldots, 5 n_y \} \\
  y_i &= -\frac{1}{10} + \big( k - {\textstyle \frac{1}{2}} \big) \Delta y + \eta_i, 
  &\qquad &\forall k \in \{1,2,\ldots, n_y \},
\end{alignat*}
and $i = 5 k n_y + j$.  The nominal node spacing is the same in the $x$ and $y$ directions: $\Delta x = \Delta y = 1/(5n_y)$.  The random perturbations $\xi_i$ and $\eta_i$ are drawn from $U[-\Delta x/4,\Delta x/4]$, and the final node set is 
\begin{equation*}
  X_{\text{foil}} = \{ \bm{x} \in \hat{X}_{\text{foil}} \;|\; x(x - 1)^2 - 16y^2 \geq 0 \}.
\end{equation*}
As with the box-circle, $N = | X_{\text{foil}} |$ depends on the sampled perturbations $\xi_i$ and $\eta_i$.  Figure~\ref{fig:airfoil} shows a particular airfoil node distribution with $n_y = 8$ and $N = 220$ nodes total.

\subsubsection{Norm-inequality tolerances}\label{sec:tolerances}

The tolerances $\bm{\tau}$ in the norm inequality~\eqref{eq:norm_inequality} are used to prevent small diagonal entries in $\mat{M}$; however, what constitutes ``small'' depends on the node distribution.  For the box-circle geometry's quasi-uniform nodes, the tolerances are computed using $1/10$ the average node ``volume'':
\begin{equation*}
  \tau_{\text{box},i} = \frac{1}{10} \frac{|\Omega_{\text{box}}|}{N} = \frac{1}{10} \frac{\big( 1 - \pi/16 \big)}{N}, \qquad \forall\; i \{1,2,\ldots,N\};
\end{equation*}
recall that $N$ depends on the particular node realization for $\Omega_{\text{box}}$.

The annulus' node distribution can be either quasi-uniform or clustered, so the following tolerance formula takes this into account:
\begin{equation*}
  \tau_{\text{ann},i} = \frac{1}{10} (\Delta r)_i  r_i (\Delta \theta),
\end{equation*}
where the angular spacing is $\Delta \theta = 2\pi/n_{\theta}$ and the radial spacing is given by
\begin{equation*}
  (\Delta r)_i = g' \left( \big( j - {\textstyle \frac{1}{2}} \big)\Delta \rho \right) \Delta \rho,\qquad \forall\; j \in \{1,2,\ldots, n_r \},\; \forall\; k \in \{1,2,\ldots, n_{\theta} \},
\end{equation*}
and recall that $i = k n_r + j$.  Note that $(\Delta r)_i  r_i (\Delta \theta)$ approximates the node's volume in Cartesian coordinates; like the box-circle tolerance, the factor of $1/10$ permits a quadrature weight one order of magnitude smaller than the nominal volume.

Finally, we use the following tolerance for the airfoil geometry:
\begin{equation*}
\tau_{\text{foil},i} = \frac{1}{10} \Big( \frac{1}{5 n_y} \Big)^{2}, \qquad \forall\; i \{1,2,\ldots,N\}.
\end{equation*}
Again, this represents one tenth of the average node ``volume,'' since the nominal node spacing in the $x$ and $y$ directions is $1/(5 n_y)$ for the airfoil geometry.

\begin{remark}
Setting $\tau_i$ to one tenth the average node ``volume'' is informed by the results of the success-rate study in Section~\ref{sec:success_rate}.  However, this heuristic will likely fail for some node distributions, particularly those that are highly nonuniform and anisotropic, so developing a general algorithm to set $\tau_i$ will be the focus of future work.
\end{remark}

\subsubsection{Linear advection equation and its discretization}\label{sec:advection}

We use the linear advection equation to study the accuracy and stability of the discretizations that use the point-cloud operators.  In two dimensions, the initial-boundary-value problem for linear advection can be written in skew-symmetric form as shown below:
\begin{alignat}{2}
  &\frac{\partial \fnc{U}}{\partial t} + {\textstyle \frac{1}{2}} \bm{\lambda} \cdot \nabla \fnc{U} + {\textstyle \frac{1}{2}} \nabla \cdot \big( \bm{\lambda} \fnc{U} \big) = 0, & \qquad \forall\; \bm{x} &\in \Omega, t \in [0,T], \label{eq:advection} \\
  &\fnc{U}(\bm{x},0) = \fnc{U}_0(\bm{x}), & \qquad \forall\; \bm{x} &\in \Omega, \label{eq:initial_condition} \\
  &\fnc{U}(\bm{x},t) = \fnc{G}(\bm{x},t), & \qquad \forall\; \bm{x} &\in \{ \Gamma \;|\; \bm{\lambda}^T \hat{\bm{n}} \leq 0 \}, t \in [0,T], \label{eq:boundary_condition}
\end{alignat}
where $\bm{\lambda} = [\lambda_x(\bm{x}), \lambda_y(\bm{x})]^T$ is the spatially-varying \emph{and} divergence-free advection velocity.  For the accuracy studies we consider the steady version of \eqref{eq:advection} and drop the initial condition~\eqref{eq:initial_condition}.

We summarize the SBP discretization of \eqref{eq:advection} for completeness, but we will not analyze the discretization formally; beyond the operators' construction, there is nothing novel about the discretization we use.  For proofs of stability, conservation, and accuracy, we direct the reader to the extensive literature.  For example, the situation considered here is similar to the divergence-free linear advection discretization described in \cite{Fernandez2017simultaneous}.

Suppose $\mat{D}_x = \mat{M}^{-1} (\mat{S}_x + \frac{1}{2} \mat{E}_x)$ and $\mat{D}_y = \mat{M}^{-1} (\mat{S}_y + \frac{1}{2} \mat{E}_y)$ are diagonal-norm SBP first-derivative operators corresponding to a given node set $X$ and mesh $\{\Omega^{\cell} \}_{\cell=1}^{N_C}$.  Furthermore, let $\mat{M}^{-1}\mat{A}$ be a dissipation operator on the same set of nodes. Then an SBP semi-discretization of \eqref{eq:advection} is 
\begin{multline}\label{eq:advection_sbp}
  \frac{d \bm{u}}{dt} + \frac{1}{2} \mat{\Lambda}_x \mat{D}_x \bm{u} + \frac{1}{2} \mat{D}_x \mat{\Lambda}_x \bm{u} + \frac{1}{2} \mat{\Lambda}_y \mat{D}_y \bm{u} + \frac{1}{2} \mat{D}_y \mat{\Lambda}_y \bm{u} \\ = -\mat{M}^{-1} \mat{A} \bm{u} + \frac{1}{2}\mat{M}^{-1} \Big[ \mat{E}_x \mat{\Lambda}_x \bm{u} + \mat{E}_y \mat{\Lambda}_y \bm{u} - \sum_{f\in F_B} \big(\mat{R}^{\face}\big)^T \; \mat{B}^{\face} \; \bm{f}^{\face}(\bm{u}) \Big]
\end{multline}
where $\bm{u} \in \mathbb{R}^{N}$ is the semi-discrete solution vector, and 
\begin{equation*}
  \mat{\Lambda}_x = \diag\big(\lambda_{x,1}, \lambda_{x,2}, \ldots, \lambda_{x,N} \big), \qquad\text{and}\qquad
  \mat{\Lambda}_y = \diag\big(\lambda_{y,1}, \lambda_{y,2}, \ldots, \lambda_{y,N} \big).
\end{equation*}
The vector-valued function $\bm{f}^{\face}(\bm{u})$ in \eqref{eq:advection_sbp} is the numerical flux for quadrature nodes on face $\face \in F_{B}$.  At the $i$th quadrature point of face $\face$, the numerical flux is given by
\begin{equation*}
[\bm{f}^{\face}(\bm{u})]_{i} \equiv
\begin{cases}
  \lambda_{n,i} \big[ \mat{R}^{\face} \bm{u} \big]_i & \text{if}\; \lambda_{n,i} > 0, \\
  \lambda_{n,i} \fnc{G}(\bm{x}_i,t), & \text{if}\; \lambda_{n,i} \leq 0,
\end{cases}
\end{equation*}
where $\lambda_{n,i} = \bm{\lambda}(\bm{x}_i)^T \hat{\bm{n}}_i$ is the component of the velocity vector parallel to $\hat{\bm{n}}_{i}$.

For unsteady problems we use the classical fourth-order Runge-Kutta method to solve~\eqref{eq:advection_sbp}.  For steady problems we use Julia's native sparse direct solver (\ie SuiteSparse~\cite{Davis2006direct}) to solve for $\bm{u}$.

\subsection{Diagonal mass-matrix studies}\label{sec:norm_studies}

This section focuses on characterizing the diagonal mass matrix and its construction.  We begin by studying the distribution of entries in $\mat{M}$ before and after solving the linear inequality \eqref{eq:norm_inequality}.  Subsequently, we investigate the success rate for finding a positive-definite $\mat{M}$ for a large sample of geometries and nodes.  Finally, we verify the accuracy of $\mat{M} = \diag(\bm{m})$ for numerical integration.

We will use ``quadrature weight'' and ``diagonal entry of $\mat{M}$'' interchangeably.  In addition, we will occasionally use $\bm{m}_{\min}$ when referring to the value of the weights before solving the linear inequality.

\subsubsection{Quadrature weight distribution}

The goal of this study is to better understand how the diagonal entries in $\mat{M}$ are distributed before and after solving the norm inequality.  We consider the annulus geometry $\Omega_{\text{ann}}$ with both quasi-uniform ($\beta=0.1$) and clustered ($\beta=4$) nodes.  For both node distributions, we use $n_r = 16$ and $n_{\theta} = 96$ nodes in the radial and angular directions, respectively.  Please refer to Sections~\ref{sec:geometries} and \ref{sec:node_distribution} for the further details regarding the geometry and node distribution.

\begin{remark}
The quasi-uniform nodes are not equidistributed --- despite being quasi-uniform in $(r,\theta)$ space --- because the nodes become more sparsely distributed as we move away from the center of the annulus.  Consequently, we should expect the quadrature weights to vary in the radial direction for both the $\beta=0.1$ and the $\beta=4$ cases.
\end{remark}

Figure~\ref{fig:uniform_norm_dist} shows the results for the quasi-uniform node distribution ($\beta=0.1$).  The figures in the first column depict $\mat{M}$ before the norm inequality is solved.  For an arbitrary node $i$, the quadrature weight $[\mat{M}]_{ii}$ is shown as a circle centered at the node's coordinate $\bm{x}_{i}$.  The size of the circle indicates the magnitude of the weight, while the fill indicates its sign: white for negative and colored for positive.  Here and throughout the results we will use red, blue, green and gray for polynomial degrees $p=1$, 2, 3, and 4, respectively.  

The mass matrices have many negative entries before optimization, and the number of negative entries grows with the degree $p$ (moving top to bottom in Figure~\ref{fig:uniform_norm_dist}).  The magnitude of the entries in $\bm{m}_{\min}$ appears to decrease somewhat away from the boundaries, particularly for the higher order mass matrices, but no other pattern is discernible. 

The figures in the center column of Figure~\ref{fig:uniform_norm_dist} depict the weights after enforcing the norm inequality.  There are no white circles in Figures \ref{fig:norm_after_uni_1}, \ref{fig:norm_after_uni_2}, \ref{fig:norm_after_uni_3}, and \ref{fig:norm_after_uni_4}.  The weights also show a clear pattern, increasing as we move from nodes near the inner radius to nodes near the outer radius.

The histograms in the rightmost column of Figure~\ref{fig:uniform_norm_dist} show the distribution of quadrature weights before and after solving the linear inequality.  Note the use of a log scale for the vertical (count) axis.  As the polynomial degree increases, the distribution of $\bm{m}_{\min}$ shifts increasingly to the left (\ie negative).  After solving the norm inequality, the distributions have a smaller spread of values, and they have all shifted to the right of the vertical axis as expected.

\begin{figure}[p]
  \begin{subfigure}[t]{0.32\textwidth}
    \centering
    \includegraphics[width=0.85\textwidth]{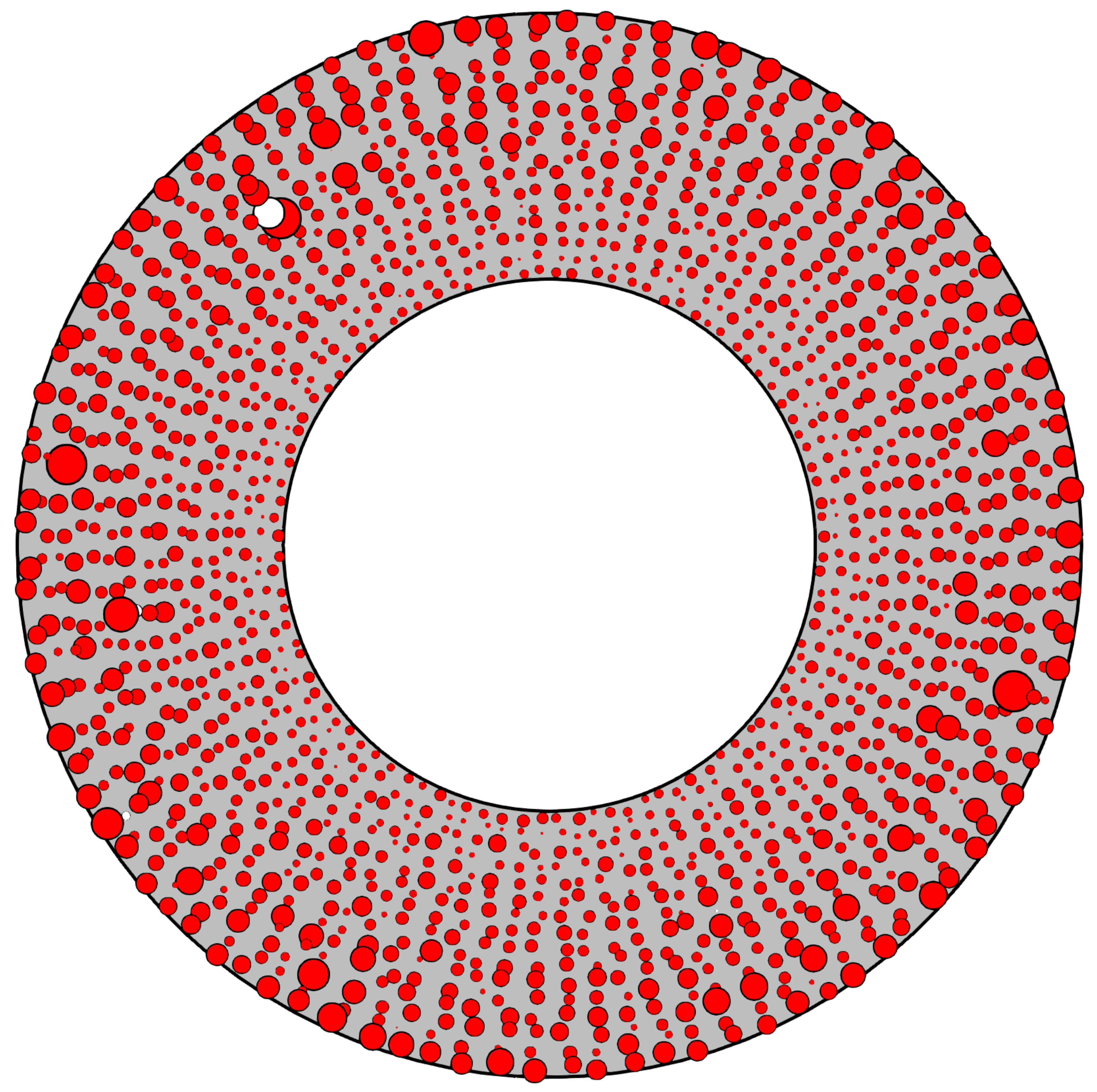}
    \caption{\small $p=1$ weights before} \label{fig:norm_before_uni_1}
  \end{subfigure}%
  \hfill%
  \begin{subfigure}[t]{0.32\textwidth}
    \centering
    \includegraphics[width=0.85\textwidth]{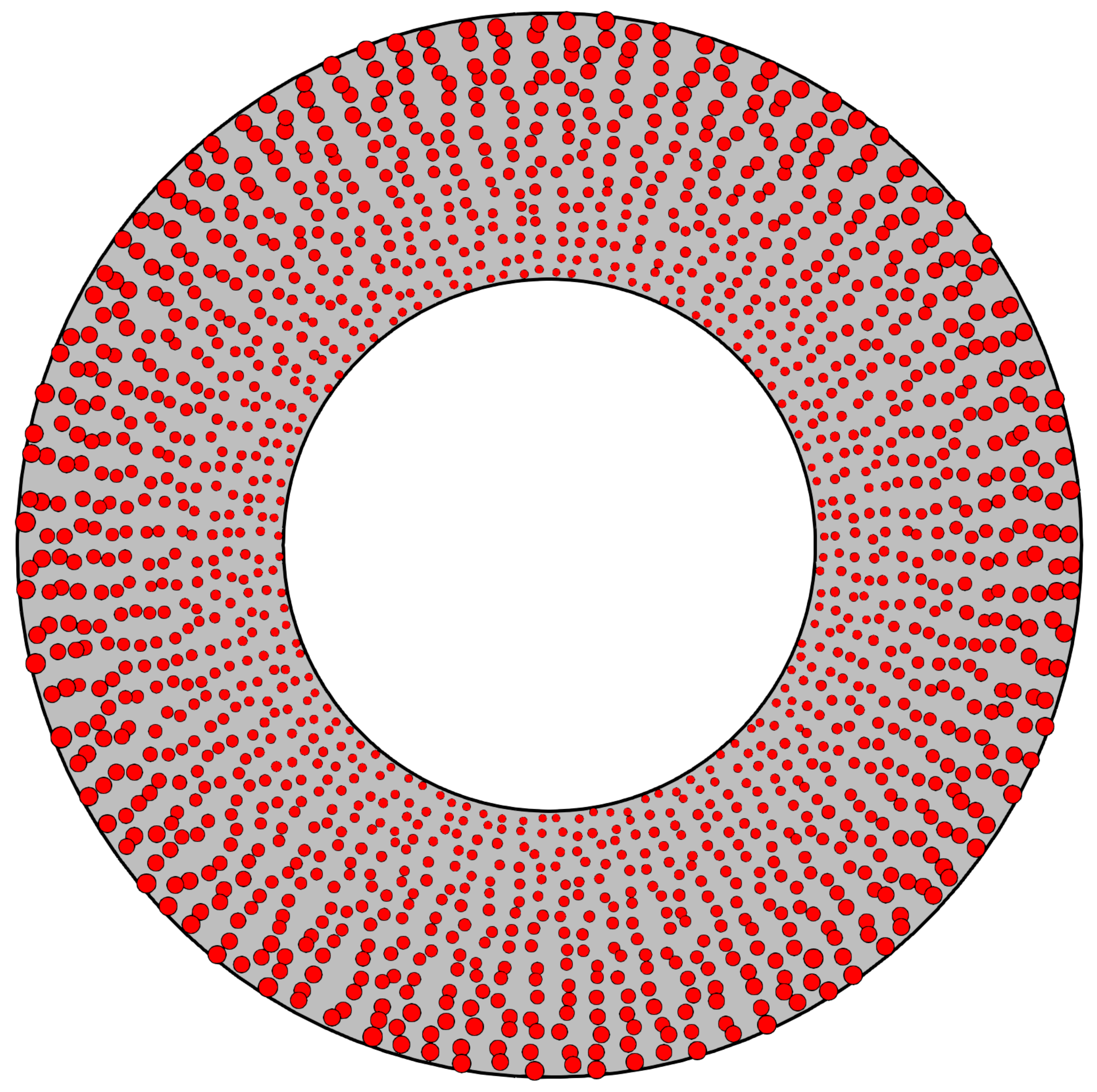}
    \caption{\small $p=1$ weights after} \label{fig:norm_after_uni_1}
  \end{subfigure}%
  \hfill%
  \begin{subfigure}[t]{0.32\textwidth}
    \centering
    \includegraphics[width=0.9\textwidth]{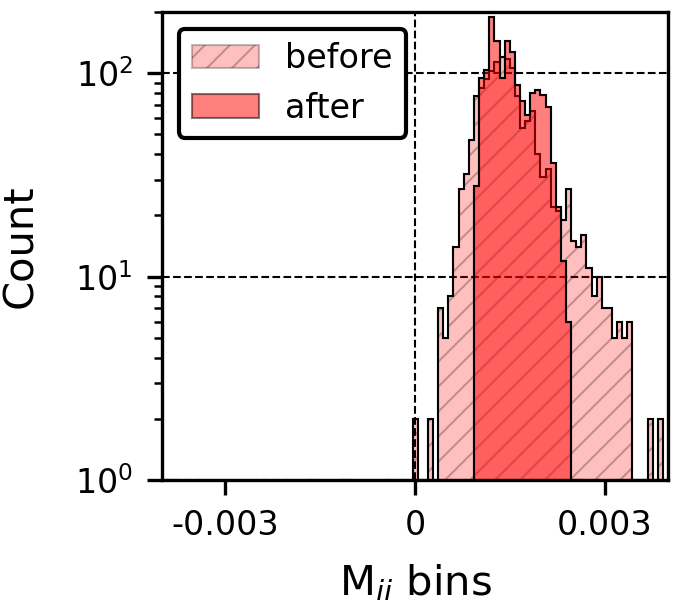}
    \caption{\small $p=1$ distribution} \label{fig:norm_dist_uni_1}
  \end{subfigure}\\%
  \begin{subfigure}[t]{0.32\textwidth}
    \centering
    \includegraphics[width=0.85\textwidth]{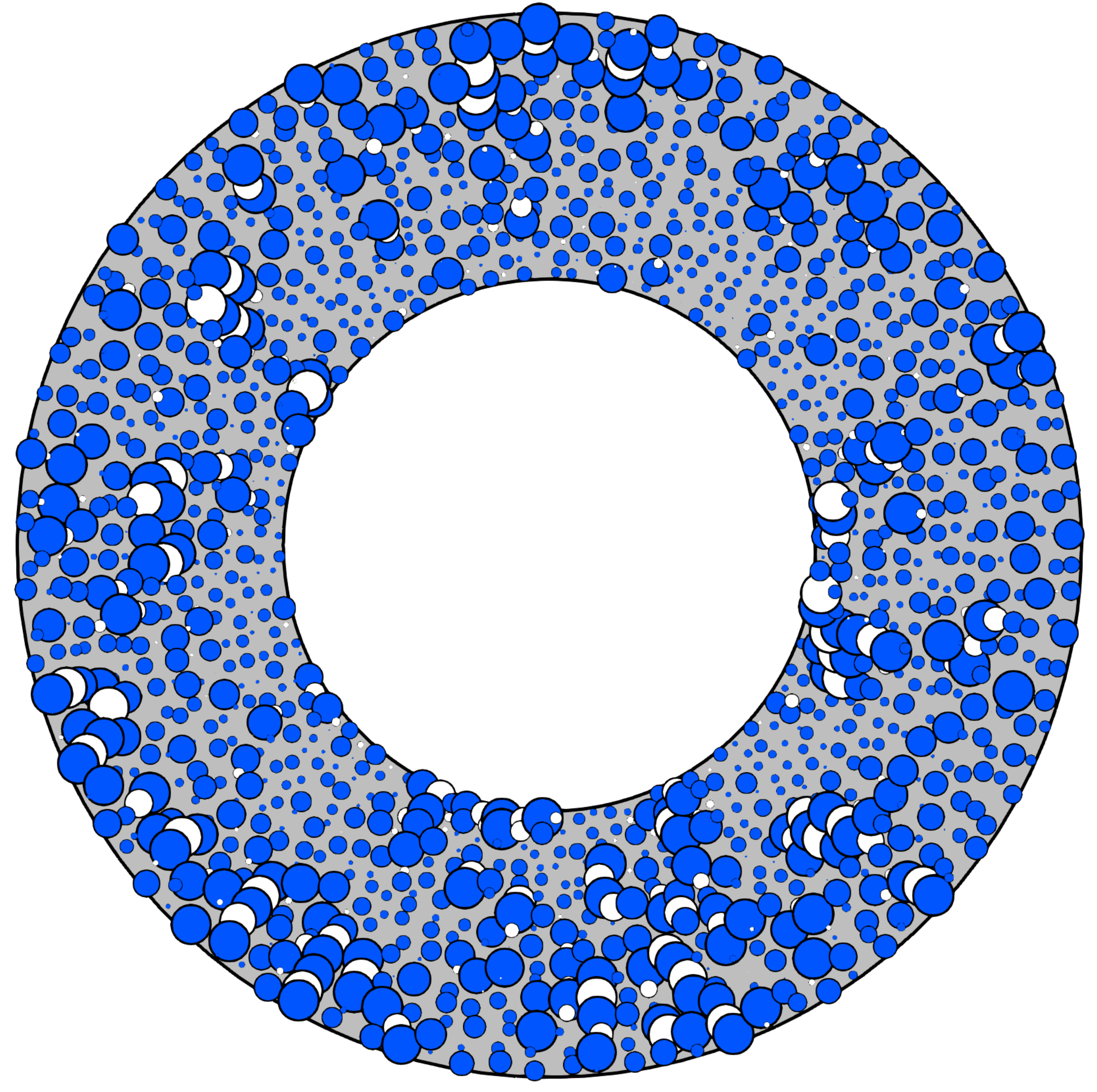}
    \caption{\small $p=2$ weights before} \label{fig:norm_before_uni_2}
  \end{subfigure}%
  \hfill%
  \begin{subfigure}[t]{0.32\textwidth}
    \centering
    \includegraphics[width=0.85\textwidth]{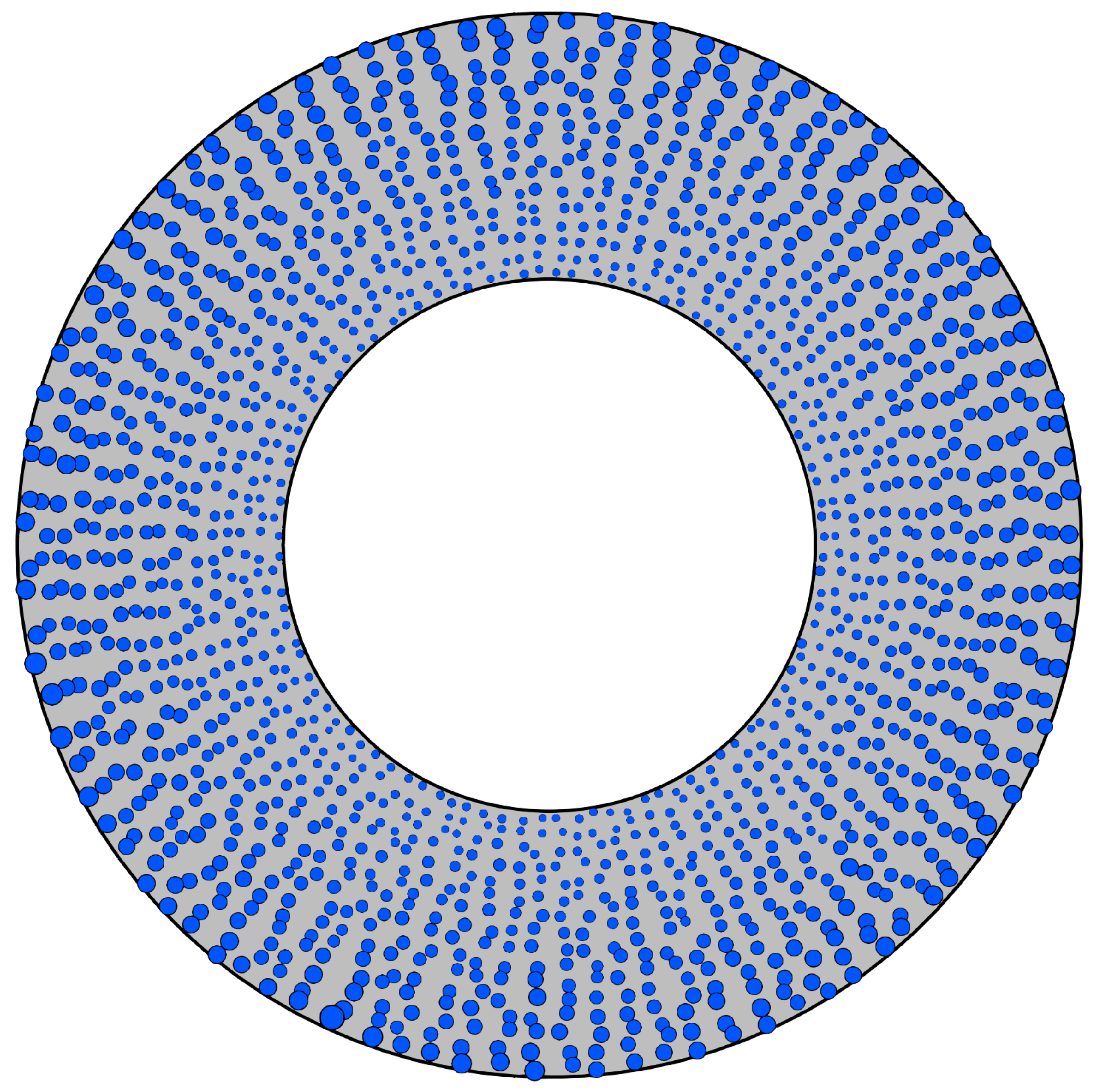}
    \caption{\small $p=2$ weights after} \label{fig:norm_after_uni_2}
  \end{subfigure}%
  \hfill%
  \begin{subfigure}[t]{0.32\textwidth}
    \centering
    \includegraphics[width=0.9\textwidth]{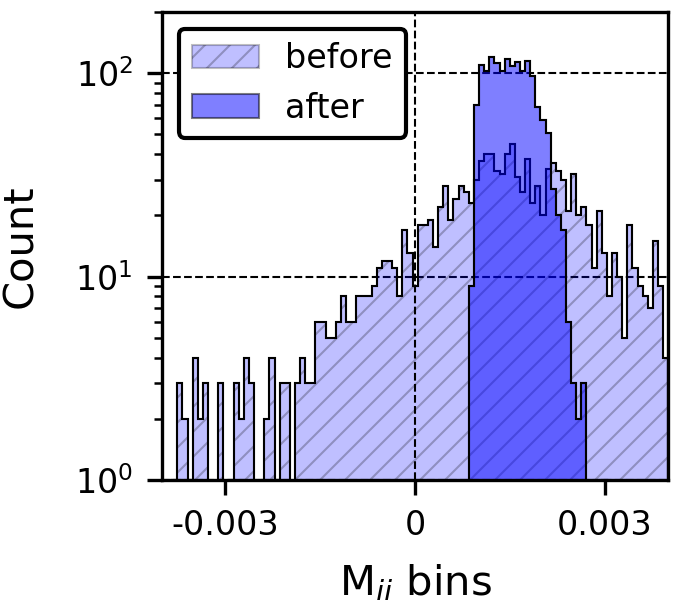}
    \caption{\small $p=2$ distribution} \label{fig:norm_dist_uni_2}
  \end{subfigure}\\%
  \begin{subfigure}[t]{0.32\textwidth}
    \centering
    \includegraphics[width=0.85\textwidth]{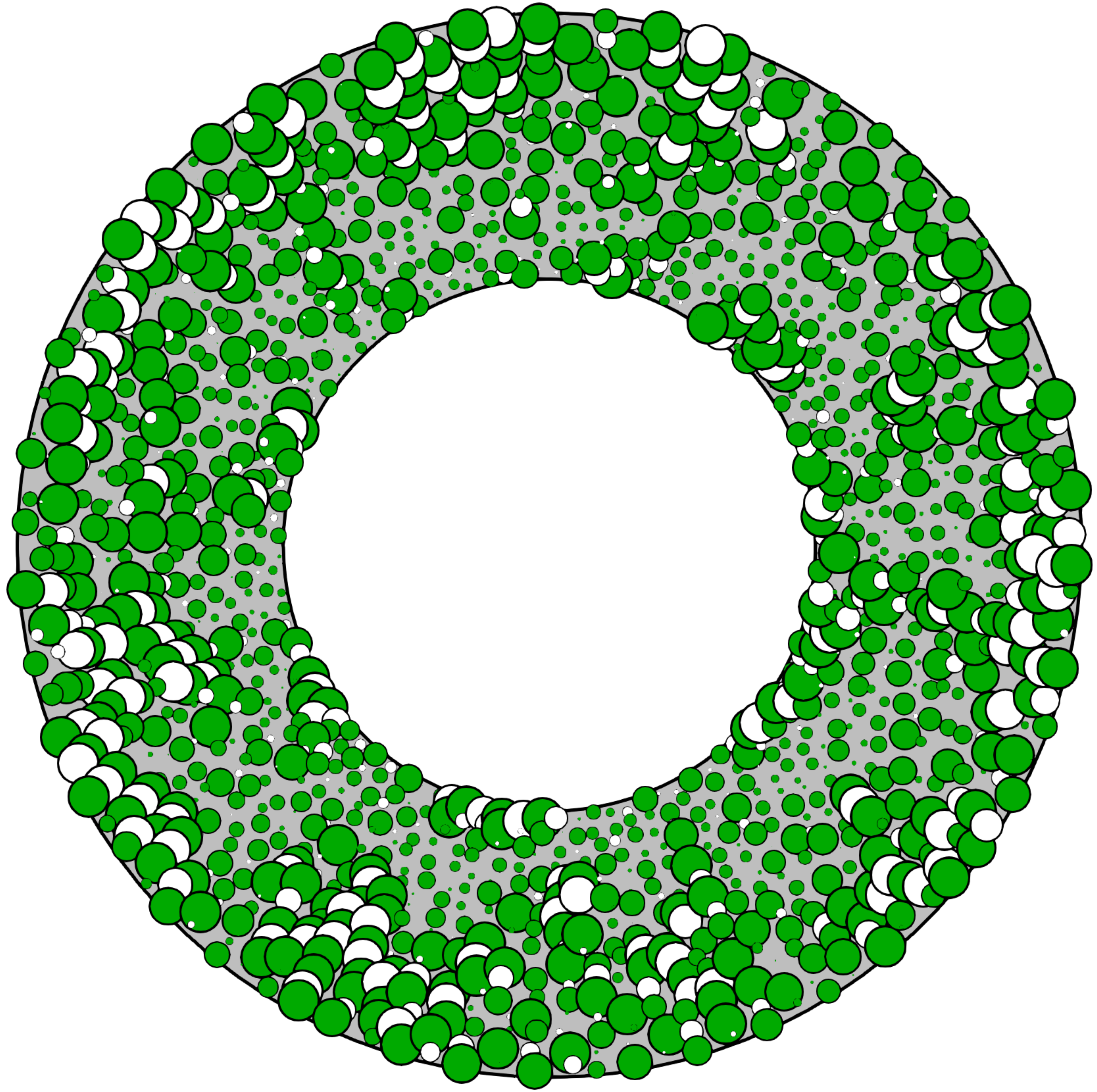}
    \caption{\small $p=3$ weights before} \label{fig:norm_before_uni_3}
  \end{subfigure}%
  \hfill%
  \begin{subfigure}[t]{0.32\textwidth}
    \centering
    \includegraphics[width=0.85\textwidth]{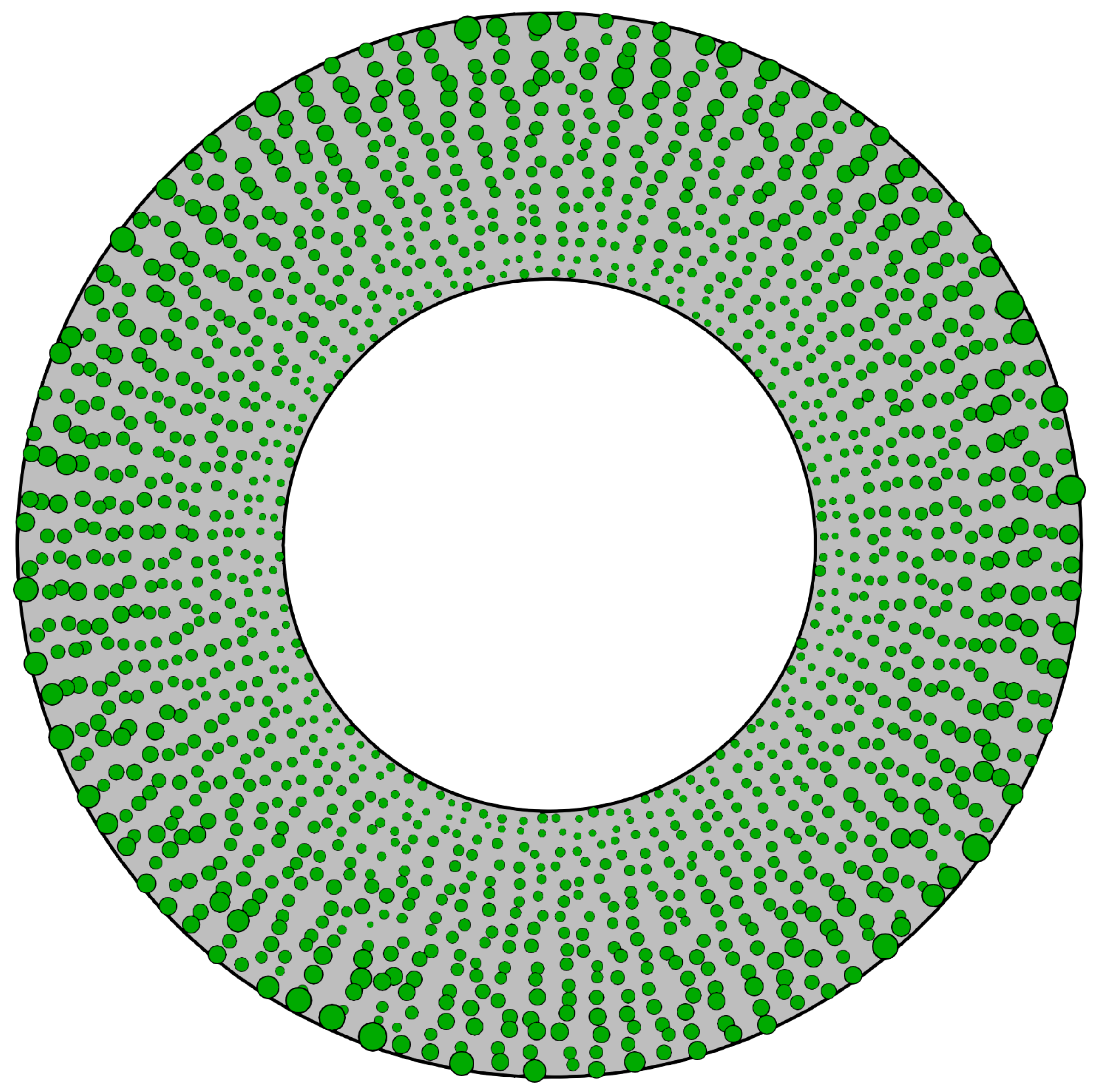}
    \caption{\small $p=3$ weights after} \label{fig:norm_after_uni_3}
  \end{subfigure}%
  \hfill%
  \begin{subfigure}[t]{0.32\textwidth}
    \centering
    \includegraphics[width=0.9\textwidth]{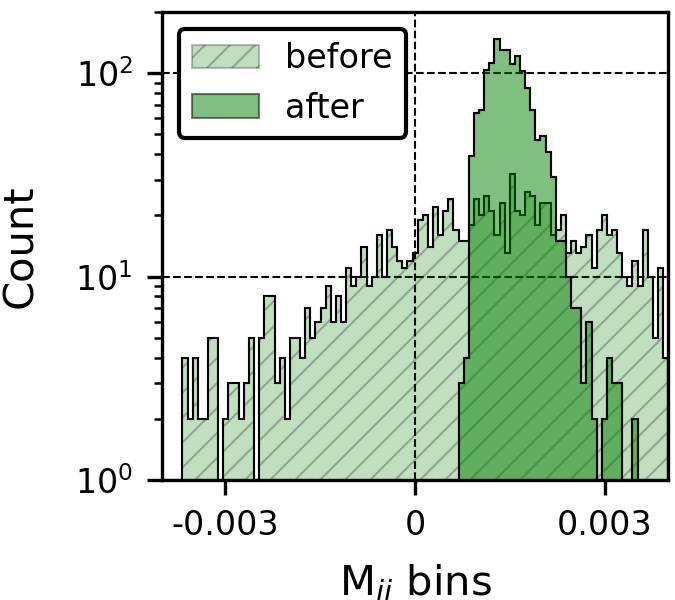}
    \caption{\small $p=3$ distribution} \label{fig:norm_dist_uni_3}
  \end{subfigure}\\%
  \begin{subfigure}[t]{0.32\textwidth}
    \centering
    \includegraphics[width=0.85\textwidth]{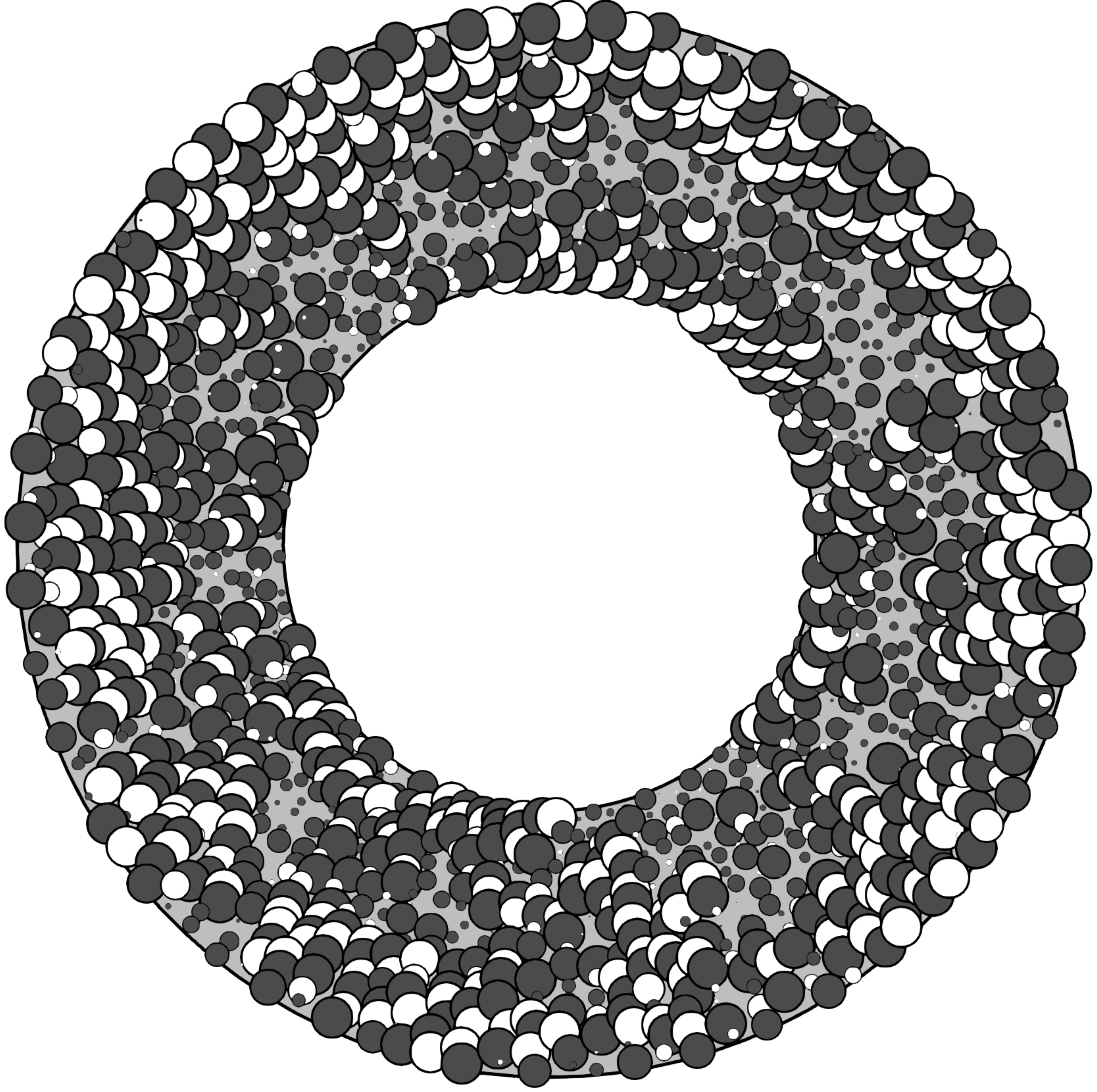}
    \caption{\small $p=4$ weights before} \label{fig:norm_before_uni_4}
  \end{subfigure}%
  \hfill%
  \begin{subfigure}[t]{0.32\textwidth}
    \centering
    \includegraphics[width=0.85\textwidth]{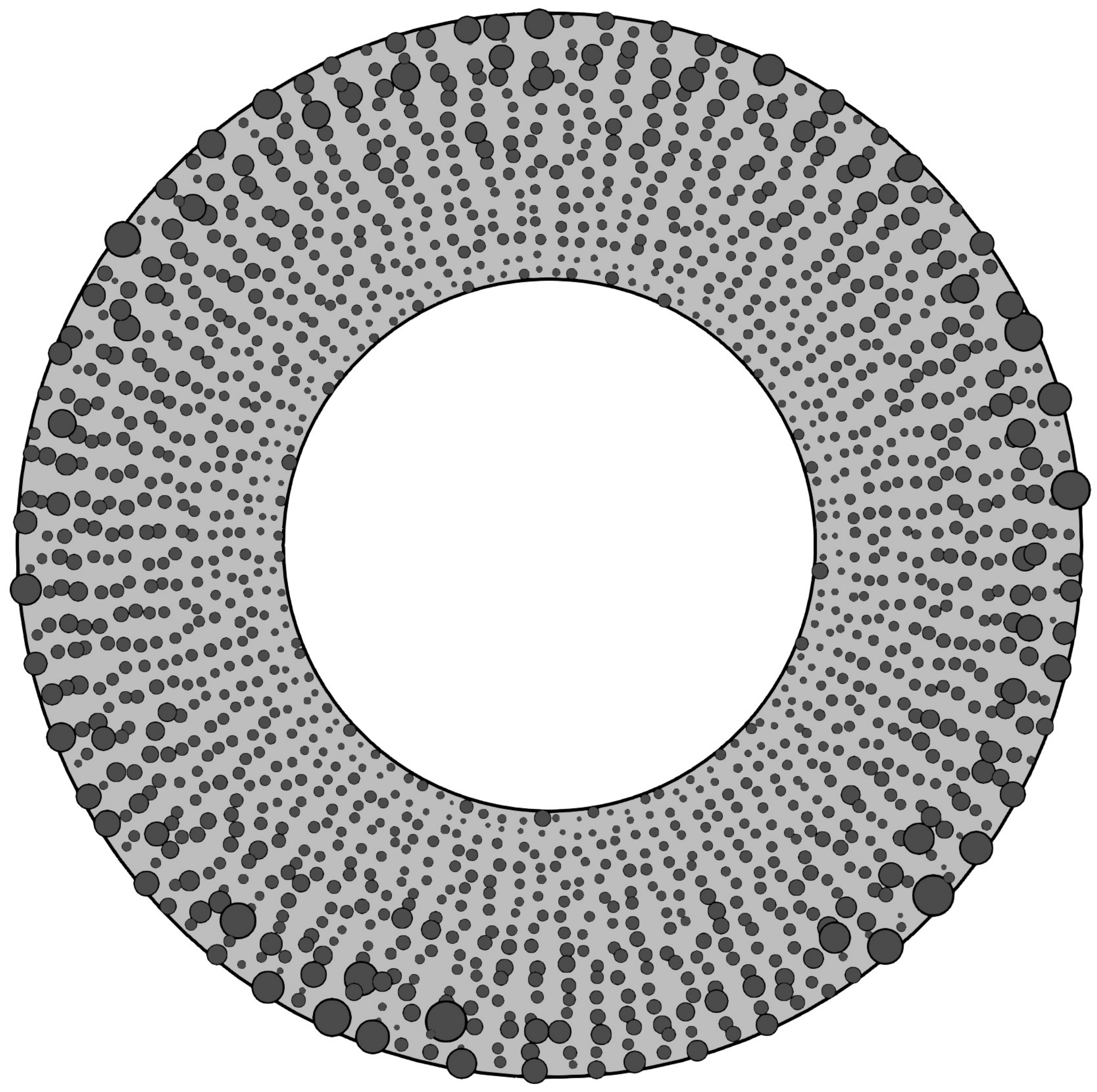}
    \caption{\small $p=4$ weights after} \label{fig:norm_after_uni_4}
  \end{subfigure}%
  \hfill%
  \begin{subfigure}[t]{0.32\textwidth}
    \centering
    \includegraphics[width=0.9\textwidth]{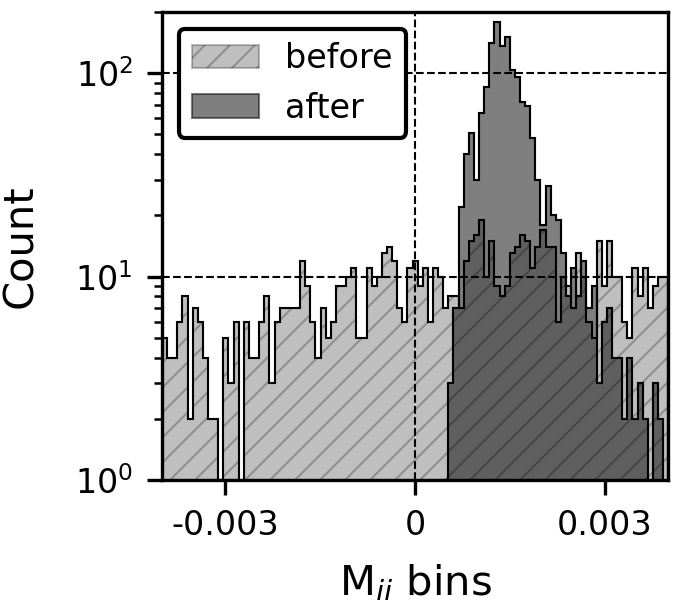}
    \caption{\small $p=4$ distribution} \label{fig:norm_dist_uni_4}
  \end{subfigure}%
  \caption{\small Quadrature-weight study for the annulus with \emph{quasi-uniform} nodes.  The left and center columns show the weights before and after solving the linear inequality, respectively.  The histograms in the right column show the distribution of the entries.} \label{fig:uniform_norm_dist}
\end{figure}

We repeat the study using $\beta=4$, which clusters the nodes toward the inner radius.  The corresponding results are shown in Figure~\ref{fig:nonuniform_norm_dist}.  Comparing with the quasi-uniform results, clustering the nodes increases the number of negative entries in $\bm{m}_{\min}$.  Nevertheless, the norm inequality was still successfully enforced in all cases.

For degrees $p > 1$, the weights after enforcing the norm inequality show some oscillation in magnitude toward the outer radius.  While this oscillation is obvious for the clustered nodes, it can also be observed in the quasi-uniform results.  We note that ``oscillations'' are also seen in the norm of classical SBP operators and Newton-Cotes quadrature formulae.

\begin{figure}[p]
  \begin{subfigure}[t]{0.32\textwidth}
    \centering
    \includegraphics[width=0.85\textwidth]{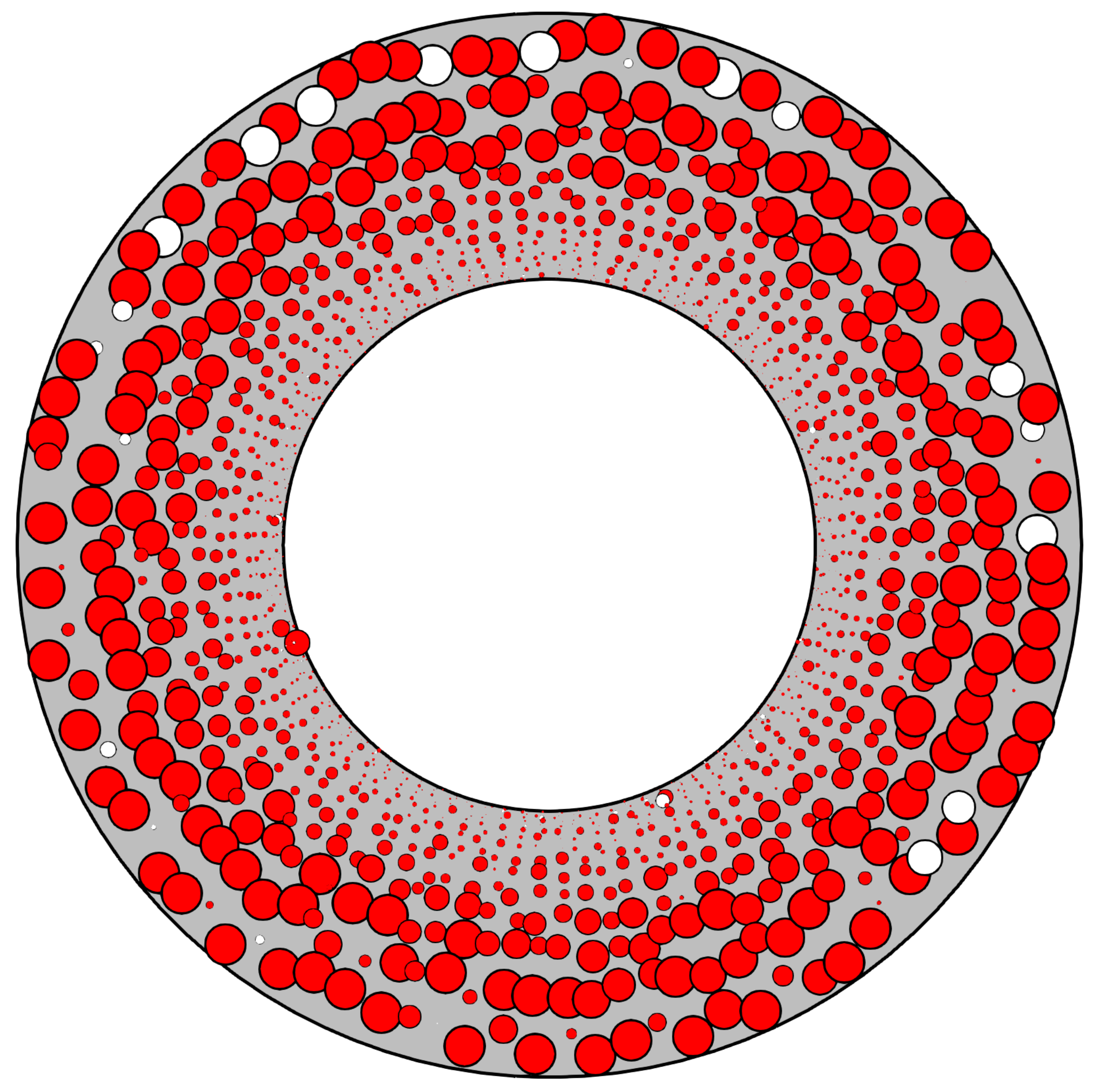}
    \caption{\small $p=1$ weights before} \label{fig:norm_before _nonuni_1}
  \end{subfigure}%
  \hfill%
  \begin{subfigure}[t]{0.32\textwidth}
    \centering
    \includegraphics[width=0.85\textwidth]{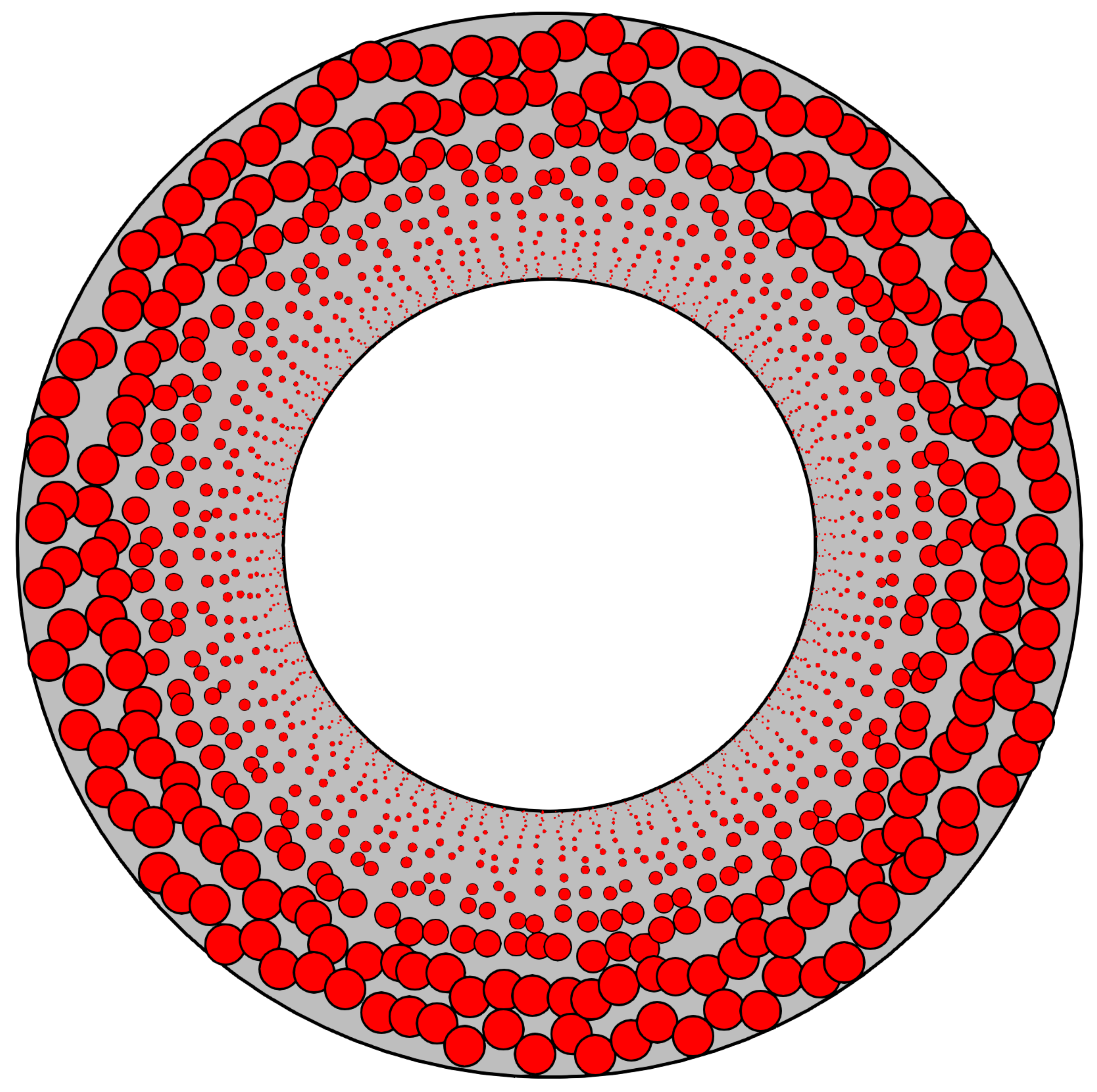}
    \caption{\small $p=1$ weights after} \label{fig:norm_after _nonuni_1}
  \end{subfigure}%
  \hfill%
  \begin{subfigure}[t]{0.32\textwidth}
    \centering
    \includegraphics[width=0.9\textwidth]{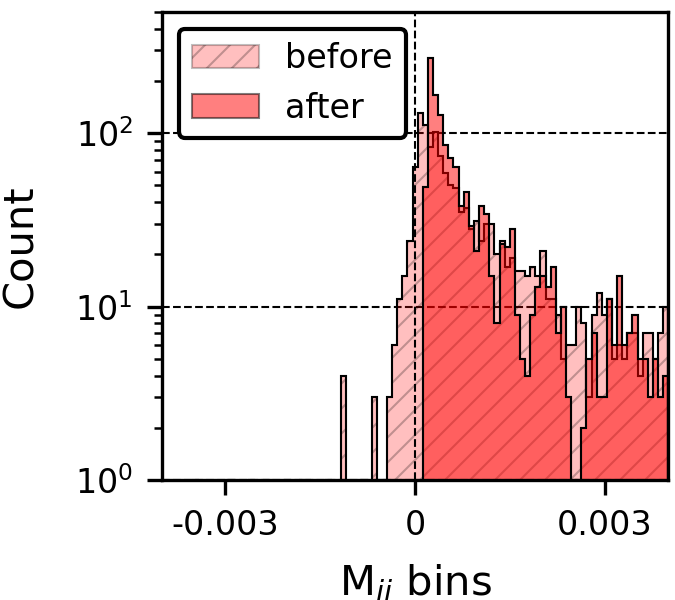}
    \caption{\small $p=1$ distribution} \label{fig:norm_dist _nonuni_1}
  \end{subfigure}\\%
  \begin{subfigure}[t]{0.32\textwidth}
    \centering
    \includegraphics[width=0.85\textwidth]{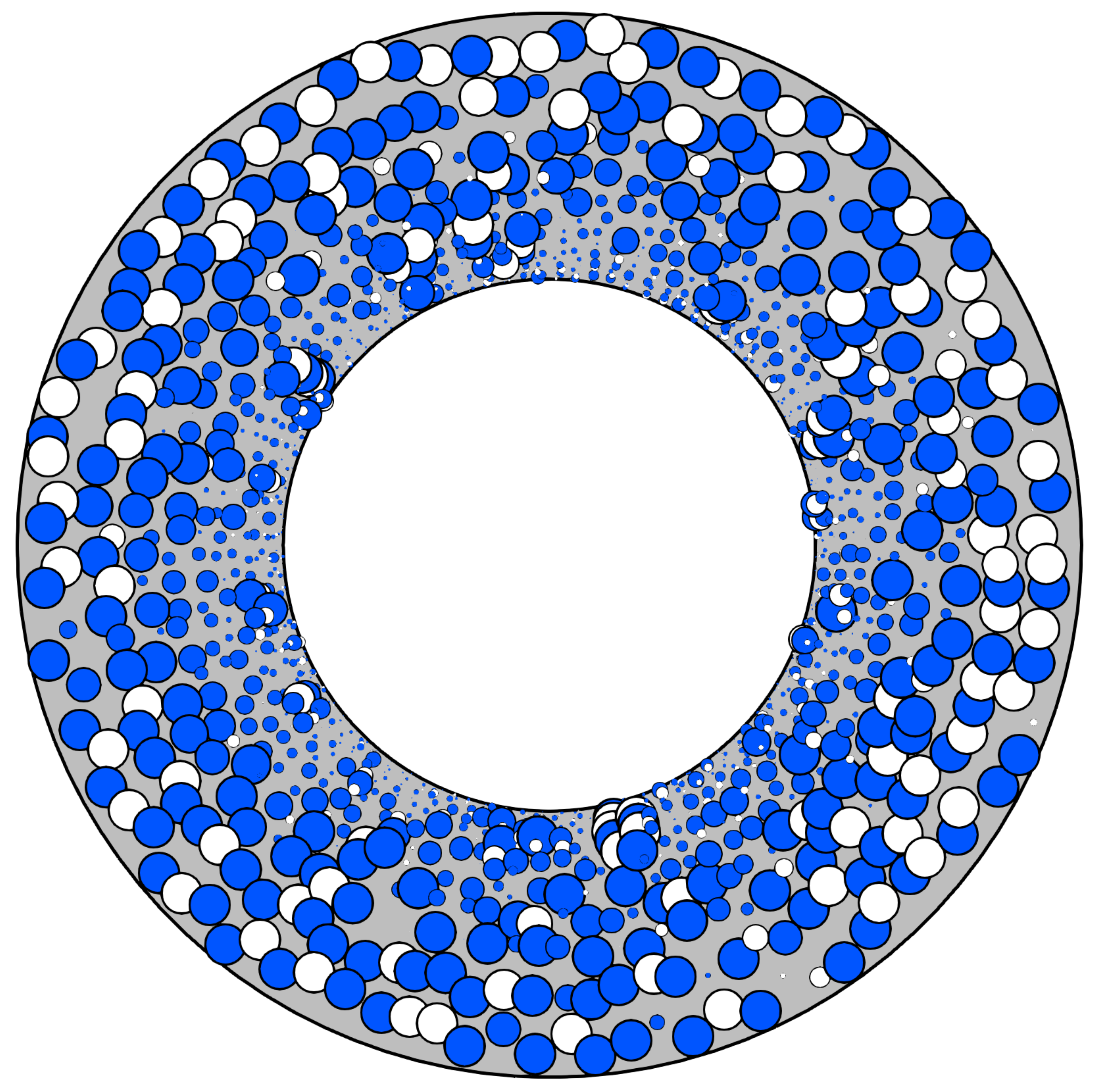}
    \caption{\small $p=2$ weights before} \label{fig:norm_before _nonuni_2}
  \end{subfigure}%
  \hfill%
  \begin{subfigure}[t]{0.32\textwidth}
    \centering
    \includegraphics[width=0.85\textwidth]{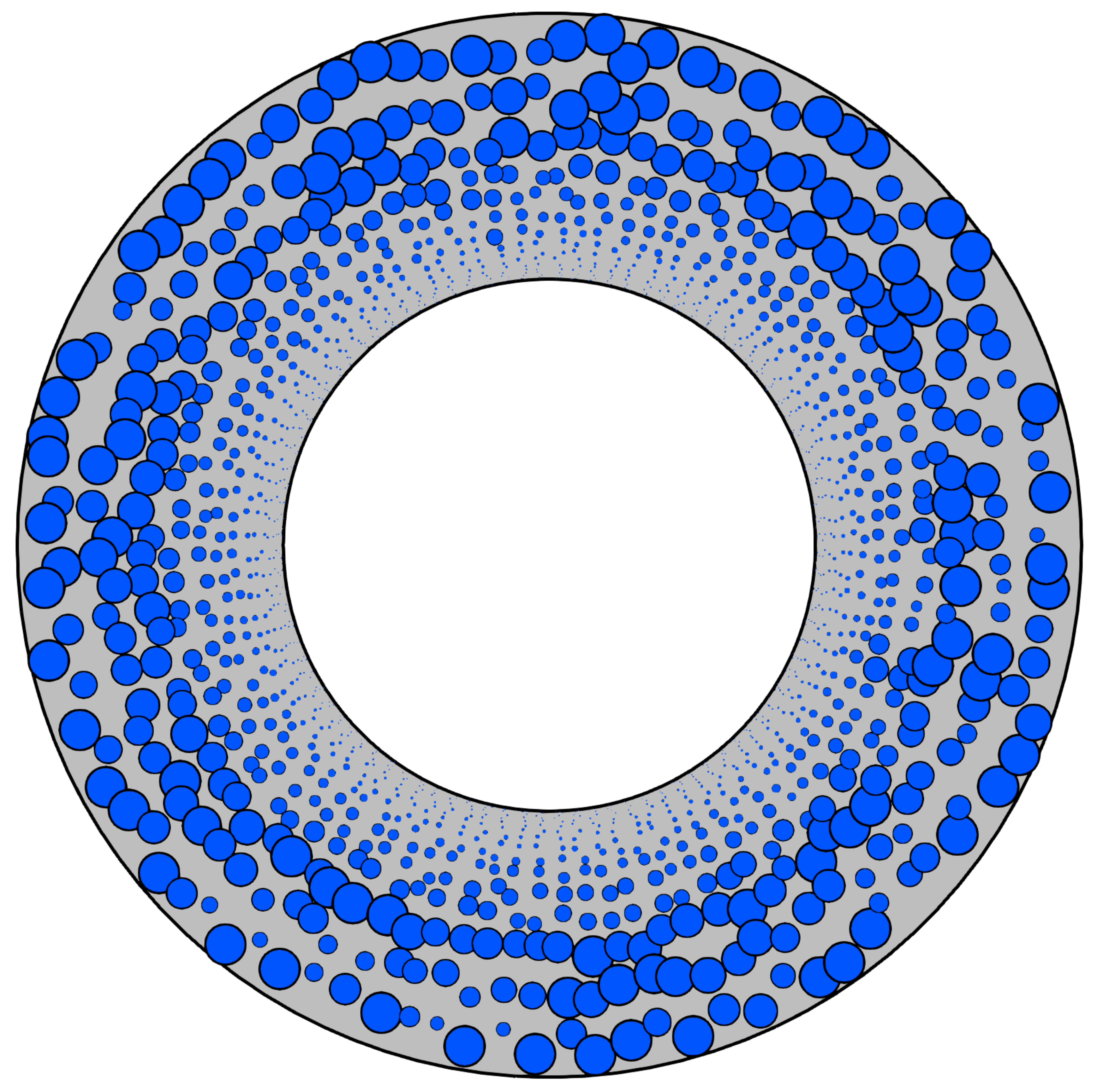}
    \caption{\small $p=2$ weights after} \label{fig:norm_after _nonuni_2}
  \end{subfigure}%
  \hfill%
  \begin{subfigure}[t]{0.32\textwidth}
    \centering
    \includegraphics[width=0.9\textwidth]{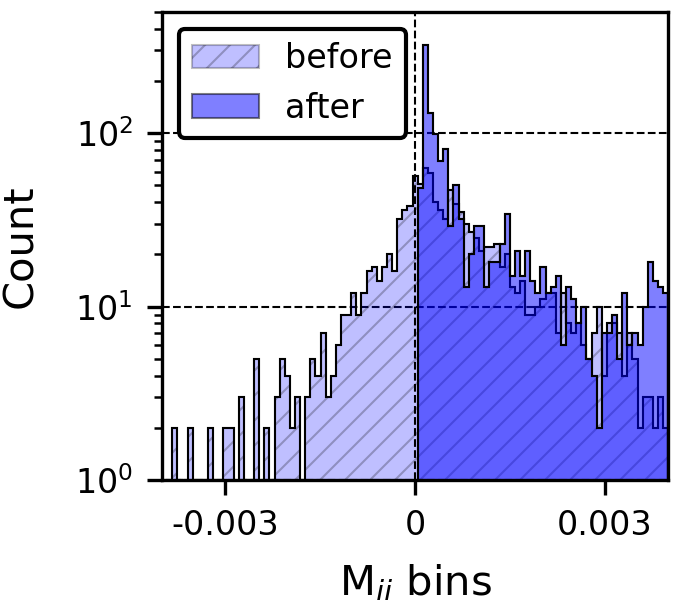}
    \caption{\small $p=2$ distribution} \label{fig:norm_dist _nonuni_2}
  \end{subfigure}\\%
  \begin{subfigure}[t]{0.32\textwidth}
    \centering
    \includegraphics[width=0.85\textwidth]{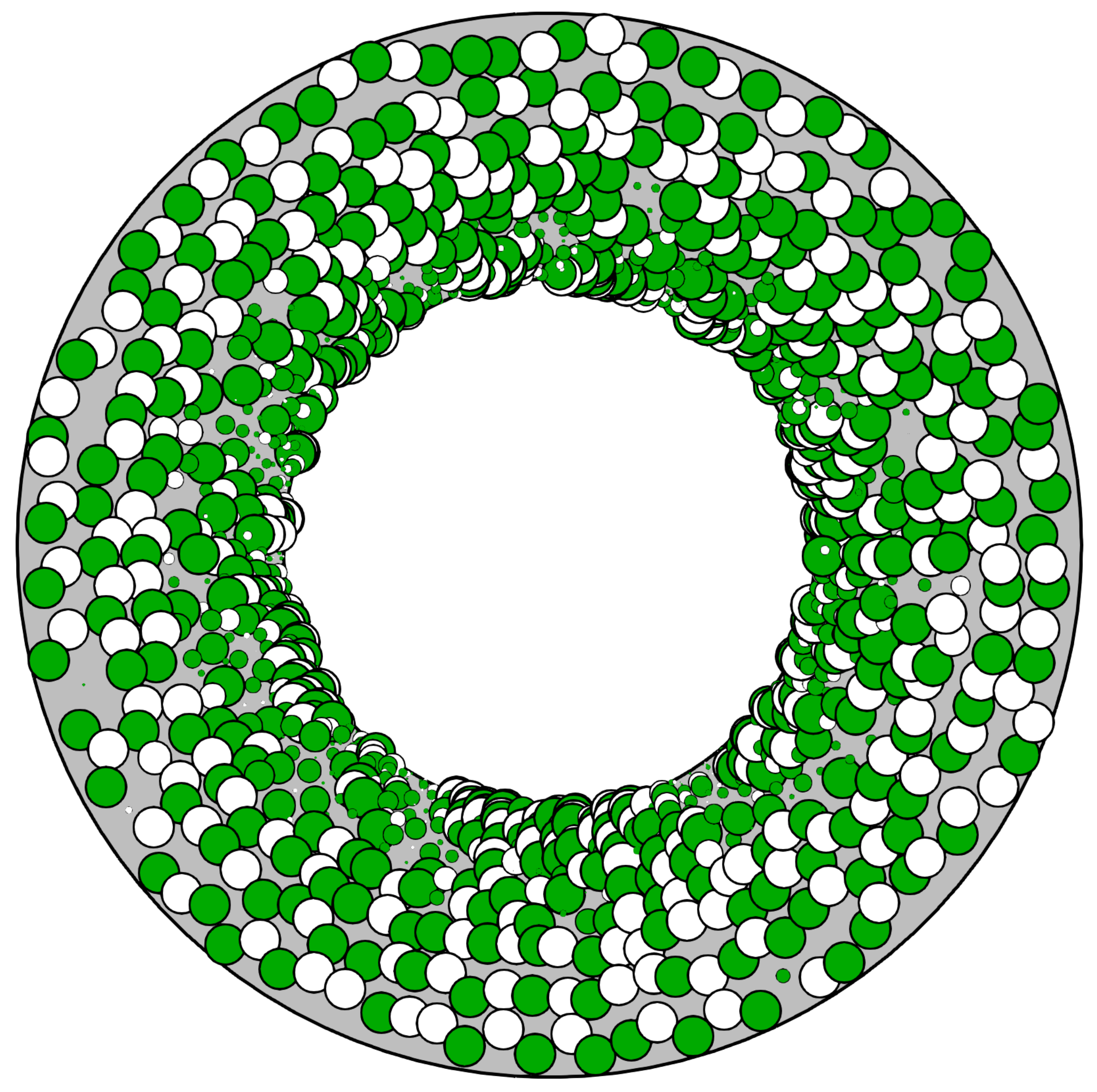}
    \caption{\small $p=3$ weights before} \label{fig:norm_before _nonuni_3}
  \end{subfigure}%
  \hfill%
  \begin{subfigure}[t]{0.32\textwidth}
    \centering
    \includegraphics[width=0.85\textwidth]{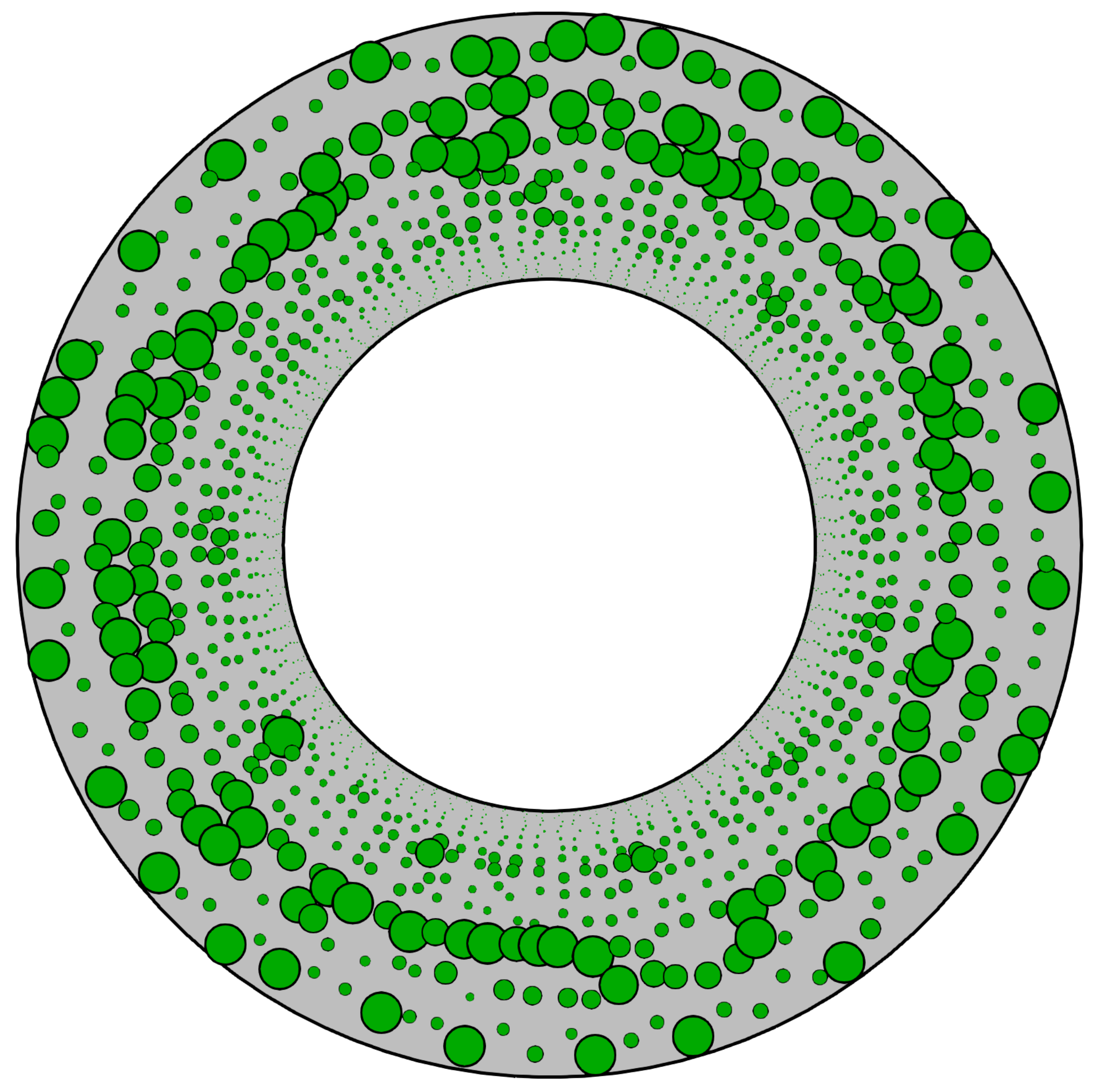}
    \caption{\small $p=3$ weights after} \label{fig:norm_after _nonuni_3}
  \end{subfigure}%
  \hfill%
  \begin{subfigure}[t]{0.32\textwidth}
    \centering
    \includegraphics[width=0.9\textwidth]{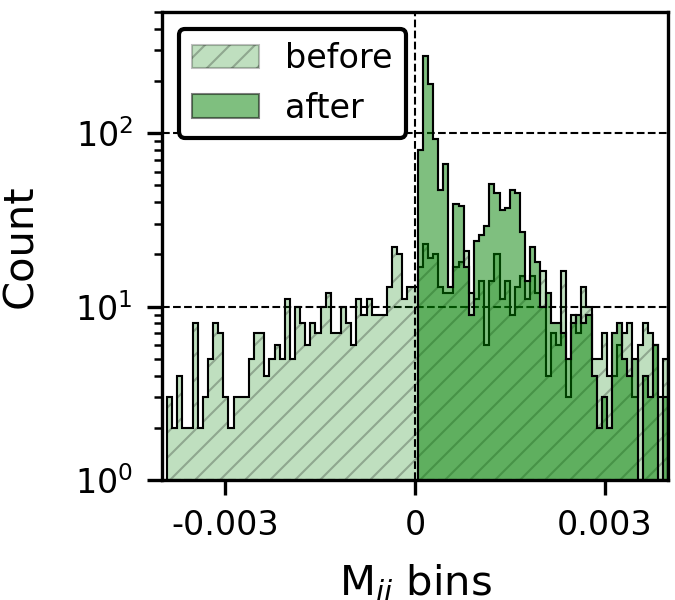}
    \caption{\small $p=3$ distribution} \label{fig:norm_dist _nonuni_3}
  \end{subfigure}\\%
  \begin{subfigure}[t]{0.32\textwidth}
    \centering
    \includegraphics[width=0.85\textwidth]{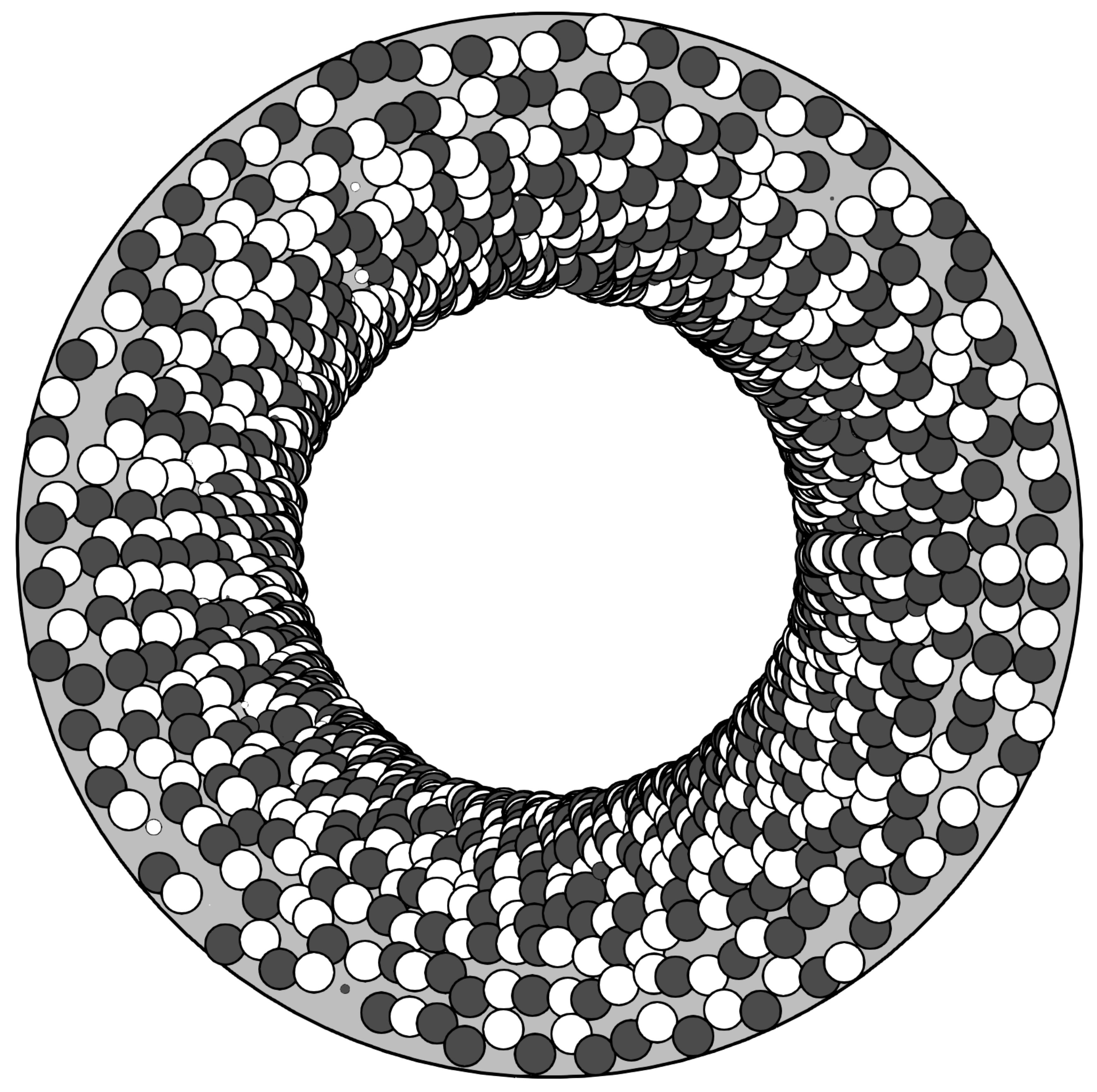}
    \caption{\small $p=4$ weights before} \label{fig:norm_before _nonuni_4}
  \end{subfigure}%
  \hfill%
  \begin{subfigure}[t]{0.32\textwidth}
    \centering
    \includegraphics[width=0.85\textwidth]{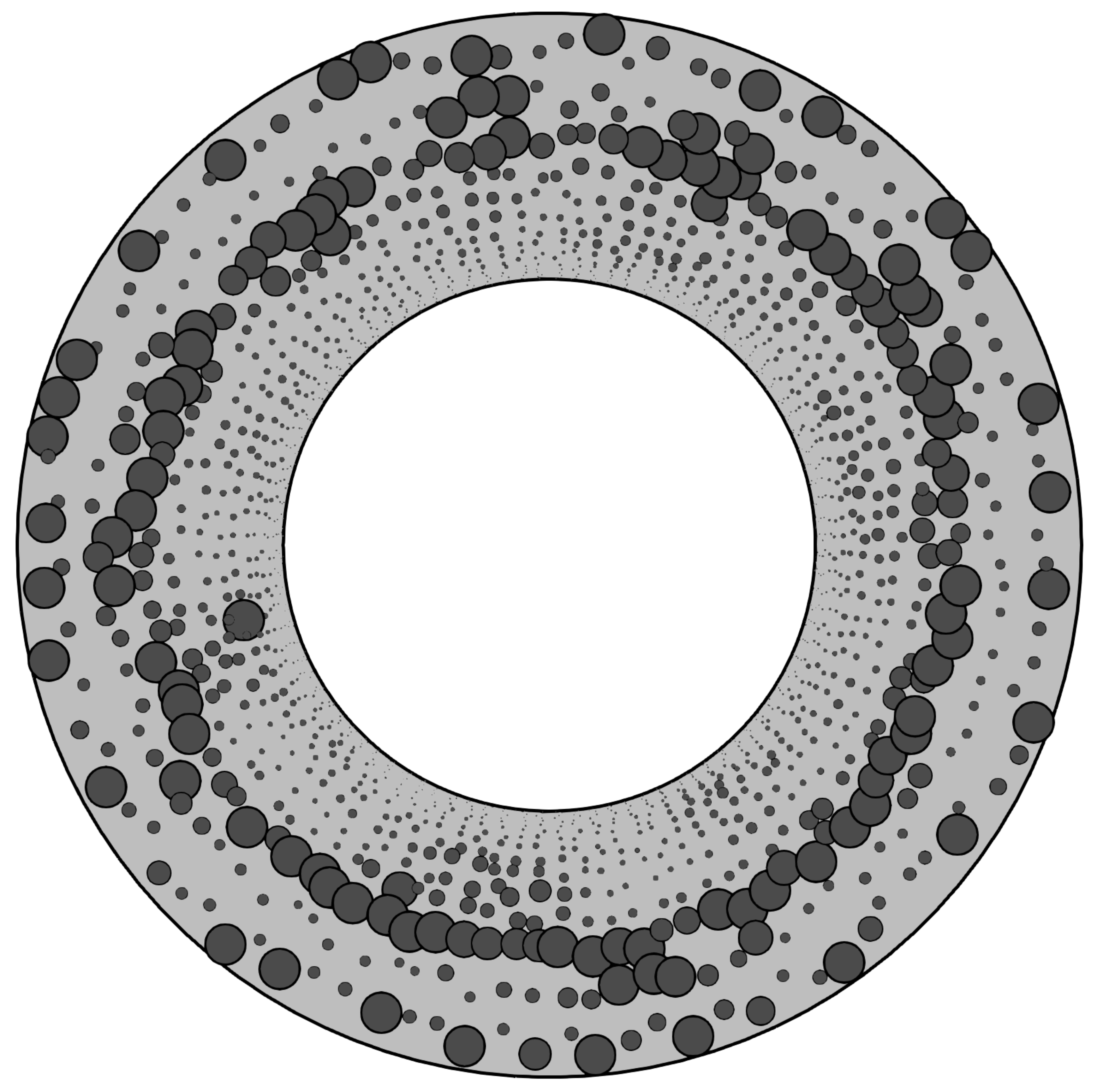}
    \caption{\small $p=4$ weights after} \label{fig:norm_after _nonuni_4}
  \end{subfigure}%
  \hfill%
  \begin{subfigure}[t]{0.32\textwidth}
    \centering
    \includegraphics[width=0.9\textwidth]{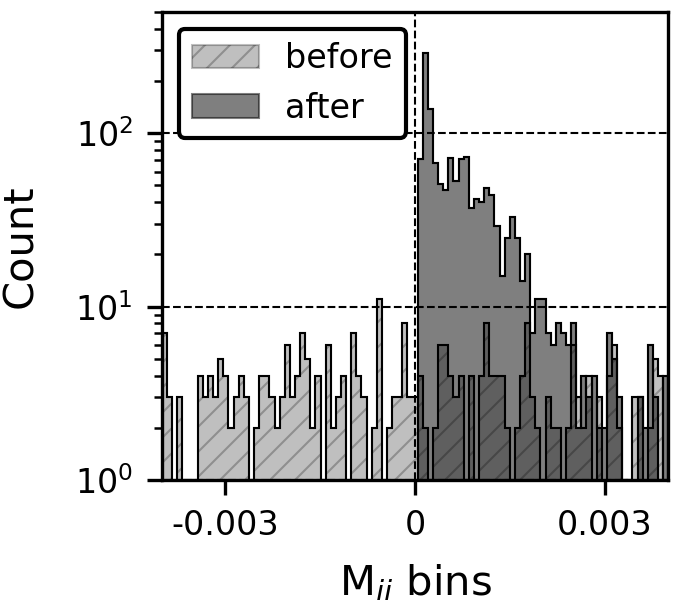}
    \caption{\small $p=4$ distribution} \label{fig:norm_dist _nonuni_4}
  \end{subfigure}%
  \caption{\small Quadrature-weight study for the annulus with \emph{non-uniform} nodes.  The left and center columns show the weights before and after solving the linear inequality, respectively.  The histograms in the right column show the distribution of the entries.} \label{fig:nonuniform_norm_dist}
\end{figure}

\subsubsection{Solution success rate}\label{sec:success_rate}

Recall that Theorem \ref{thm:feasibility} does not guarantee that the norm inequality has a feasible solution, in general.  It only implies that, for equidistributed $X_N$, there will be a feasible solution for sufficiently large $N$.  Thus, the study presented in this section seeks to characterize what ``sufficiently large'' means in practice.  The study also investigates the role that the tolerances $\tau_i$ play in determining feasibility.

To understand when the norm inequality fails to have a solution, we generate 1000 different geometries based on conic sections.  Specifically, for each sample $k$, the level-set is defined as 
\begin{equation*}
  \phi_k(\bm{x}) = 1 - \zeta_k \frac{x^2}{\xi_k} - \frac{y^2}{\eta_k},
\end{equation*}
where the realizations $\xi_k$ and $\eta_k$ are drawn from the uniform distribution $U[0.01, 0.99]$.  The parameter $\zeta_k \in \{-1,1\}$ selects elliptic ($\zeta_k = 1$) or hyperbolic ($\zeta_k = -1$) curves with equal probability.

The node generation process for each level-set $\phi_k(\bm{x})$ is similar to the one described in Section~\ref{sec:geometries} for $\Omega_{\text{box}}$ and $\Omega_{\text{foil}}$.  We first create a quasi-uniform node distribution over the box $[-1,1]^2$ with $n_x = n_y$ nodes in each coordinate direction.  We then add perturbations $\xi_i$ and $\eta_i$ to $x_i$ and $y_i$, respectively, where $\xi_i, \eta_i \sim U[-\Delta x/4, \Delta x/4]$ and $\Delta x = 2/n_x$.  Finally, we remove nodes $i$ for which $\phi_k(\bm{x}_i) < 0$; if this produces fewer nodes than the minimum required by Algorithm~\ref{alg:stencil}, we add one to $n_x$ and repeat the process.  For each geometry we consider four node resolutions corresponding to $n_x \in \{8, 16, 32, 64\}$.

As mentioned above, we are also interested in how the tolerances $\tau_{i}, i=1,2,\dots,N$ impact feasibility of the norm inequality.  Intuitively, $\tau_i$ should be smaller than the average node ``volume'' $\Delta x^2 = 4/n_x^2$; however, other than this upper bound on the $\tau_i$, the relationship between the tolerances and feasibility is not immediately obvious.  To help elucidate this relationship, we consider three values for $\tau_i$:  
\begin{equation*}
  \tau^{\text{large}} = \frac{1}{(n_x)^2},\qquad 
  \tau^{\text{small}} = \frac{1}{(10 n_x)^2},\qquad\text{and}\qquad
  \tau^{\text{tiny}} = \frac{1}{(100 n_x)^2}.
\end{equation*}
The large tolerance is one quarter of the average node ``volume,'' while the small and tiny tolerances are 1/400 and 1/40000 of the average, respectively.

Figure~\ref{fig:success_study_geos} shows some conic geometries generated for the success-rate study.  For $p=4$ on the coarsest node resolution ($n_x=8$) and $\tau_i = \tau_i^{\text{small}}$, the norm inequality was successfully solved on geometries that are colored blue, and the inequality was not successfully solved for the red geometries.  We emphasize that \emph{the coloring does not reflect the success for different $n_x$ or $p$}.  The figure is included primarily to illustrate the types of geometries that are generated and that there is no particular pattern in the failures for $p=4$ on the coarsest resolution.

\begin{figure}[tbp]
\begin{center}
  \includegraphics*[width=\textwidth]{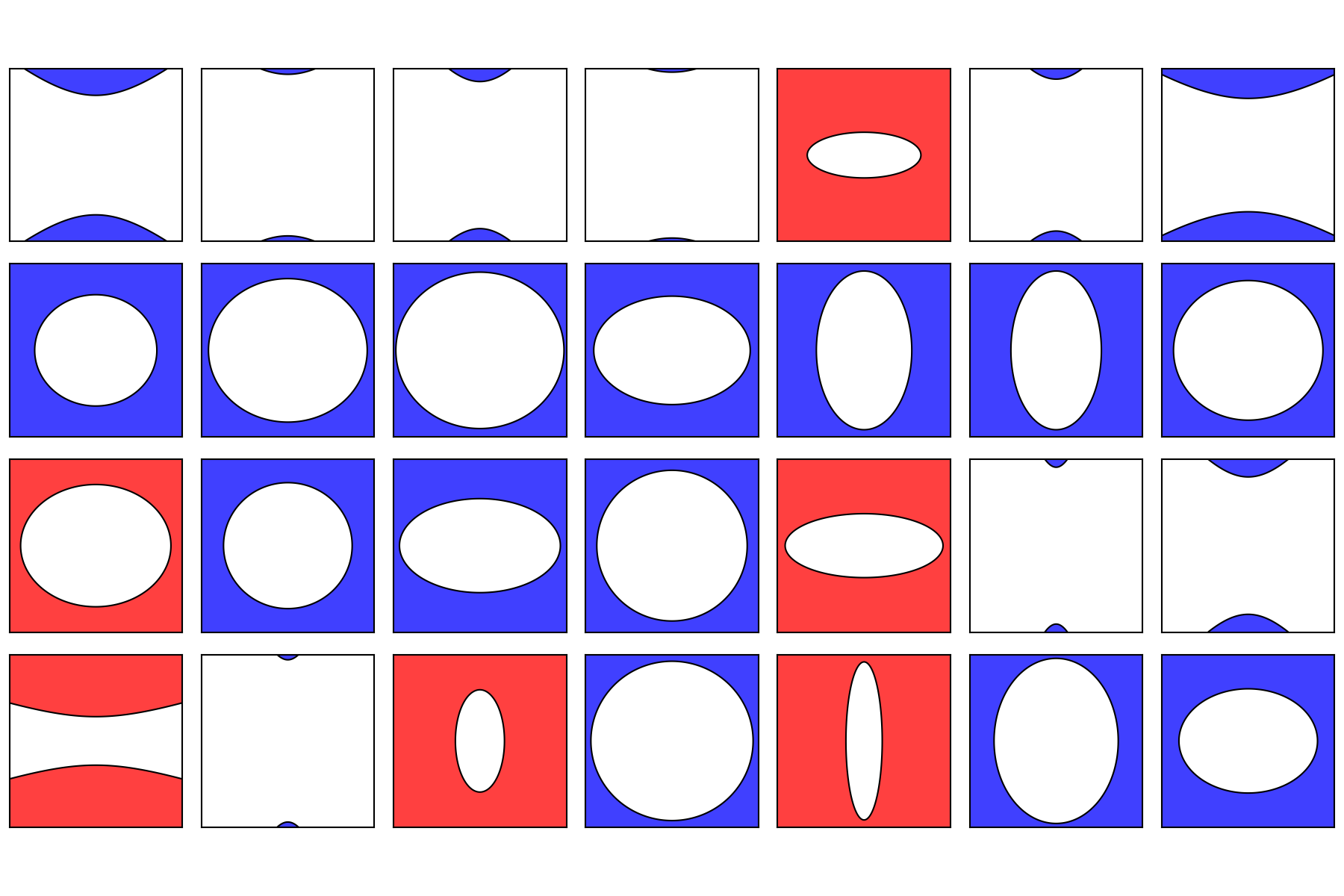}
  \caption[]{\small Example geometries used in the success-rate study; the white region indicates the domain $\Omega$.  Using $\tau_i = \tau^{\text{small}}$, the $p=4$ norm was positive-definite on the coarsest mesh for the blue geometries, and it was negative-definite on the coarsest mesh for red geometries.}\label{fig:success_study_geos}
\end{center}
\end{figure}

Table~\ref{tab:success_rate} lists the percentage of geometries for which the norm-inequality problem was successfully solved.  Each cell in the table is the success rate for a particular tolerance $\tau_i$, degree $p$, and resolution $n_x$ based on the 1000 sampled geometries.  Three trends emerge from the data in the table.
\begin{itemize}
\item For a fixed degree $p$ and node resolution $n_x$, decreasing the tolerance helps improve feasibility, up to a point.  In particular, we see a dramatic improvement in success rate in going from $\tau^{\text{large}}$ to $\tau^{\text{small}}$, but limited improvement in success rate when decreasing the tolerance from $\tau^{\text{small}}$ to $\tau^{\text{tiny}}$.
\item For a fixed tolerance and node resolution, the norm inequality becomes increasingly infeasible as $p$ grows.
\item For a fixed tolerance and degree $p$, the probability of solving the norm inequality increases with the node resolution.
\end{itemize}
The third trend is consistent with the conclusions of Theorem~\ref{thm:feasibility}.

This study considered only quasi-uniform nodes, and the success rate is likely to decrease for strongly anisotropic node distributions.  Nevertheless, regardless of the node distribution, this study and Theorem~\ref{thm:feasibility} suggest that adding nodes is a simple and effective remedy for an infeasible norm-inequality problem.  In addition, the study suggests that there exists a threshold for $\tau_i$ below which feasibility is relatively insensitive.

\begin{table}[tbp]
  \begin{center}
  \caption{Success rate generating a positive-definite norm across 1000 different geometries.}\label{tab:success_rate}
  \begin{tabular}{p{4cm}p{1cm}rrrr}
          &       & \multicolumn{4}{c}{\textbf{degree} $(p)$} \\\cline{3-6}
 $\tau_i$ & $n_x$ & 1  &   2   &   3   &  4  \\\hline
\rule{0ex}{3ex}%
$\tau_i^{\text{large}} = 1/(n_x)^2$
&  8  & 92.1\%  & 89.1\%  & 55.9\%  & 2.0\%  \\
& 16  & 99.9 \% & 99.9\%  & 89.9\%  & 71.9\% \\
& 32  & 100.0\% & 100.0\% & 99.5 \% & 96.2\% \\
& 64  & 100.0\% & 100.0\% & 100.0\% & 99.9\% \\\hline 
\rule{0ex}{3ex}%
$\tau_i^{\text{small}} = 1/(10 n_x)^2$
&  8  & 100.0\% & 100.0\% &  79.7\% &  9.6\% \\
& 16  & 100.0\% & 100.0\% &  92.9\% & 75.5\% \\
& 32  & 100.0\% & 100.0\% &  99.6\% & 97.2\% \\
& 64  & 100.0\% & 100.0\% & 100.0\% & 99.9\% \\\hline 
\rule{0ex}{3ex}%
$\tau_i^{\text{tiny}} = 1/(100 n_x)^2$
&  8  & 100.0\% & 100.0\% &  80.0\% &  9.7\% \\
& 16  & 100.0\% & 100.0\% &  92.9\% & 75.8\% \\
& 32  & 100.0\% & 100.0\% &  99.6\% & 97.2\% \\
& 64  & 100.0\% & 100.0\% & 100.0\% & 99.9\% \\\hline
  \end{tabular}
\end{center}
\end{table}

\subsubsection{Norm quadrature accuracy}\label{sec:quad_accuracy}

The quadrature defined by the nodes $X$ and weights $\bm{m}$ should numerically integrate smooth functions with an error $\text{O}(h^{2p})$, where $h$ is some measure of node resolution.  This rate of convergence is expected because i) the weights $\bm{m}^{\cell}$ integrate polynomials up to degree $2p-1$ exactly and ii) each cell $\cell$ uses the closest $N^{\cell}$ nodes to construct its stencil.  The goal of the study presented in this section is to verify the expected convergence rate of the quadrature.

We assess the accuracy of the quadrature using the box-circle, annulus, and airfoil domains.  For the box-circle geometry, we consider five node resolutions with $n_x \in \{10, 20, 40, 80, 160\}$.  For the annulus, we take $n_{\theta} = 6 n_r$ with $n_r \in \{ 6, 12, 24, 48 \}$.  We consider two node distributions for the annulus: quasi-uniform ($\beta=0.1$) and radially clustered toward the inner circle ($\beta=4$).  The airfoil geometry uses node resolutions defined by $n_y \in \{ 4, 8, 16, 32, 64 \}$; recall that $n_x = 5n_y$ for the airfoil.  We generate 10 distinct node samples for each geometry and resolution, to avoid drawing conclusions based on outliers.

We exclude samples from the plots below if the norm inequality fails to be feasible.  This was necessary for some samples from the annulus with the clustered node distribution and some samples from the airfoil geometry.  Specifically, three samples from the annulus and two samples from the airfoil geometry failed for $p=3$ on the coarsest distributions, and all ten samples for $p=4$ failed on the coarsest distributions for both problems.

The integral to be estimated on $\Omega_{\text{box}}$ is 
\begin{equation*}
  I_{\text{box}} = \int_{\Omega_{\text{box}}} \frac{\cos(2\theta)}{r} \, d\Omega = 0,
\end{equation*}
where $(r,\theta)$ are polar coordinates using the box center $\bm{x}_c$ as the origin.  The integral used on the annulus domain is 
\begin{equation*}
  I_{\text{ann}} = \int_{\Omega_{\text{ann}}} \frac{e^r}{r} \, d\Omega = 2 \pi \big(e - e^{\frac{1}{2}} \big),
\end{equation*}
and the integral used on the airfoil domain is 
\begin{equation*}
  I_{\text{foil}} = \int_{\Omega_{\text{foil}}} e^x \, d\Omega =
  \frac{3\,e}{4} - \frac{5}{8} \sqrt{\pi} \erfi(1),
\end{equation*}
where $\erfi(x)$ is the imaginary error function, which we evaluate using the \texttt{SpecialFunctions.jl} package. The numerical estimates for $I_{\text{box}}$, $I_{\text{ann}}$, and $I_{\text{foil}}$ are computed using the formula 
\begin{equation*}
  I_{N} = \bm{1}^T \mat{M} \bm{f} = \bm{m}^T \bm{f},
\end{equation*}
where $\bm{f} \in \mathbb{R}^{N}$ is the appropriate integrand evaluated at the nodes in $X$.

The results of the quadrature accuracy study are presented in Figure~\ref{fig:fun_error}.  The absolute functional error is plotted versus the nominal mesh spacing for the box-circle (Fig.~\ref{fig:fun_error_box}), the two node distributions for the annulus (Figs.~\ref{fig:fun_error_annulus_uniform} and \ref{fig:fun_error_annulus_nonuniform}), and the airfoil (Fig.~\ref{fig:fun_error_airfoil}).  The nominal mesh spacing, $h$, is based on the mean quadrature weight size:
\begin{equation}\label{eq:h_nominal}
  h = \sqrt{\frac{1}{N} \sum_{i=1}^N m_i}.
\end{equation}
The solid lines in the convergence plots indicate the expected asymptotic behavior for the functional errors, \ie, $\text{O}(h^{2p})$.  These reference lines are colored based on the corresponding $p$ and are anchored to the smallest error in each data set.

Overall, the quadrature errors appear to converge at the expected rate of $2p$ or greater.  The exceptions are the $p=4$ errors on the annulus with quasi-uniform nodes; the $p=1$ errors on the annulus with clustered nodes; and the $p=3$ and $p=4$ errors on the airfoil, where round-off errors are impacting the results on the finest node distributions. 

\begin{figure}[tbp]
  \begin{subfigure}[t]{0.49\textwidth}
    \centering
    \includegraphics[width=\textwidth]{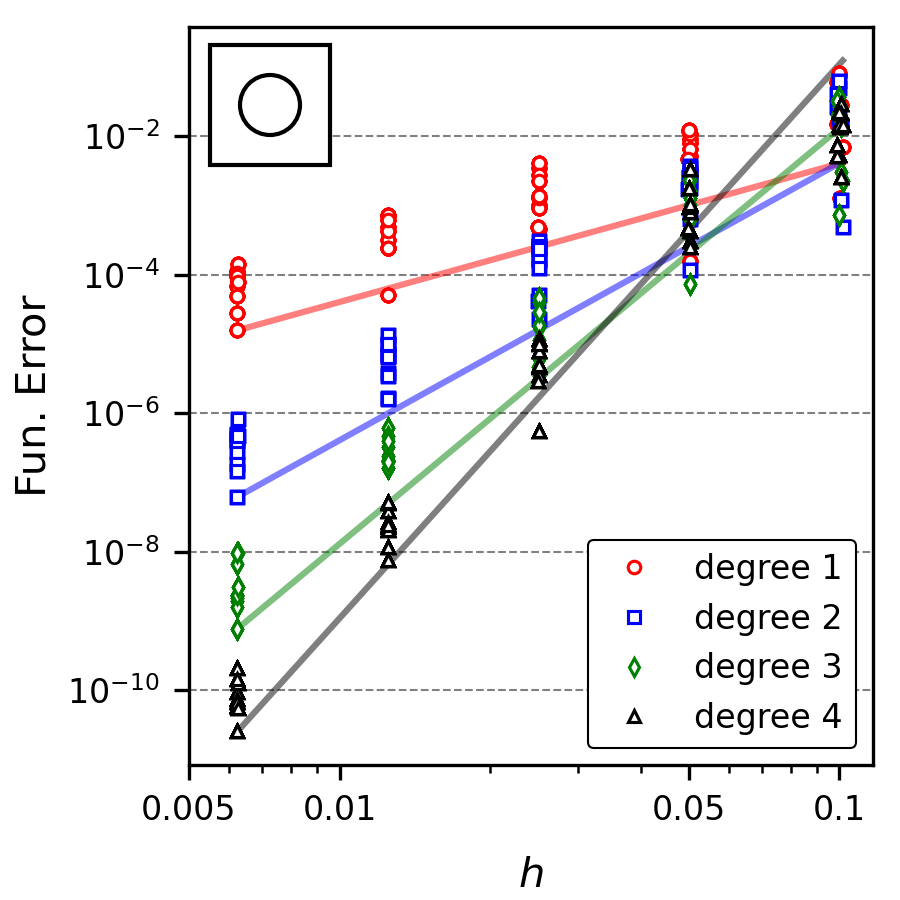}
    \caption{\small box-circle} \label{fig:fun_error_box}
  \end{subfigure}%
  \hfill%
  \begin{subfigure}[t]{0.49\textwidth}
    \centering
    \includegraphics[width=\textwidth]{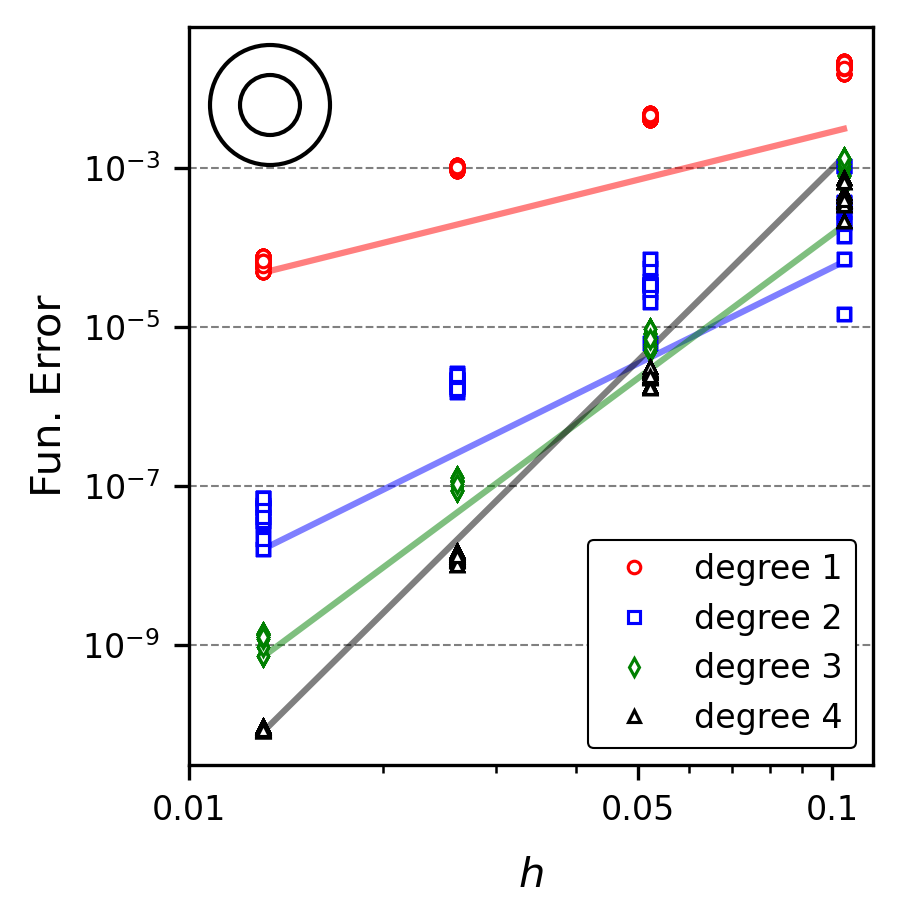}
    \caption{\small annulus, quasi-uniform} \label{fig:fun_error_annulus_uniform}
  \end{subfigure}\\[2ex]%
  \begin{subfigure}[t]{0.49\textwidth}
    \centering
    \includegraphics[width=\textwidth]{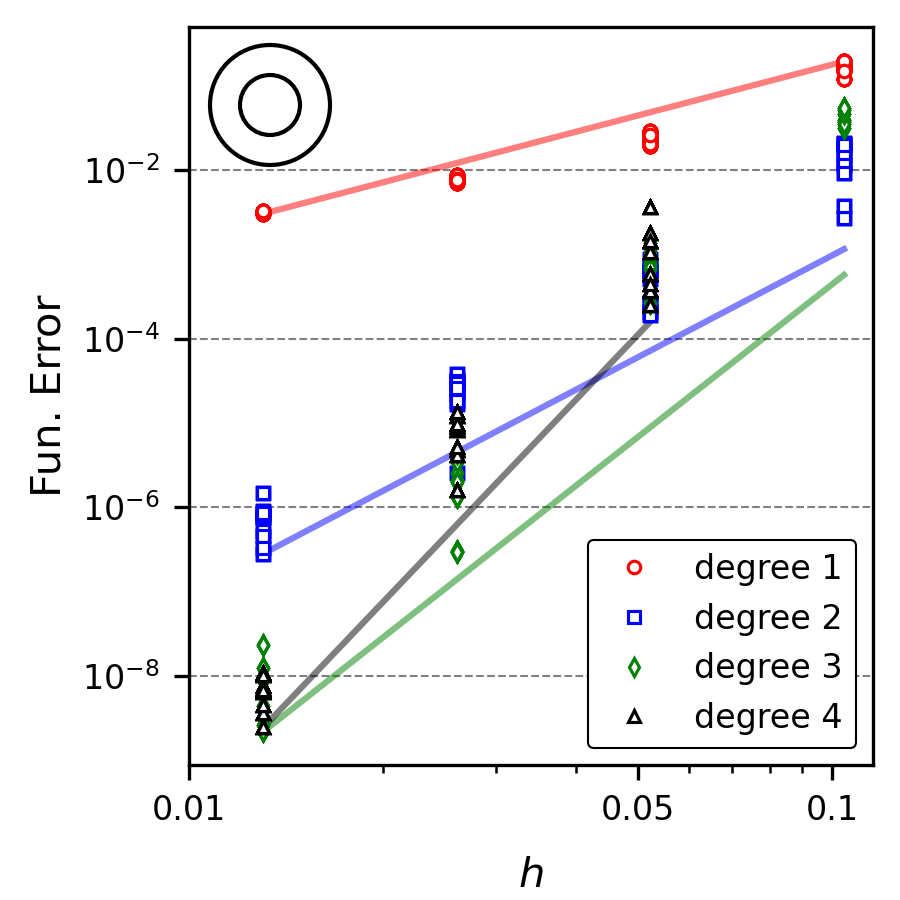}
    \caption{\small annulus, non-uniform} \label{fig:fun_error_annulus_nonuniform}
  \end{subfigure}%
  \hfill 
  \begin{subfigure}[t]{0.49\textwidth}
    \centering
    \includegraphics[width=\textwidth]{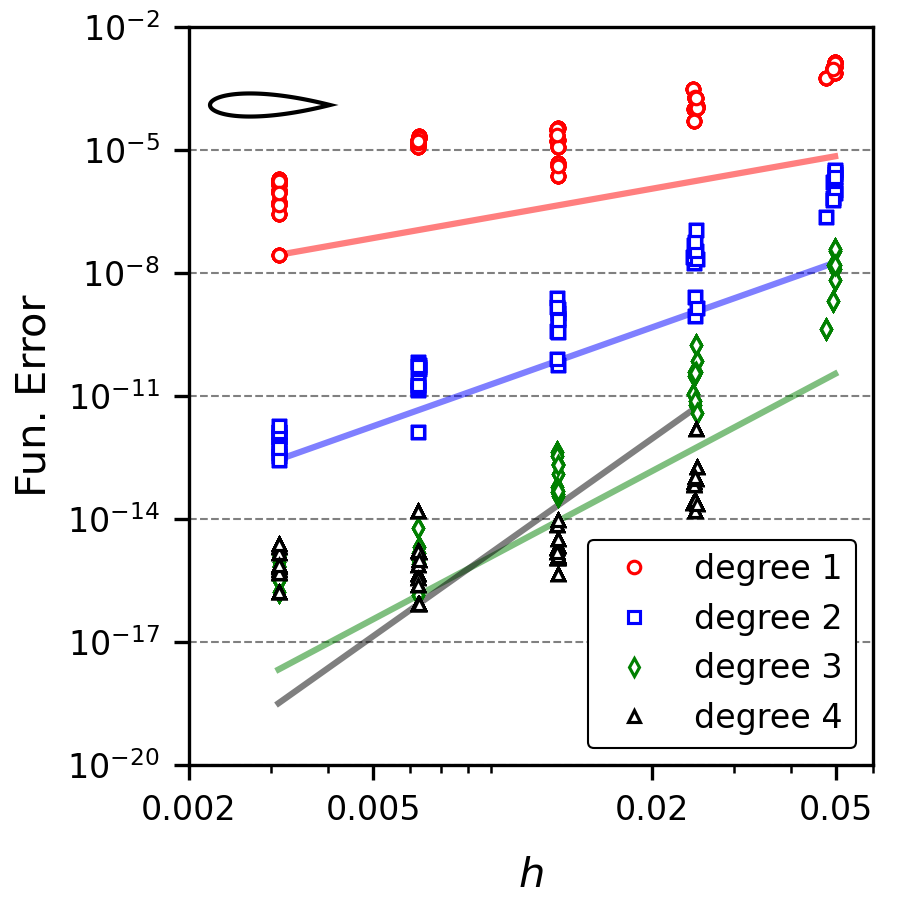}
    \caption{\small airfoil} \label{fig:fun_error_airfoil}
  \end{subfigure}
  \caption{\small Functional errors versus nominal mesh spacing for different node samples and geometries.} \label{fig:fun_error}
\end{figure}

\subsection{Steady advection accuracy study}\label{sec:sbp_accuracy_verify}

Having verified and characterized the diagonal mass matrix $\mat{M}$, we now turn our focus to the SBP operators themselves.  In this section, we use the steady advection equation to verify the solution accuracy of the SBP discretization~\eqref{eq:advection_sbp}.

The accuracy study uses constant-coefficient advection with $\bm{\lambda} = [1,1]^T$.  We adopt the method of manufactured solutions with the exact solution 
\begin{equation*}
  \fnc{U}(\bm{x}) = \exp(x + y),
\end{equation*}
and add the appropriate source term to the right-hand side of \eqref{eq:advection}.

We adopt the same geometries and node distributions as the ones used in the quadrature-accuracy study (Section~\ref{sec:quad_accuracy}), except that we only consider the four coarser node resolutions for the box-circle and airfoil geometries.  For each sample, the $L^2$ error for the degree $p$ discretization is approximated using the corresponding SBP quadrature:
\begin{equation}\label{eq:l2error}
  L^2\;\text{error} \approx \sqrt{ \big( \bm{u} - \bm{u}_{\text{exact}} \big)^T \mat{M} \big(\bm{u} - \bm{u}_{\text{exact}} \big)},
\end{equation}
where $\big[ \bm{u}_{\text{exact}}\big]_i = \fnc{U}(\bm{x}_i)$ for all $i\in \{1,2,\ldots,N\}$.  As before, the nominal mesh spacing is based on the mean quadrature weight size as defined by Equation~\eqref{eq:h_nominal}.

Figure~\ref{fig:error} summarizes the results of the SBP-accuracy study.  The reference lines in Figures~\ref{fig:error_box}, \ref{fig:error_annulus_uniform}, \ref{fig:error_annulus_nonuniform}, and \ref{fig:error_airfoil} show the design convergence rates of $p+1$, and the lines are anchored to the smallest error from the corresponding degree $p$ data set. Comparing these reference lines to the data at the finest and second-finest node resolutions, we conclude that the errors converge at the expected asymptotic rates for most of the geometries and distributions, with the following exceptions: the $p=4$ error on the annulus with quasi-uniform nodes appears to be closer to a rate of four rather than five; the $p=1$ airfoil discretization appears to be first-order only, and; the $p=4$ airfoil discretization is between third and fourth order.

\begin{figure}[tbp]
  \begin{subfigure}[t]{0.49\textwidth}
    \centering
    \includegraphics[width=\textwidth]{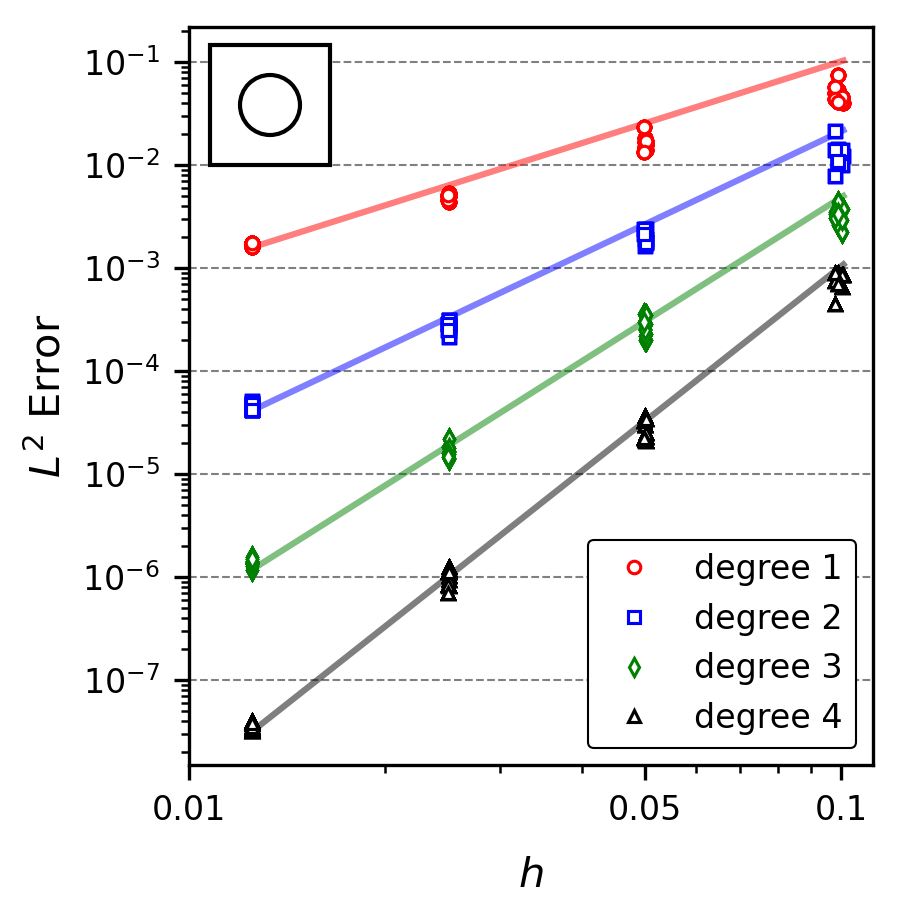}
    \caption{\small box-circle} \label{fig:error_box}
  \end{subfigure}%
  \hfill%
  \begin{subfigure}[t]{0.49\textwidth}
    \centering
    \includegraphics[width=\textwidth]{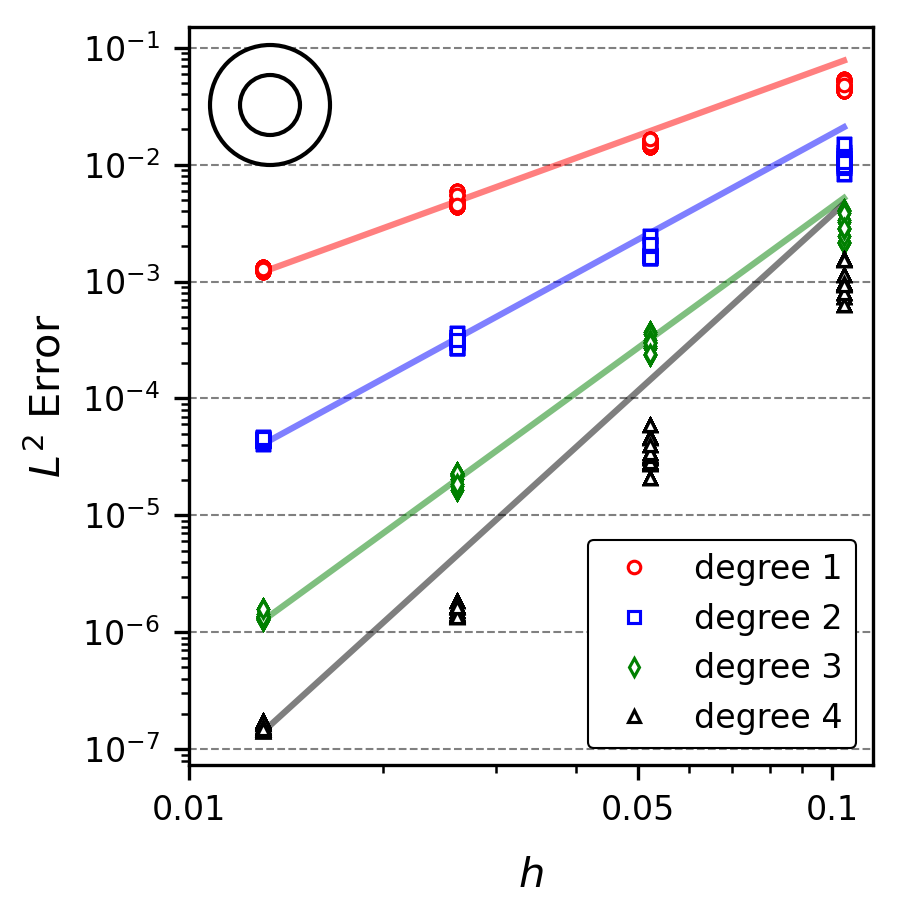}
    \caption{\small annulus, quasi-uniform} \label{fig:error_annulus_uniform}
  \end{subfigure}\\[2ex]%
  \begin{subfigure}[t]{0.49\textwidth}
    \centering
    \includegraphics[width=\textwidth]{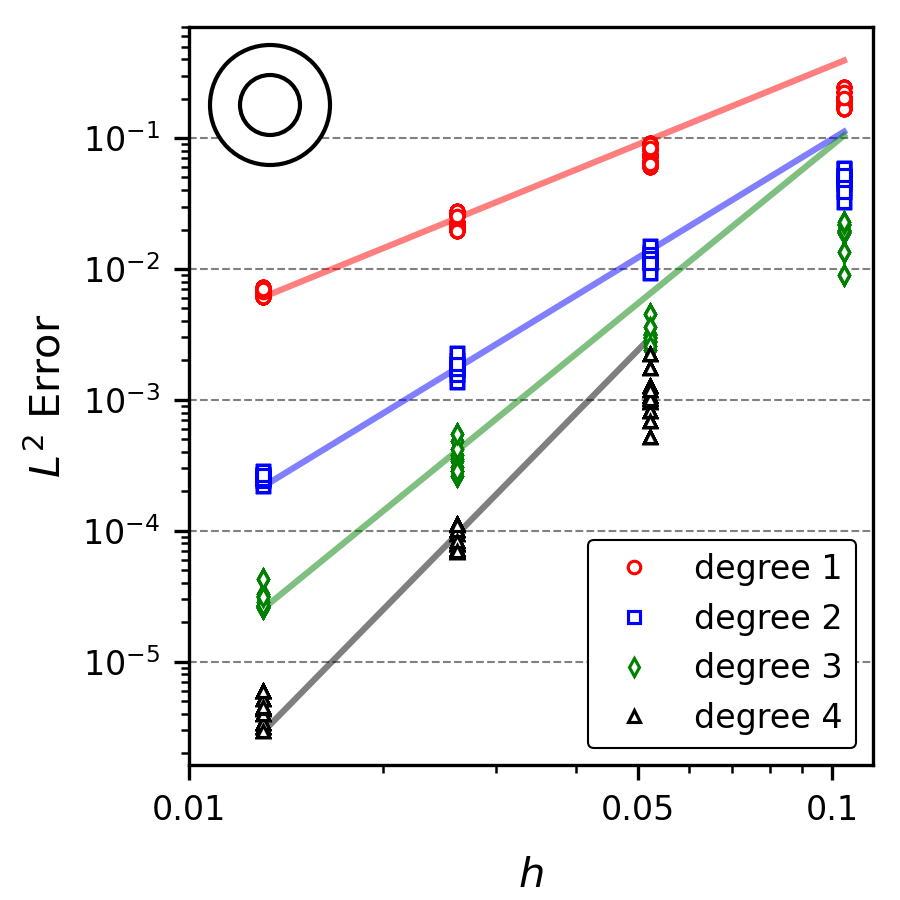}
    \caption{\small annulus, non-uniform} \label{fig:error_annulus_nonuniform}
  \end{subfigure}%
  \hfill%
  \begin{subfigure}[t]{0.49\textwidth}
    \centering
    \includegraphics[width=\textwidth]{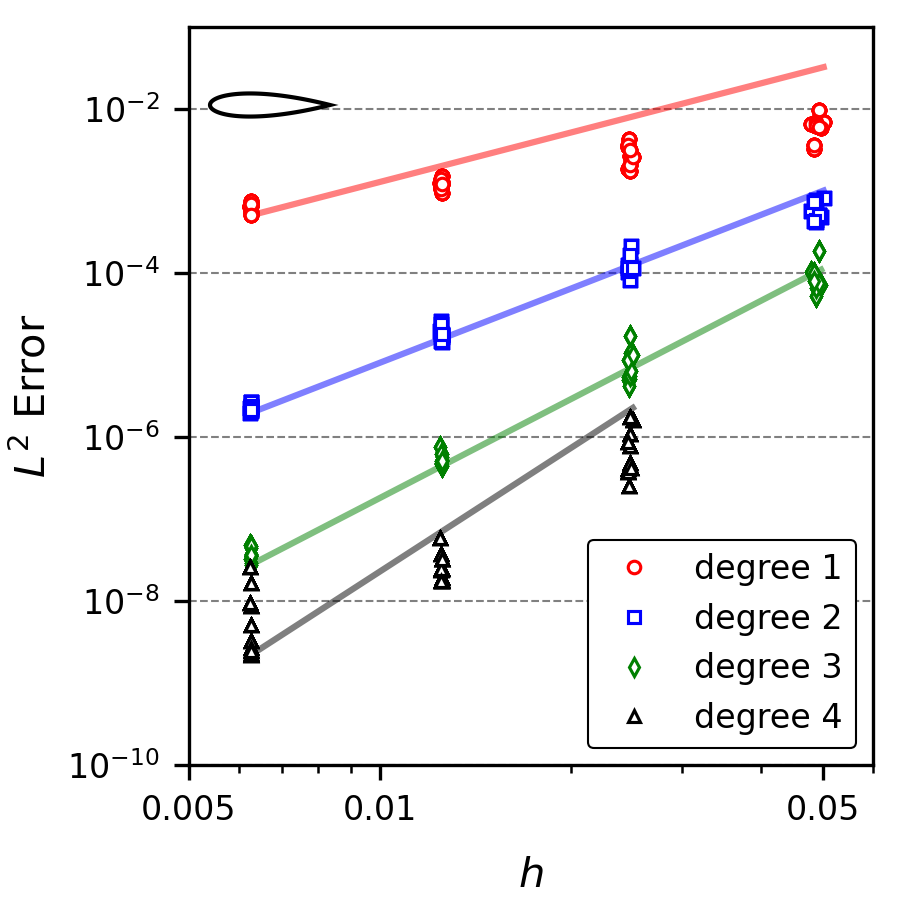}
    \caption{\small airfoil} \label{fig:error_airfoil}
  \end{subfigure}%
  \caption{\small $L^2$ errors versus nominal mesh spacing for different node samples and geometries from the steady advection accuracy study}. \label{fig:error}
\end{figure}

\subsection{Efficiency study}

Point-cloud SBP operators are less sparse than classical (tensor-product) SBP finite-difference operators, which has implications for the relative efficiency of the point-cloud discretizations.  To help quantify their efficiency, this section compares the performance of point-cloud discretizations with that of tensor-product SBP discretizations applied to the linear advection problem.

We adopt the annulus domain for the efficiency study, because it simplifies mesh generation for the classical SBP operators.  The tensor-product discretization uses a conforming grid with node coordinates given by $\bm{x}_{i} = [r_i \cos(\theta_i), r_i \sin(\theta_i)]^T$,  where
\begin{alignat*}{2}
  r_{i} &= \frac{1}{2(n_r -1)} \big(n_r - 1 + j - 1 \big), & 
  \qquad \forall\, j &\in \{1,2,\ldots,n_r\}, \\
  \theta_{i} &= \frac{2\pi(k-1)}{n_{\theta}}, &
  \qquad \forall\, k &\in \{1,2,\ldots,n_{\theta} \},
\end{alignat*}
and $i=k n_r + j$.  For the radial direction we take $n_r \in \{11, 21, 41, 81, 161\}$ and for the angular direction $n_{\theta} = 6 (n_r-1) + 1$.

The tensor-product discretization is standard.  In the radial direction we use the second- ($p=1$), third- ($p=2$), fourth- ($p=3$), and fifth- ($p=4$) order classical SBP operators from Mattsson and Nordstr\"{o}m \cite{Mattsson2004summation} as implemented in the \texttt{SummationByPartsOperators.jl} package~\cite{Ranocha2021SBP}.  Grids for which $n_r$ is too small for the given SBP operators are excluded from the study.  In the angular direction we use centered, periodic finite-difference operators of order $2p$; these periodic operators have the same finite-difference coefficients as the interior stencils of the radial SBP operators.  We use upwind numerical flux functions along the boundaries $r=1/2$ and $r=1$ to impose the boundary conditions.  Finally, we compute the terms in the mapping Jacobian using the appropriate finite-difference operator, following the typical tensor-product discretization in $(r,\theta)$ space; see, for example, \cite{Pulliam1986efficient}.

The nodes for the point-cloud SBP operators are generated as described in Section~\ref{sec:node_distribution}, and they use the quasi-uniform ($\beta=0.1$) distribution with $n_{\theta} = 6 n_r$ and $n_{r} \in \{6,12,24,48\}$.  We show results for only one node-set realization from Section~\ref{sec:sbp_accuracy_verify} to avoid cluttering the figures, but the results for the other nine realizations are similar.

For each discretization, we solve the linear advection problem based on the manufactured solution given in Section~\ref{sec:sbp_accuracy_verify} and compute the $L^2$ error using Equation~\eqref{eq:l2error}.  Figure~\ref{fig:efficiency} shows the $L^2$ errors plotted versus the number of nodes, $N$, as well as the number of non-zeros ($nnz$) in the system matrix.  We use $nnz$ as a proxy for computational cost, since the point-cloud implementation has not been optimized.

The first (left) plot in Figure~\ref{fig:efficiency} shows that, for this problem, the $L^2$ errors from the two discretizations overlap with one another when plotted versus number of nodes.  However, the second (right) plot shows that the classical SBP discretization is significantly more efficient with respect to $nnz$.  Indeed, for the same $L^2$ error and degree $p$, the classical SBP discretization requires between four ($p=1$) and ten ($p=4$) times fewer non-zeros in the system matrix.

\begin{figure}[tbp]
\begin{center}
  \includegraphics[width=\textwidth]{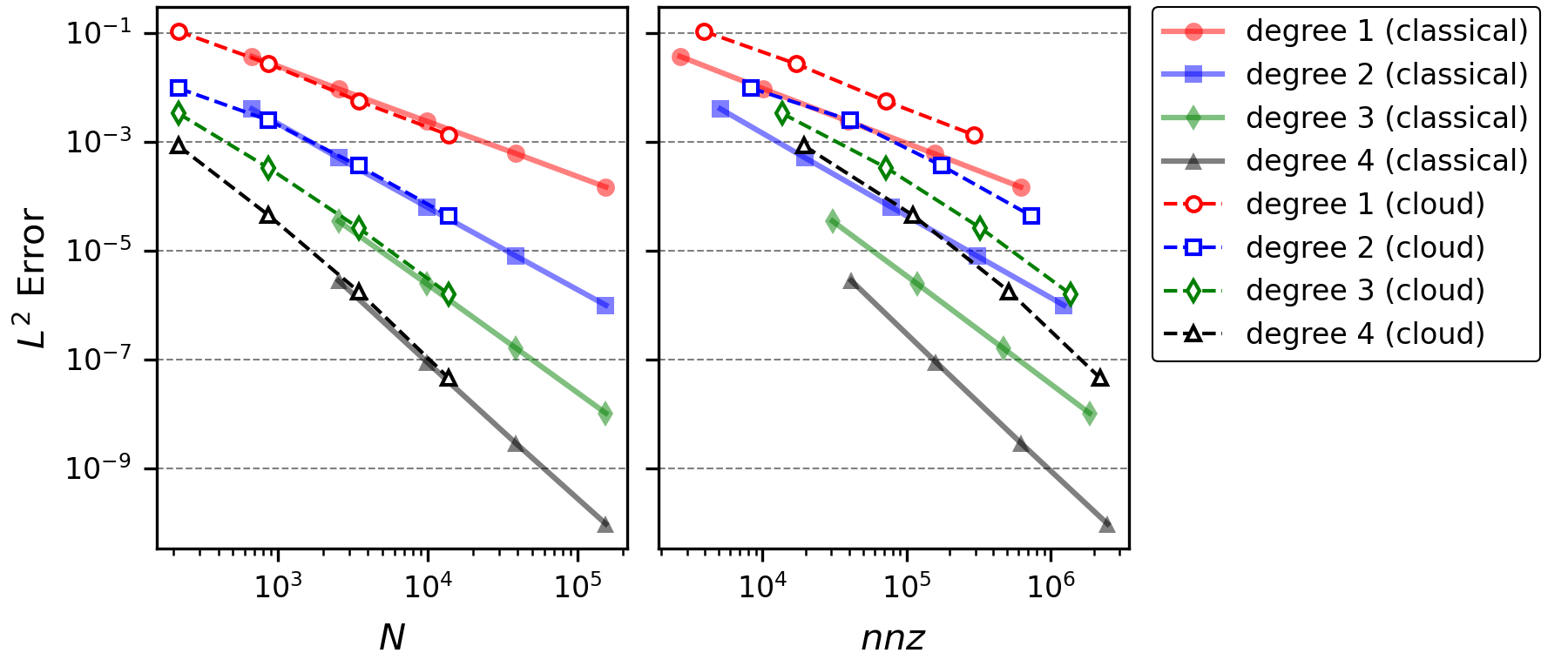}
  \caption{\small $L^2$ error versus number of nodes (left figure) and $L^2$ error versus the number of non-zeros in the advection system matrix (right figure).  Results are shown for both classical and point-cloud SBP discretizations of the advection problem on the annulus domain.} \label{fig:efficiency}
\end{center}
\end{figure}

\subsection{Energy stability study}\label{sec:energy_stab_verify}

This study is intended to verify the energy stability of the SBP discretization~\eqref{eq:advection_sbp}.  We consider the annulus geometry $\Omega_{\text{ann}}$ with the spatially varying advection field
\begin{equation*}
  \bm{\lambda} = [\lambda_x(\bm{x}), \lambda_y(\bm{x})]^T = \left[ \frac{-y}{2(x^2 + y^2)}, \frac{x}{2(x^2 + y^2)} \right]^T,
\end{equation*}
which corresponds to an irrotational vortex centered at the origin.  The initial condition is a Gaussian bump centered at $\bm{x}_{b} = [3/4,0]$:
\begin{equation*}
  \fnc{U}_0(\bm{x}) = \exp\big( - 4\|\bm{x} - \bm{x}_b\|^2 \big).
\end{equation*}

Boundary conditions are not required for this problem, because the velocity is everywhere parallel to the boundary.  Indeed, one can show that the SBP discretization \eqref{eq:advection_sbp} reduces to 
\begin{equation}\label{eq:advection_sbp_annulus}
  \mat{M} \frac{d \bm{u}}{dt} = \underbrace{-\frac{1}{2} \mat{\Lambda}_x \mat{Q}_x \bm{u} + \frac{1}{2} \mat{Q}_x^T \mat{\Lambda}_x \bm{u} - \frac{1}{2} \mat{\Lambda}_y \mat{Q}_y \bm{u} + \frac{1}{2} \mat{Q}_y^T \mat{\Lambda}_y \bm{u} -\mat{A} \bm{u}}_{\displaystyle \equiv \bm{r}}.
\end{equation}
This discretization is energy stable, since $\mat{M}$ is positive definite, $\mat{\Lambda}_x \mat{Q}_x - \mat{Q}_x^T \mat{\Lambda}_x$ and $\mat{\Lambda}_y \mat{Q}_y - \mat{Q}_y^T \mat{\Lambda}_y$ are skew symmetric, and $\mat{A}$ is positive semi-definite.  Left-multiplying the discretization by $\bm{u}^T$ we find the equation governing the rate-of-change of the solution energy:
\begin{equation}
  \frac{d}{dt} \big( {\textstyle \frac{1}{2}} \bm{u}^T \mat{M} \bm{u} \big) = \bm{u}^T \bm{r} = - \bm{u}^T \mat{A} \bm{u}.
\end{equation}
We consider two cases for the energy-stability results: one where the dissipation is absent --- so $\bm{u}^T \bm{r} = 0$ --- and one where $\mat{A}$ is present --- so $\bm{u}^T \bm{r} \leq 0$.

We use $\beta=1/10$ to generate a quasi-uniform node distribution for the annulus, with $n_r = 48$ and $n_{\theta} = 288$ nodes in the radial and angular directions, respectively.  We solve the ordinary differential equation~\eqref{eq:advection_sbp_annulus} for one period, $T = 2\pi$, using the classical fourth-order Runge Kutta method (RK4).  Since RK4 is conditionally stable, we limit the time step to $\Delta t = 2/\rho_p$, where $\rho_p$ is the spectral radius of the degree $p$ system matrix corresponding to~\eqref{eq:advection_sbp_annulus}; $\rho_p$ is computed using the Julia interface to the ARPACK library~\cite{Lehoucq1998arpack}.

Figure~\ref{fig:bump_ic} shows the initial condition, and Figure~\ref{fig:bump_final_p4} shows the numerical solution after one period using the $p=4$ discretization.  The initial condition and  solution are interpolated to the cells for these visualizations.  The advection velocity increases from the outer to the inner radii of the annulus, so, as time progresses, the Gaussian bump becomes twisted around the domain.

\begin{figure}[tbp]
  \begin{subfigure}[t]{0.3\textwidth}
    \centering
    \includegraphics[width=\textwidth]{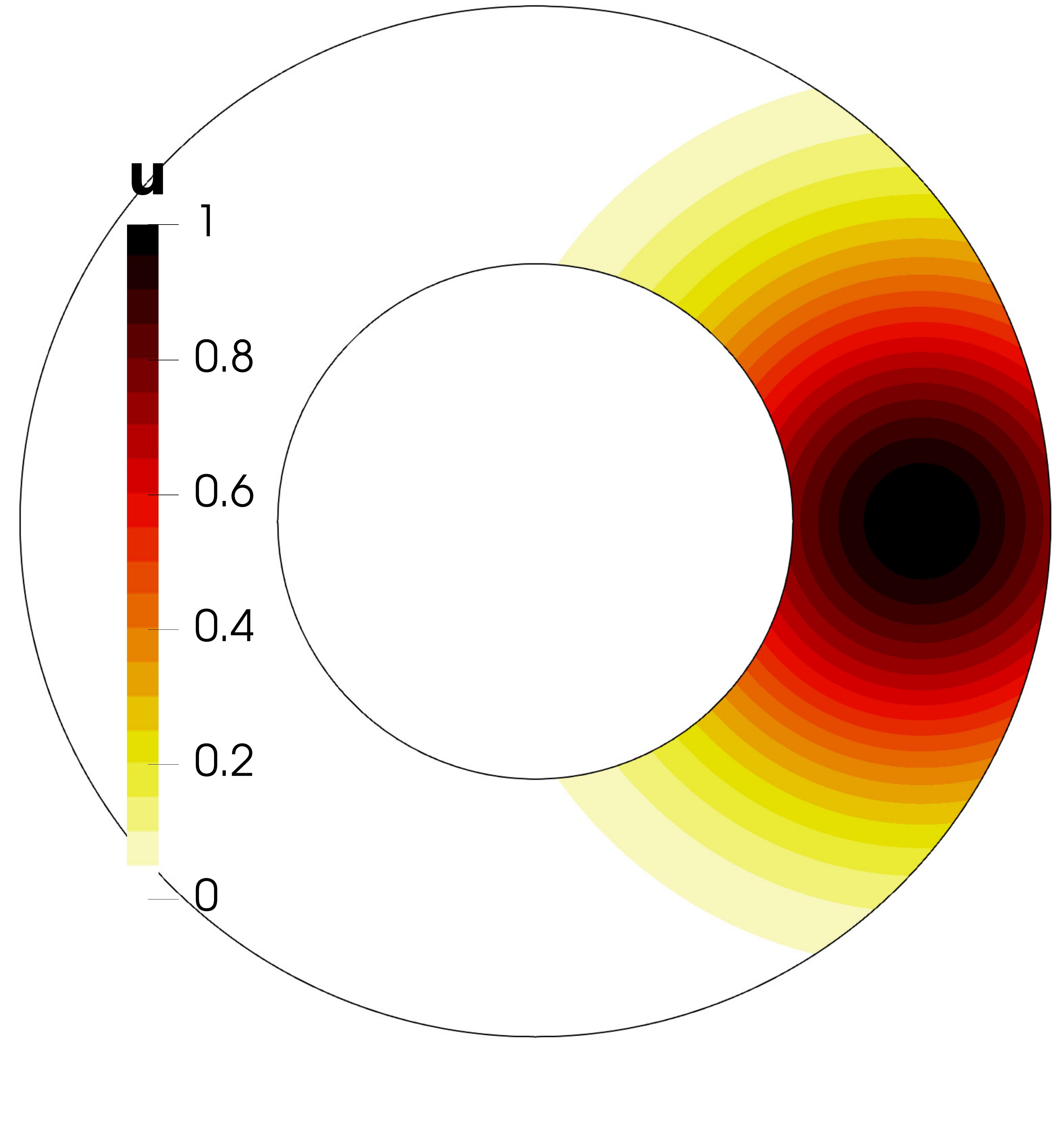}
    \caption{\small Initial condition.} \label{fig:bump_ic}
  \end{subfigure}%
  \hfill%
  \begin{subfigure}[t]{0.3\textwidth}
    \centering
    \includegraphics[width=\textwidth]{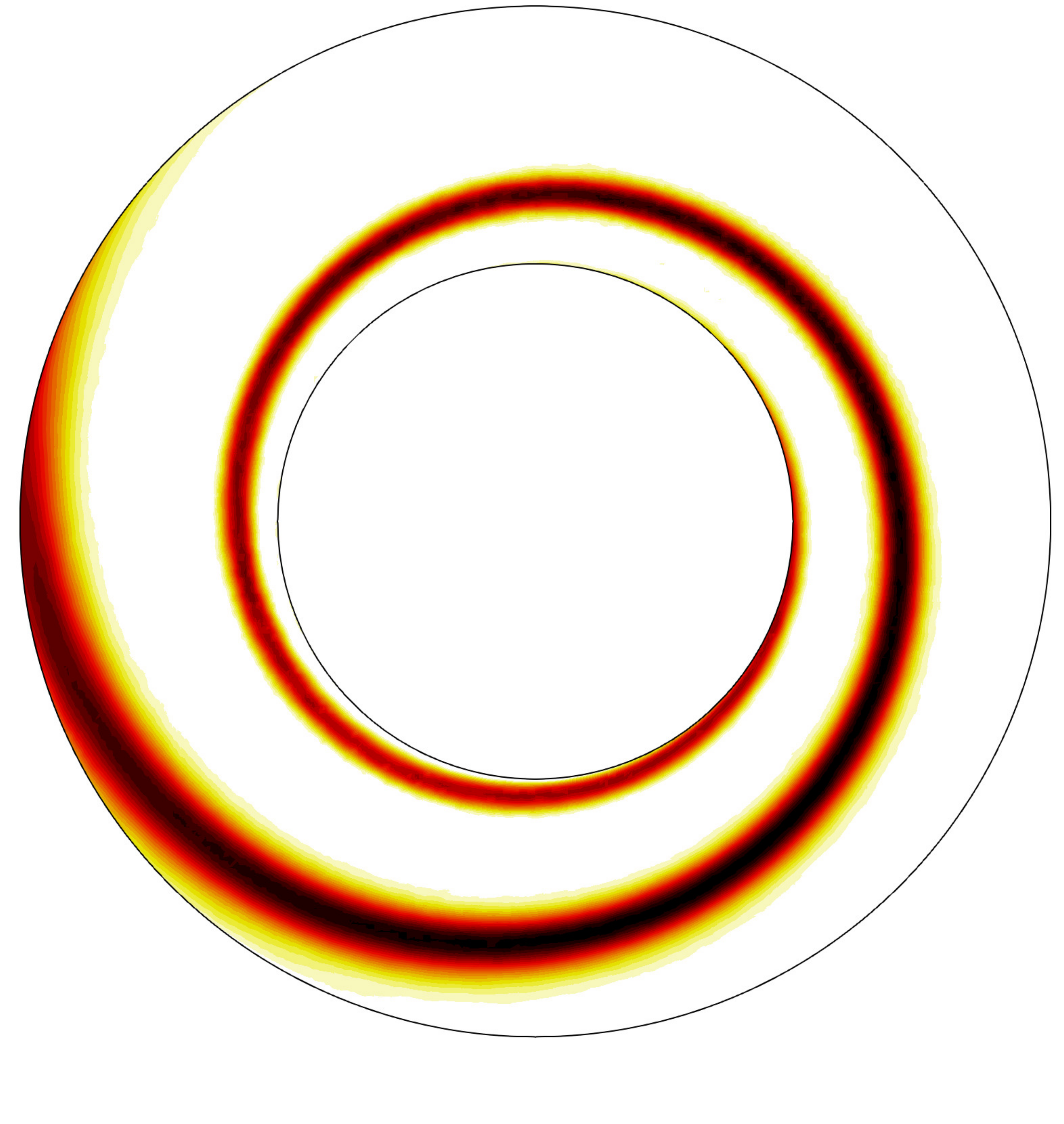}
    \caption{\small Final solution, $p=4$.} \label{fig:bump_final_p4}
  \end{subfigure}%
  \hfill%
  \begin{subfigure}[t]{0.39\textwidth}
    \centering
    \includegraphics[width=\textwidth]{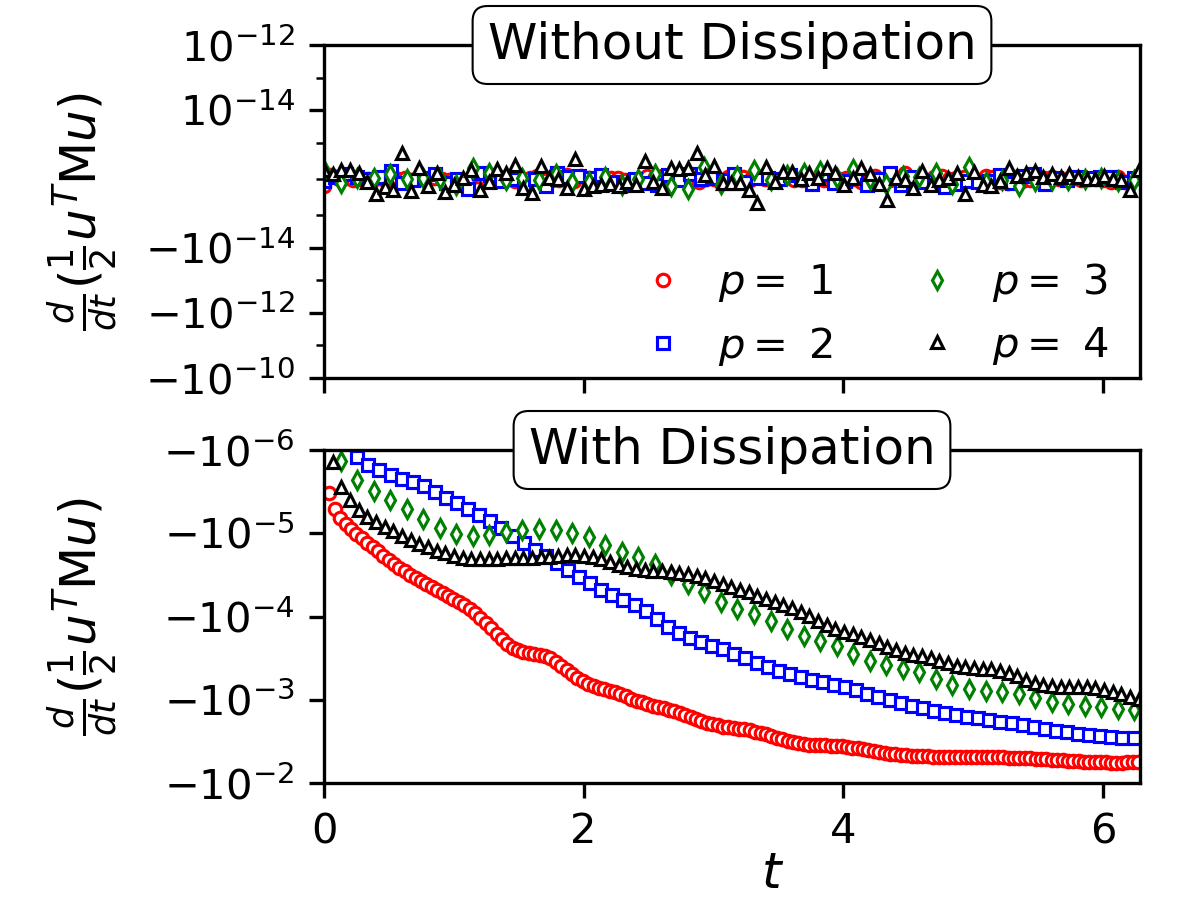}
    \caption{\small Energy rate} \label{fig:energy}
  \end{subfigure}%
  \caption{\small Initial condition (left), final discrete solution ($p=4$, center), and rate-of-change of energy (right) for the annulus problem.} \label{fig:bump}
\end{figure}

Figure~\ref{fig:energy} shows the rate-of-change of the energy versus time for the discretizations without dissipation and the discretizations with dissipation.  The rates are computed by evaluating $\bm{u}^T \bm{r}$ at the beginning of each time step.  Without dissipation the rate-of-change of energy is close to machine precision, as expected.  When dissipation is included, the rate-of-change is strictly negative; furthermore, as $p$ increases the magnitude of the rate-of-change generally decreases.

We conclude this section with Figure~\ref{fig:nodiss_vs_diss_errors}, which plots the absolute error distributions at the final time with and without dissipation.  To conserve space, the figure shows the errors for $p=1$ and $p=4$ only; the results for $p=2$ and $p=3$ are qualitatively similar, but have error magnitudes between the $p=1$ and $p=4$ errors.  The exact solution at node $\bm{x}_i$ and time $t=2\pi$ is approximated by numerically solving the ordinary differential equation $d\bm{x}/dt = \bm{\lambda}(\bm{x})$ with the terminal condition $\bm{x}(t=2\pi) = \bm{x}_i$, backward in time from $t=2\pi$ to $t=0$ using Verner's 9/8 Runge-Kutta method~\cite{Verner2010numerically} as implemented in the \texttt{DifferentialEquations.jl} Julia package~\cite{Rackauckas2017differentialequations}, and then substituting the solution into $\mathcal{U}_0(\bm{x})$.

Figure~\ref{fig:nodiss_vs_diss_errors} demonstrates the value of including some numerical dissipation with the point-cloud SBP discretizations; the solutions with dissipation contain fewer high-frequency oscillations compared to the solutions without dissipation.  Furthermore, the $p=4$ errors are smaller in magnitude relative to the $p=1$ errors, as expected.

\begin{figure}[tp]
  \begin{subfigure}[t]{0.49\textwidth}
    \centering
    \includegraphics[width=1.0\textwidth]{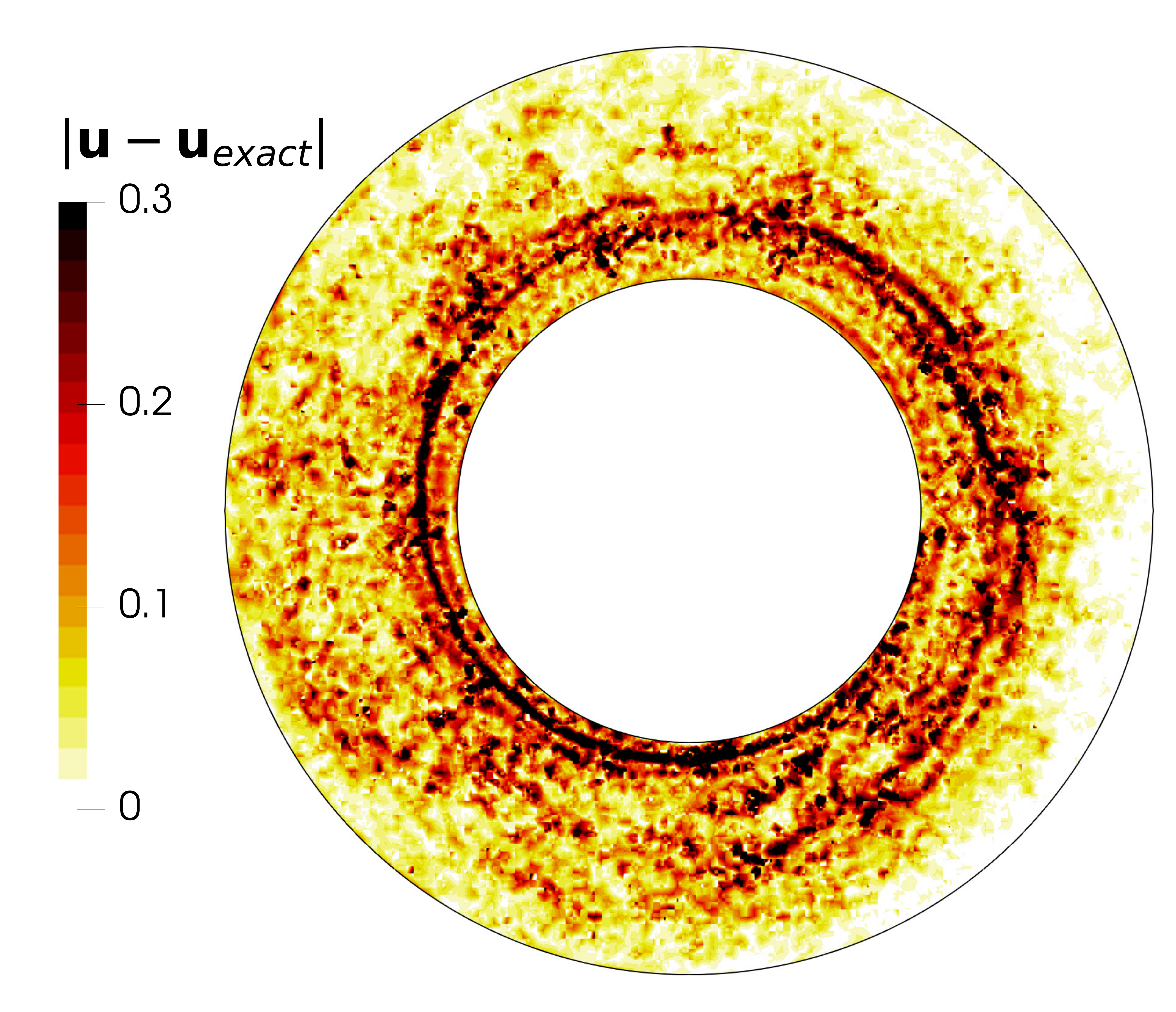}
    \caption{\small $p=1$, no dissipation} \label{fig:error_nodiss_p1}
  \end{subfigure}%
  \hfill%
  \begin{subfigure}[t]{0.49\textwidth}
    \centering
    \includegraphics[width=1.0\textwidth]{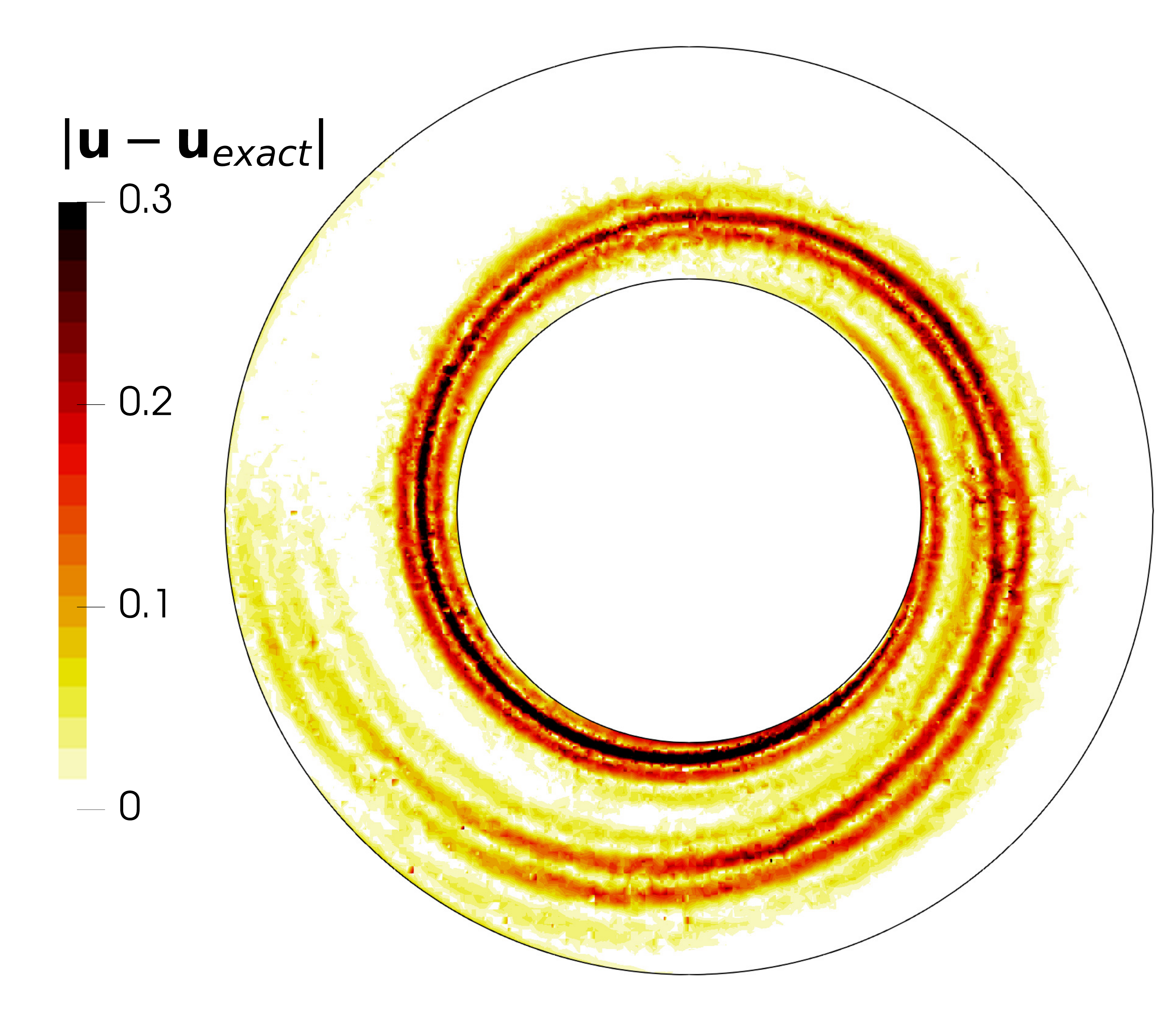}
    \caption{\small $p=1$, with dissipation} \label{fig:error_p1}
  \end{subfigure}\\%
  \begin{subfigure}[t]{0.49\textwidth}
    \centering
    \includegraphics[width=1.0\textwidth]{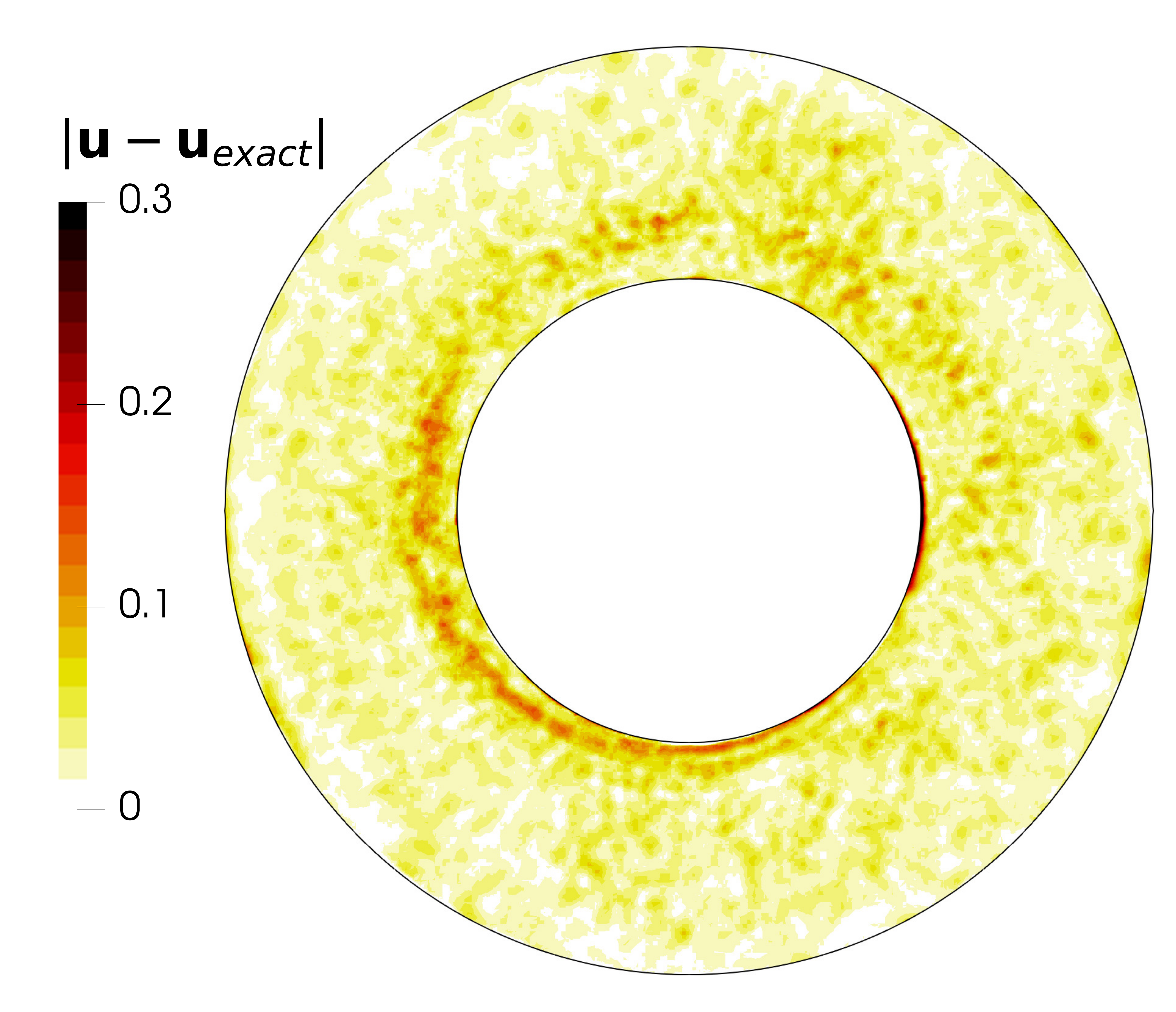}
    \caption{\small $p=4$, no dissipation} \label{fig:error_nodiss_p4}
  \end{subfigure}%
  \hfill%
  \begin{subfigure}[t]{0.49\textwidth}
    \centering
    \includegraphics[width=1.0\textwidth]{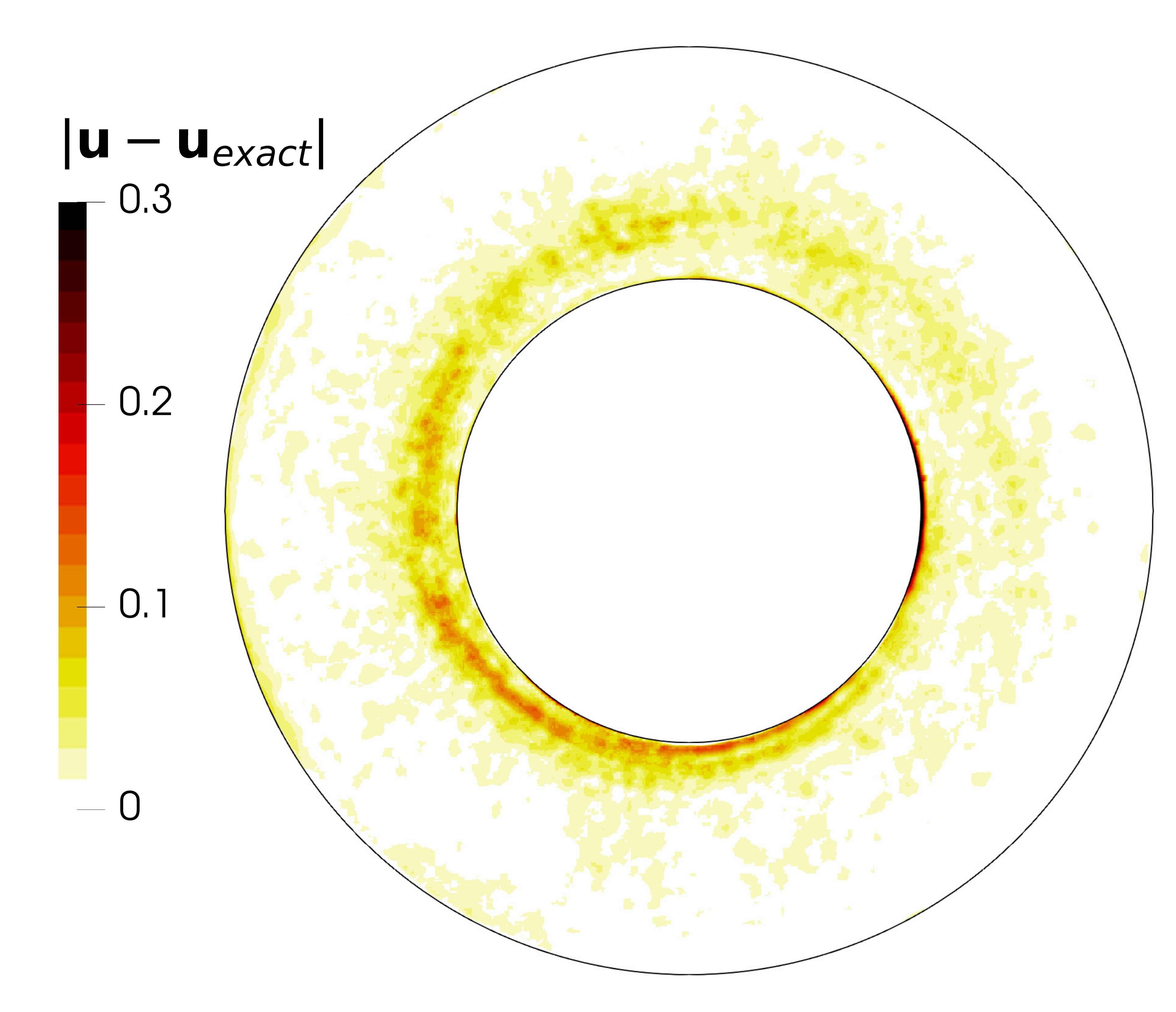}
    \caption{\small $p=4$, with dissipation} \label{fig:error_p4}
  \end{subfigure}%
  \caption{\small Absolute error distributions at the final simulation time $t=2\pi$ for $p=1$ (top) and $p=4$ (bottom) with no dissipation (left) and dissipation added (right). } \label{fig:nodiss_vs_diss_errors}
\end{figure}

\subsection{Component runtimes and scaling}\label{sec:runtimes}

The paper's final study investigates the relative computational cost and scaling of the components used to construct the point-cloud SBP operators.  While our implementation has not been optimized, we hope the timings presented below provide some guidance regarding potential computational bottlenecks that will need to be addressed in future work.

We gather component runtimes using the airfoil geometry, $\Omega_{\text{foil}}$, and the linear advection problem described in Section~\ref{sec:sbp_accuracy_verify}.  We consider the same distributions as the accuracy study, with point resolutions defined by $n_y \in \{4, 8, 16, 32, 64\}$ and $n_x = 5 n_y$.  For each discretization, we time the following components in the construction process:
\begin{description}
\item[mesh:] Background mesh generation (Section~\ref{sec:quad_mesh}).
\item[stencil:] Stencil construction for each cell (Section~\ref{sec:stencil}).
\item[$\mat{M}$ opt.:] Mass-matrix optimization (Section~\ref{sec:norm_solve}).
\item[$\mat{S}_x$ \& $\mat{E}_x$:] Construction of $\mat{S}_x$ and $\mat{E}_x$ (Sections~\ref{sec:sym_part}, \ref{sec:skew_part}, and \ref{sec:assemble}).
\item[$\mat{A}$ (diss.):] Construction of the dissipation operator (Section~\ref{sec:dissipation}).
\item[system:] Forming the discretization (Equation~\eqref{eq:advection_sbp}).
\end{description}
While the list above is not exhaustive, the runtimes of these components were found to be representative of the overall costs.

Figure~\ref{fig:timing} shows the average component runtimes versus point resolution $n_x$, where the average times are based on ten distribution samples\footnote{More precisely, we gather 11 samples and discard the first sample to avoid including Julia's compile times.}.  The linear inequality fails for $p=4$ on the coarsest distribution ($n_x=20$) for all samples, so the average times for components that depend on the mass matrix are not shown for degree 4 and $n_x = 20$.  Note that a log scale is used for the axes, and a light gray reference is included in each subplot to show linear cost.

We highlight three observations based on Figure~\ref{fig:timing} and the component times more generally.
\begin{itemize}
\item  Using data from the two finest grids, we find that the components scale between $t \propto n_x^{1.2}$, for the background mesh construction, and $t \propto n_x^{1.5}$ for the $\mat{M}$ optimization problem.  Thus, the times scale super linearly with $n_x$ and not linearly with the number of nodes ($N \propto n_{x}^2$), which is somewhat surprising.  However, inspecting subcomponent timings reveals that the cost is driven by the number of cut cells, which scales with $n_x$ and not $N$.  This is true even for the mass-matrix optimization, which is dominated by the time required to set up the linear inequality \eqref{eq:linear_program} and not the time to solve it.
\item Constructing $\mat{S}_x$ and $\mat{E}_x$ is the most expensive component at present.  This can be attributed to the cut-cell related costs, particularly computing quadratures for the cut cells and their faces.
\item The system build --- shown as brown diamonds in Figure~\ref{fig:timing} --- is the fastest component for $p=1$ but the second slowest for $p=4$.  This reflects the decrease in sparsity in the system matrix as $p$ increases, which then requires more operations when building the matrix.
\end{itemize}

\begin{remark}
While we have performed limited three-dimensional simulations, our preliminary results suggest that the cut cells continue to dominate the component run times.  Consequently, we expect the construction algorithm to scale at approximately $N^{2/3}$ in three dimensions.
\end{remark}

\begin{figure}[tbp]
  \begin{center}
    \includegraphics[width=\textwidth]{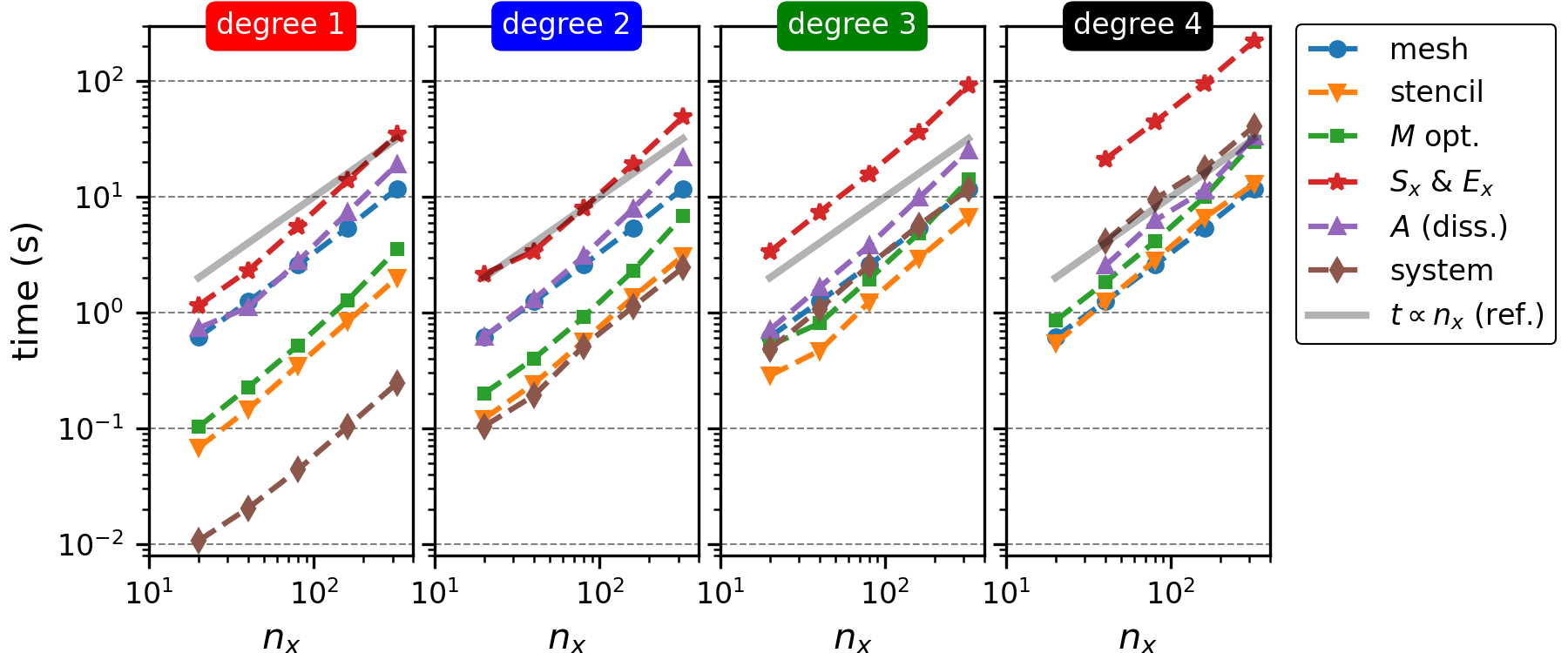}
    \caption{\small CPU time, in seconds, for six representative components of the point-cloud SBP construction algorithm.} \label{fig:timing}
  \end{center}
\end{figure}

\section{Summary and future work}\label{sec:conclude}

The paper described an algorithm to construct diagonal-norm SBP operators on complex geometries with the nodes defined on a point cloud.  The algorithm relies on a background Cartesian cut-cell mesh, so it is not mesh-free, but the generation of this mesh is easily automated.  Using the mesh, the algorithm constructs cell-based \emph{degenerate} SBP operators whose norm is not necessarily positive definite.  To compute the cell-based skew-symmetric matrices, \eg $\mat{S}_x^{\cell}$, we introduced a formula that avoids solving a large sparse linear system.  The cell-based operators are assembled into global (degenerate) SBP operators.

The diagonal mass matrix may not be positive-definite, in general.  Therefore, we introduced additional degrees of freedom into the cell stencils to develop a linear inequality that enforces a positive-definite mass matrix.  We presented a theorem that elucidates when the linear inequality has a solution; for equidistributed nodes, the theorem suggests that solutions exist for a sufficiently large set of nodes, which is corroborated by the results.  In this work, the linear inequality was reformulated as a linear optimization problem and solved with an interior-point algorithm.

We conducted several numerical studies to verify the theoretical claims in the paper. Studies of the diagonal mass matrix demonstrated that the entries (\ie quadrature weights) are well-behaved after solving the linear inequality, and that a solution exists for all but the coarsest node distributions.  We confirmed that the quadrature induced by the diagonal norm is order $2p$ for smooth functions.  We used the SBP operators to discretize the linear advection equation and confirmed the accuracy and stability of the schemes.  Our results also show that the point-cloud SBP operators are not as efficient as classical (tensor-product) finite-difference SBP operators.  We investigated how the primary components of the construction algorithm scale and found that the cut-cells dictate the computational cost of most components.

There are numerous aspects of the point-cloud SBP operators that remain to be investigated, but we limit our discussion to just a few.  Theoretically, we would like to see a more precise characterization of when the null space of $\mat{Z}^T$ is equal to the Vandermonde matrix $\mat{V}_{2p-1,\sdim}$.  This could have implications for constructing node distributions and stencils that ensure the norm inequality is always feasible.  Practically, we would like to use the operators on large-scale fluid-flow simulations, which will require some form of parallel implementation.  To reduce computational cost, we are interested in exploiting the flexibility of the point-cloud operators for solution adaptation and shape optimization.

Finally, given that the present work builds on the generalized DGD method~\cite{Yan2022thesis}, future work could combine point-cloud SBP operators with radial-basis function (RBF) discretizations.  For instance, RBFs could be incorporated into the accuracy conditions of the point-cloud SBP method.  This has the potential to extend the energy-stable, global RBF-SBP method~\cite{Glaubitz2024energy} by enabling a diagonal mass matrix and a sparse system matrix, \ie an energy-stable, local RBF-FD method.

\section*{Acknowledgements}

The first author was supported by the Office of Naval Research [grant number N00014-23-1-2698] with Dr. Mark Spector as the project manager.  This work was also partially supported by the National Science Foundation [grant numbers 1825991 and 1554253].  The authors gratefully acknowledge this support.

\appendix

\section{Theorem 4 Details}\label{app:detail}

Here we provide additional steps showing how $\mat{Q}_x \mat{V}$ can be written as the five sums in Equation~\eqref{eq:five_sums}.  First, we substitute the symmetric-skew-symmetric decomposition of $\mat{Q}_x$:
\begin{equation*}
  \mat{Q}_x \mat{V} = \Big( \mat{S}_x + \frac{1}{2} \mat{E}_x \Big) \mat{V}
\end{equation*}
Next, we substitute the global $\mat{S}_x$ and $\mat{E}_x$ definitions given by Equations~\eqref{eq:global_S} and \eqref{eq:global_E}, respectively.
\begin{multline}\label{eq:app_detail}
  \mat{Q}_x \mat{V} = \sum_{\cell=1}^{N_{C}} \big( \mat{P}^{\cell} \big)^T \; \mat{S}_x^{\cell} \; \mat{P}^{\cell} \; \mat{V} + \frac{1}{2} \sum_{\face \in F_I} \big( \mat{R}^{\face,-} \big)^{T} \; \mat{B}^{\face} \; \mat{N}_{x}^{\face} \; \mat{R}^{\face,+} \; \mat{V} \\
  - \frac{1}{2} \sum_{\face \in F_I} \big( \mat{R}^{\face,+} \big)^{T} \; \mat{B}^{\face} \; \mat{N}_{x}^{\face} \; \mat{R}^{\face,-} \; \mat{V} 
  + \frac{1}{2} \sum_{\face \in F_B} \big( \mat{R}^{\face} \big)^T \; \mat{B}^{\face} \; \mat{N}_{x}^{\face} \; \mat{R}^{\face} \; \mat{V}
\end{multline}
In the first sum above, the skew-symmetric operator for cell $\cell$ can be expressed as $\mat{S}_x^{\cell} = \mat{Q}_x^{\cell} - \frac{1}{2} \mat{E}_x^{\cell}$.  Using this identity, and the interpolated Vandermonde matrices defined in the Proof of Theorem~\ref{thm:skew_QR}, Equation~\eqref{eq:app_detail} becomes
\begin{multline*}
  \mat{Q}_x \mat{V} = \sum_{\cell=1}^{N_{C}} \big( \mat{P}^{\cell} \big)^T \; \mat{Q}_x^{\cell} \; \mat{V}^{\cell} - \frac{1}{2} \sum_{\cell=1}^{N_{C}} \big( \mat{P}^{\cell} \big)^T \; \mat{E}_x^{\cell} \; \mat{V}^{\cell} + \frac{1}{2} \sum_{\face \in F_I} \big( \mat{R}^{\face,-} \big)^{T} \; \mat{B}^{\face} \; \mat{N}_{x}^{\face} \; \mat{V}^{\face} \\
  - \frac{1}{2} \sum_{\face \in F_I} \big( \mat{R}^{\face,+} \big)^{T} \; \mat{B}^{\face} \; \mat{N}_{x}^{\face} \; \mat{V}^{\face}
  + \frac{1}{2} \sum_{\face \in F_B} \big( \mat{R}^{\face} \big)^T \; \mat{B}^{\face} \; \mat{N}_{x}^{\face} \; \mat{V}^{\face}.
\end{multline*}
The five sums on the right-hand side above are equal to the five sums on the right-hand side of \eqref{eq:five_sums}.

\bibliographystyle{elsarticle-num}
\bibliography{refs}

\end{document}